\documentclass[reqno]{amsart}
\usepackage[utf8x]{inputenc}
\usepackage{amssymb,amsmath,amsthm,amsfonts,mathtools}
\usepackage[normalem]{ulem}
\usepackage{hyperref}
\hypersetup{colorlinks=true,allcolors=blue}
\usepackage{cleveref}
\usepackage{enumitem}
\usepackage[dvipsnames]{xcolor}
\usepackage{microtype}
\usepackage{tikz,faktor}
\usepackage{float,array}
\usepackage{subcaption}
\usetikzlibrary{angles,quotes,calc}

\usepackage[margin = 1.2in]{geometry}
\usepackage{graphicx}

\newtheorem*{ack}{Acknowledgements}

\newenvironment{claim}[1]{\par\noindent\underline{Claim}:\space#1}{}
\newenvironment{claimproof}[1]{\par\noindent\underline{Proof}:\space#1}{\hfill $\blacksquare$}

\newcommand{\R}{\mathbb R}
\newcommand{\N}{\mathbb N}
\newcommand{\relmiddle}[1]{\mathrel{}\middle#1\mathrel{}}
\newcommand{\newtau}{\scalebox{1.44}{$\tau$}}

\newcommand{\LpLS}{Lorentzian pre-length space }
\newcommand{\LLS}{Lorentzian length space }

\DeclareMathOperator{\CBA}{CBA}

\DeclareMathOperator{\CAT}{CAT}
\newcommand{\ma}{\measuredangle}

\usepackage{pifont}
\newcommand{\xmark}{\ding{55}\,}
\newcommand{\bd}{\partial}

\makeatletter
\def\l@subsection{\@tocline{2}{0pt}{2.5pc}{5pc}{}}
\makeatother

\title{Timelike ideal boundary of non-positively curved Lorentzian spaces}

\author[Burgos]{Sa\'ul Burgos}
\address[Burgos]{Departamento de Geometr\'ia y Topolog\'ia \& IMAG, Universidad de Granada, Spain}
\email{sburgos@ugr.es}

\author[Che]{Mauricio Che}
\address[Che]{Faculty of Mathematics, University of Vienna, Austria}
\email{mauricio.adrian.che.moguel@univie.ac.at}

\author[Prados--Abad]{Miguel Prados--Abad}
\address[Prados--Abad]{Faculty of Mathematics, University of Vienna, Austria}
\email{miguel.prados.abad@univie.ac.at}

\date{\today}

\usepackage{thmtools}

\declaretheorem[
	name=Theorem,
	numberwithin=section
	]{thm}
\declaretheorem[
	name=Lemma,
	numberlike=thm,
	]{lem}
\declaretheorem[
	name=Lemma,
	numbered=no,
	]{lem*}
\declaretheorem[
	name=Proposition,
	sibling=thm,
	]{pop}
\declaretheorem[
	name=Corollary,
	sibling=thm,
	]{cor}

\declaretheorem[
	name=Definition,
	style=definition,
        numberlike=thm,
	]{dfn}

\declaretheorem[
	name=Remark,
	numberlike=thm,
        style=definition,
	]{rem}

\begin{document}

\begin{abstract}
We introduce the notion of timelike ideal boundary of a Lorentzian length space as the set of asymptotic classes of future or past-directed timelike geodesic rays, a construction complementary to the causal boundary in the sense of Geroch–Kronheimer–Penrose and akin to the concept of ideal boundary of a metric space. We endow this timelike ideal boundary with a natural cone topology and an angular metric, and establish upper curvature bounds for the resulting metric space. Finally, we consider generalized cones as a model and study the relation between the timelike ideal boundary and both the metric ideal boundary of the fiber and the asymptotic behavior of the warping function.
\medskip

\noindent
\emph{Keywords:} Lorentzian length spaces, metric geometry, ideal boundary, Lorentzian geometry, synthetic curvature bounds
\medskip

\noindent
\emph{MSC2020:} 
    51K10, 
    53B30, 
    53C23, 
    53C50. 
\end{abstract}

\maketitle

\setcounter{tocdepth}{1}
\tableofcontents

\section{Introduction}

The notion of \textit{boundary} of a spacetime is both compelling and elusive, and has given rise to a variety of rigorous mathematical formulations over the years. Among these, the {\em conformal boundary}, introduced by Penrose \cite{penrose_asymptotic_1963}, is perhaps the most widely used. It provides the framework in which notions such as null infinity and black holes are defined. Despite its importance and utility, the conformal boundary has certain limitations: a spacetime may fail to admit a conformal compactification, and even when such a compactification exists, there is generally no canonical way to construct it. A classical alternative that addresses some of these issues is the work of Geroch,~Kronheimer, and Penrose \cite{geroch_ideal_1972}, who defined the \textit{causal boundary} of a spacetime (see also \cite{flores_final_2011} for a systematic and modern treatment solving the subtleties of the original construction). In this framework, boundary points are defined in terms of the timelike pasts (respectively, futures) of future-directed (respectively, past-directed) inextendible timelike curves. These sets are referred to as \textit{terminal indecomposable pasts} (TIPs) and \textit{terminal indecomposable futures} (TIFs). In this way, every inextendible timelike curve acquires an endpoint in the resulting completion, which justifies the terminology. Alternative approaches to spacetime boundaries have also been proposed: among others, Geroch's \textit{$g$-completion} \cite{geroch_local_1968}, Schmidt's \textit{bundle completion} \cite{schmidt_new_1971}, Scott--Szekeres' \textit{abstract completion} \cite{scott_abstract_1994}, the \textit{isocausal extensions} by García-Parrado--Senovilla \cite{garcia-parrado_causal_2003} or, more recently, the \textit{directed completion} of a partially ordered set \cite{zhao_dcpo-completion_2010, gigli_hyperbolic_2025} and the \textit{forward completion} of a Lorentzian pre-length space \cite{mondino-saemann-2025}.

On the other hand, in recent years, a synthetic approach to the study of spacetimes and General Relativity has developed rapidly, particularly through the framework of \textit{Lorentzian pre-length spaces}. Introduced by Kunzinger--Sämann in \cite{kunzinger-saemann2018}, this setting abstracts key structural features of spacetimes, including the causal and timelike relations as well as the time separation function. It thereby provides a flexible framework in which to investigate problems in General Relativity that require low regularity assumptions. Within this context, several classical constructions and notions from Lorentzian geometry have been extended and systematically studied. An important line of research within this theory, motivated by the classical energy conditions in General Relativity, concerns the formulation of synthetic lower curvature bounds. A notion of lower timelike sectional curvature bounds via triangle comparison was already introduced in \cite{kunzinger-saemann2018}. The first synthetic notions of lower timelike Ricci curvature bounds for measured Lorentzian pre-length spaces followed a few years later \cite{cavalletti-mondino} after similar constructions in the smooth setting \cite{mccann_displacement_2020,mondino_suhr_ot_2022}. Complementary to this, the study of Lorentzian pre-length spaces with upper timelike curvature bounds, also introduced in \cite{kunzinger-saemann2018}, has largely been driven, not by General Relativity, but by analogies with the theory of $\CAT(K)$ spaces in metric geometry. Recent developments in this vein include Alexandrov's Patchwork \cite{beran-napper-rott2025}, a Lorentzian Cartan--Hadamard Theorem \cite{eros-gieger-2025}, and a Lorentzian analogue of the Splitting Theorem for $\CAT(0)$ spaces \cite{barton-beran-che-gieger-roehrig-rott}. 

In the context of $\CAT(0)$ spaces, there exists a notion of \textit{ideal boundary}, which was introduced by Eberlein--O'Neill in \cite{eberlein_visibility_1973} (see also \cite{BGS,BH}). This notion consists of a set of equivalence classes of asymptotic  geodesic rays and can be intuitively understood as the set of directions one can visualize at infinity, which explains why this concept has also been referred to as \textit{visual boundary} in the literature. Given a $\CAT(0)$ space $X$, its ideal boundary $\bd X$ can be endowed with a natural topology (called the \textit{cone topology}) and a natural metric (called the \textit{angular distance}) which make $\bd X$ into a complete and $\pi$-geodesic space (i.e., there exist minimizing geodesics in $\bd X$ connecting ideal points at angular distance $\leq \pi$) which satisfies the $\CAT(1)$ condition.

In this manuscript, we introduce the notion of \textit{timelike ideal boundary} for Lorentzian pre-length spaces in the sense of Kunzinger--S\"amann, inspired by the definition of the ideal boundary of metric spaces due to Eberlein--O'Neill. Intuitively speaking, the timelike ideal boundary admits the following natural physical interpretation: its points correspond to equivalence classes of ``asymptotic'' freely falling observers. More precisely, each future-directed timelike geodesic ray represents the worldline of an ideal free falling observer, and two such observers define the same ideal boundary point if, up to a finite shift on their proper time, they remain timelike related. Analogously for past-directed timelike geodesic rays. 

While the idea of asymptotic timelike geodesic rays already appears in \cite{barton-beran-che-gieger-roehrig-rott} (see \autoref{def:asymptotic}), the approach developed here provides a natural and systematic framework for organizing these objects into a notion of timelike ideal boundary. Furthermore, it is worth noting that, to the best of the authors' knowledge, this construction has not been systematically developed even in the smooth spacetime setting.

\begin{dfn}\label{def:timelike ideal boundary}
Given a Lorentzian pre-length space $Y$, the \textit{future timelike ideal boundary} of $Y$ is the set $\bd^+Y$ of equivalence classes of future-directed timelike geodesic rays under the relation of being asymptotic. The \textit{future timelike ideal completion} of $Y$ is the set $Y^*=Y\cup \bd^+Y$. We call elements of $\bd^+Y$ \textit{future timelike boundary points}. By reversing causal and timelike relations we obtain the notion of \textit{past} timelike ideal boundary, denoted by $\bd^-Y$, and so on.
\end{dfn}

Observe that, in contrast with the notion of causal boundary, the present construction is formulated using timelike geodesic rays instead of more general timelike curves. Rather than employing the associated timelike pasts or futures, we consider an equivalence relation on rays themselves, identifying those that are ``asymptotically aligned'' in the sense of timelike relation up to a finite time shift. The two constructions also differ in the ``granularity'' of the information they encode. On the one hand, the causal boundary includes terminal indecomposable pasts and futures which cannot arise as the timelike pasts or futures of timelike geodesic rays. Intuitively speaking, these are pasts or futures of timelike curves that are asymptotically null. On the other hand, the timelike ideal boundary is typically finer: distinct equivalence classes of timelike geodesic rays may give rise to the same terminal indecomposable past or future, so that a single point of the causal boundary may correspond to multiple points of the ideal boundary (see \autoref{cor:components of bd are inside TIFs}).

Similarly to the metric setting, we mainly work with Lorentzian pre-length spaces with timelike curvature bounded above by zero globally in the sense of \cite{kunzinger-saemann2018}. We prove that, given such a Lorentzian pre-length space $Y$, the future timelike ideal boundary $\bd^+Y$ can be endowed with a natural \textit{cone topology} (\autoref{pop:cone_top}). Moreover, we endow $\bd^+Y$ with a natural \textit{angular metric} $\ma$ (\autoref{pop:main1}), and obtain the following result as a consequence of \autoref{pop:completeness}, \autoref{cor:geodesicity}, and \autoref{thm:main2}. 

\begin{thm}
Let $Y$ be a proper, globally hyperbolic, strongly causal, locally causally closed and regular \LpLS satisfying $\CBA(0)$ globally. Then $(\bd^+ Y,\ma)$ is a complete, geodesic $\CAT(-1)$ space.
\end{thm}

In the final sections, as examples of our construction, we show that the future timelike ideal boundary of a product space of the form $\prescript{-}{}I\times X$, where $I=(a,\infty)$ for some $-\infty\leq a <\infty$ and $X$ is a $\CAT(0)$ space, is a warped product of the form $[0,\infty)\times_{\sinh} \bd X$, where $\bd X$ is the ideal boundary of $X$ in the metric sense (\autoref{thm:minkowski product}), as one would expect from the positive-signature case. We also obtain a partial converse of such a result: if both the future and the past timelike ideal boundaries of a \LpLS with timelike curvature globally bounded above by zero agree with those of a generalized product and there is a natural compatibility relation between the future and past apexes, then the space itself must be a generalized cone (\autoref{thm:rigidity}).

Finally, we consider generalized cones of the form $Y=\prescript{-}{}{\mathbb{R}}\times_f X$, in the sense of \cite{alexander_generalized_2023}, and under suitable conditions on the asymptotic behavior of the corresponding warping function $f$, we describe the underlying set of the timelike ideal boundary $\bd^+Y$ and, in some cases, we also describe the metric. Concretely, we obtain the following statement, with the notation from \autoref{def:nomenclature}:

\begin{thm}
Let $(X,d)$ be a complete geodesic space and let $f\colon \R\to (0,\infty)$ be convex. The following implications about the future timelike ideal boundary $\bd^+Y$ of the generalized product $Y=\prescript{-}{}{\mathbb{R}}\times_f X$ hold:
\begin{enumerate}
    \item If $\lim_{t\to\infty}f(t) = 0$ quickly, it consists of just a point.
    \item If $\lim_{t\to\infty}f(t) = 0$ slowly and $X$ is locally compact and $\CAT(-\varepsilon)$ for some $\varepsilon>0$, then $\bd^+Y$ is isometric to the quotient of the set $[0,\infty)\times \bd X$ under the equivalence relation $(0,\xi)\sim(0,\xi')$, endowed with the discrete distance with value infinity.
    \item If $\lim_{t\to\infty}f(t) = L>0$ and $X$ is locally compact and $\CAT(-\varepsilon)$ for some $\varepsilon>0$, then $\bd^+Y$ is isometric to the metric warped product $[0,\infty)\times_{\sinh} \bd X$.
    \item If $\lim_{t\to\infty}f(t) = \infty$ slowly, then $\bd^+Y$ is in bijection with the set $X$. 
    \item If $\lim_{t\to\infty}f(t) = \infty$ quickly, then $\bd^+Y$ is isometric to the set $X$ endowed with the discrete infinite distance. 
\end{enumerate}
\end{thm}

\begin{ack}
The authors would like to express their gratitude with Tobias Beran, Jose Luis Flores Dorado, Stacey Harris, J\'onatan Herrera, and Roland Steinbauer  for very fruitful conversations at an early stage of this project. MC and MPA were funded by the Austrian Science Fund (FWF) [Grant DOI’s: \href{https://www.fwf.ac.at/forschungsradar/10.55776/EFP6}{10.55776/EFP6}, \href{https://www.fwf.ac.at/forschungsradar/10.55776/STA32}{10.55776/STA32}]. 
SB acknowledges the support by the IMAG-María de Maeztu grant CEX2020-001105-MCIN/AEI/10.13039/501100011033 as well as the project PID2024-156031NB-I00 funded by MICIU/AEI/10.13039/{501100011033/ERDF/EU.}
For open access purposes, the authors have applied a CC BY public copyright license to any author-accepted manuscript version arising from this submission.
\end{ack}

\section{Preliminaries}

\subsection{The metric ideal boundary}\label{subs:metric-ideal-boundary}

In this subsection we recall the notion of the ideal boundary of a $\CAT(0)$ space, following \cite[Ch.~II.8-9]{BH}. This construction serves as inspiration for the definition of the Lorentzian ideal boundary and some results from the metric setting here mentioned will be adapted to the Lorentzian framework in the following sections.

\begin{dfn}\label{def:metric-geodesic-rays-asymptoticity}
    A \textit{geodesic ray} in a metric space $(X,d)$ is an isometric embedding $\rho \colon [0,\infty) \to X$, i.e., $d(\rho(s),\rho(t))=|t-s|$, for all $s,t\in[0,\infty)$. Two geodesic rays $\rho_1,\rho_2$ are called \textit{asymptotic} if their images remain at bounded distance from each other, i.e., if there exists some constant $K\in \R^+$ such that
    \[
    d\bigl(\rho_1(t),\rho_2(t)\bigr) < K, \qquad \forall t\in[0,\infty).
    \]
\end{dfn}

\begin{dfn}\label{def:metric-ideal-boundary}
    The \textit{ideal boundary} $\bd X$ of a metric space $(X,d)$ is defined as the set of equivalence classes of asymptotic geodesic rays in $X$. The \textit{ideal completion} is the disjoint union $\overline{X}=X\sqcup \bd X$. We will call the points in $\bd X$ \textit{boundary points}. Given a geodesic ray $\rho$, its equivalence class will be denoted by $\rho(\infty)$.
\end{dfn}

Usually, one works with metric spaces with infinite diameter; otherwise there are no geodesic rays and therefore the ideal boundary is empty. On the other hand, the ideal boundary might be empty even when the diameter of $X$ is infinite \cite[Theorem~3.4]{buckley_natural_2013}. The ideal boundary is usually considered in the framework of complete $\CAT(0)$ spaces. One can show \cite[Proposition~II.8.2]{BH} that in such a setting, for every $\xi\in\bd X$ and every point $x\in X$ there is a unique representative $\rho\in\xi$ (i.e., $\rho(\infty)=\xi$) starting at $x$. We usually denote such a representative by $\xi_x$. 

The ideal completion $\overline{X}$ of a complete $\CAT(0)$ space $X$ can be endowed with a natural topology, called the \textit{cone topology} satisfying a number of interesting properties; e.g., its restriction to $X$ is the metric topology, $X$ is open and dense in $\overline{X}$, and the extension to $\overline{X}$ of every isometry of $X$ is a homeomorphism \cite[Corollary~II.8.9]{BH}.

Moreover, a natural distance can be defined in $\bd X$ making use of the angle between geodesic rays at their starting point. Specifically, the \textit{angular distance} $\ma$ in $\bd X$ is defined as
\begin{equation}\label{eq:def-metric-angular-distance}
\ma(\xi,\eta)=\sup_{x\in X} \ma_x(\xi_x,\eta_x), 
\end{equation}
where $\xi_x,\eta_x$ are the unique representatives of $\xi,\eta$, respectively, starting at $x$, and $\ma_x$ is the upper angle between curves in a metric space, defined by means of comparison angles in $\R^2$ \cite[Ch.~I.1]{BH}, i.e., 
\begin{equation}\label{eq:def-metric-angle-between-curves}
\ma_x(\xi_x,\eta_x)=\limsup_{s,t\searrow0} \widetilde{\ma}_x\bigl(\xi_x(s),\eta_x(t)\bigr),
\end{equation}
where, in turn, the comparison angles (see \autoref{def:comparison-triangle} for the Lorentzian analogue) are defined as 
\begin{equation}\label{eq:def-metric-comparison-angle}
\widetilde{\ma}_x\bigl(\xi_x(s),\eta_x(t)\bigr)=
\ma^{\R^{2}}_{\bar{x}}\bigl(\overline{\xi_x(s)}, \overline{\eta_x(t)}\bigr)=\cos^{-1}
\left(
\frac{s^2+t^2-d(\xi_x(s),\eta_x(t))^2}{2st}\right).
\end{equation}

With this definition, the angular distance induces a topology in $\bd X$ which is finer than the restriction to $\bd X$ of the cone topology. Moreover, if $X$ is a complete $\CAT(0)$ space, then $(\bd X,\ma)$ is a complete $\CAT(1)$ space \cite[Proposition~II.9.7, Theorem~II.9.13]{BH} which need not be a length space. However, being it $\pi$-geodesic, the intrinsic metric associated to $\ma$, called the \textit{Tits metric} $d_T$, coincides locally with $\ma$. Consequently, $(\bd X,d_T)$ is also a complete $\CAT(1)$ space and any pair of points at finite Tits distance can be joined by a geodesic \cite[Theorem~II.9.20]{BH}.

\subsection{Lorentzian pre-length spaces and timelike curvature bounds}\label{subs:LLS}
In this subsection we recall basic notions about Lorentzian pre-length spaces, based on \cite{kunzinger-saemann2018} and \cite{beran-saemann2023}. Throughout this and the following sections we will denote a Lorentzian pre-length space $(Y,d,\ll,\leq,\tau)$ simply by $Y$. We also use the classical notation for chronological and causal pasts and futures. Namely, given $P\subset Y$, define
\begin{align*}
I^-(P)&=\{q\in Y: q\ll p\ \text{ for some }\ p\in P\},\\
J^-(P)&=\{q\in Y: q\leq p\ \text{ for some }\ p\in P\},
\end{align*}
and analogously with $I^+(P)$ and $J^+(P)$. We also write $I^\pm(p)$ and $J^\pm(p)$ instead of $I^\pm(\{p\})$ and $J^\pm(\{p\})$ for any $p\in Y$. Given $p\leq q$, the corresponding causal diamond is denoted by $J(p,q)$.

\begin{dfn}
Let $Y$ be a Lorentzian pre-length space and $U \subseteq Y$ be an open set.  
\begin{itemize}
    \item $U$ is called \textit{timelike geodesically connected} if whenever $x,y \in U$ with $x \ll y$, there exists a future-directed maximal geodesic in $U$ from $x$ to $y$.  
    \item $U$ is called \textit{strictly timelike geodesically connected} if whenever $x,y \in U$ with $x \ll y$, there exists a future-directed maximal geodesic in $U$ from $x$ to $y$, and moreover any future-directed maximal geodesic in $U$ from $x$ to $y$ is timelike.  
    \item $Y$ is called \textit{locally strictly timelike geodesically connected} if it is covered by strictly timelike geodesically connected neighborhoods.
    \item $Y$ is called \textit{regular} if for all $x, y \in Y$ such that $x \ll y$ all geodesics connecting $x$ and $y$ are timelike.
\end{itemize}
\end{dfn}

\begin{dfn}\label{def:comparison-triangle}
Let $Y$ be a Lorentzian pre-length space, and let $\triangle p_1 p_2 p_3$ be a timelike triangle in $Y$, i.e., $p_1\ll p_2\ll p_3$. A \textit{comparison triangle} for $\triangle p_1p_2p_3$ in the Minkowski plane is a timelike triangle $\triangle \bar{p}_1 \bar{p}_2 \bar{p}_3$ in $\mathbb{R}^{1,1}$ such that
\[
\tau(p_i, p_j) = \tau(\bar{p}_i,\bar{p}_j) \quad \text{for } i < j,
\]
and the \textit{comparison angle} $\widetilde{\ma}_{p_i}(p_j,p_k)$ is given by
\[
\widetilde{\ma}_{p_i}(p_j,p_k) = \ma_{\bar{p}_i}^{\mathbb{R}^{1,1}}( \bar{p}_j, \bar{p}_k)=\cosh^{-1}\biggl(\frac{\tau_{i,j}^2+\tau_{i,k}^2-\tau_{j,k}^2}{2\sigma\tau_{i,j}\tau_{i,k}}\biggr),
\]
where $\sigma=1$ if $i\in\{1,3\}$ and $\sigma=-1$ otherwise, and $\tau_{i,j}=\max\{\tau(p_i,p_j),\tau(p_j,p_i)\}$.
\end{dfn}

The following is an elementary result about Minkowski space, analogous to the fact that the sum of interior angles of a Euclidean triangle equals $\pi$. See \cite[Lemma~2.4]{barton-beran-che-gieger-roehrig-rott}.

\begin{lem}\label{lem:sum of angles of triangle}
Let $a\ll b\ll c$ be points in the Minkowski plane, $\mathbb{R}^{1,1}$. Then
\[\ma_a(b, c) + \ma_c(a, b) = \ma_b(a, c).\]
\end{lem}

\begin{dfn}\label{def:upper-angle}
Let $Y$ be a Lorentzian pre-length space, where $\tau$ is locally finite valued.  Let $\alpha, \beta \colon [0,\varepsilon) \to Y$ be two future-directed or past-directed timelike curves with 
\[
\alpha(0) = \beta(0) =: x.
\]
Let
\begin{equation}\label{eq:set-upper-angle}
    A := \bigl\{ (s,t) \in (0,\varepsilon)^2 :
        \alpha(s) \leq \beta(t)\ \text{ or }\ \beta(t) \leq \alpha(s)
    \bigr\}.
\end{equation}
The \textit{upper angle} between $\alpha$ and $\beta$ at $x$ is defined as
\begin{equation}
    \ma_{x}(\alpha,\beta) 
    := \limsup_{\substack{(s,t) \in A \\ s,t \searrow 0}} 
       \widetilde{\ma}_{x}\bigl(\alpha(s),\beta(t)\bigr).
\end{equation}
\end{dfn}

\begin{dfn}\label{def:comparison_neighb}
Let $Y$ be a Lorentzian pre-length space. An open subset $U$ is called a \textit{non-positive timelike comparison neighborhood}, if
\begin{enumerate}
    \item $\tau|_{U \times U}$ is finite and continuous;
    \item $U$ is strictly timelike geodesically connected; and
    \item for all timelike triangles $\Delta p_1 p_2 p_3$ in $U$, for $q_1, q_2$ two points on different sides $\alpha$ and $\beta$, and for all comparison situations 
    $\Delta \bar{p}_1 \bar{p}_2 \bar{p}_3, \bar{q}_1, \bar{q}_2$ in the Minkowski plane, the time separations satisfy
    \[
       \tau(q_1,q_2) \;\geq\; \overline{\tau}(\bar{q}_1,\bar{q}_2).
    \]
\end{enumerate}
If every $x\in Y$ has a non-positive timelike comparison neighborhood, then \textit{$Y$ satisfies $\CBA(0)$}. Moreover, if $Y$ itself is a non-positive timelike comparison neighborhood, we say that \textit{$Y$ satisfies $\CBA(0)$ globally}. 
\end{dfn}

\begin{rem}
One can analogously define $\CBA(K)$ for general $K\in\mathbb{R}$. Since we are mainly interested in Lorentzian pre-length spaces satisfying $\CBA(0)$, \autoref{def:comparison_neighb} is sufficient for our purposes, even though we make reference to the general $\CBA(K)$ condition in some parts of the manuscript. We refer the reader to \cite{kunzinger-saemann2018} for this general treatment.
\end{rem}

\begin{dfn}\label{def:monotonicity}
Let $Y$ be a Lorentzian pre-length space. We say that $Y$ satisfies \textit{non-positive timelike monotonicity comparison} if every point in $Y$ possesses a neighborhood $U$ such that: 
\begin{enumerate}
    \item $\tau|_{U \times U}$ is finite and continuous;
    \item $U$ is timelike geodesically connected; and
    \item whenever $\alpha \colon [0, a] \to U$, $\beta \colon [0, b] \to U$ are timelike maximal geodesics in $U$ with $x := \alpha(0) = \beta(0)$, the function
    \[
        \theta \colon A \to \mathbb{R},\qquad        \theta(s,t) := \widetilde{\ma}_{x} \bigl(\alpha(s), \beta(t)\bigr),
    \]
    where $A$ is as in Equation~\eqref{eq:set-upper-angle} for $\varepsilon<\min\{a,b\}$, is non-decreasing in $s$ and $t$ if $\alpha$ and $\beta$ have the same time-orientation, and non-increasing if they have opposite time-orientations.
\end{enumerate}
\end{dfn}

Notice that in the last item of \autoref{def:monotonicity} we are using \textit{unsigned} angles instead of the signed angles used originally in \cite{beran-saemann2023}. As a consequence, we obtain one condition for curves with the same time orientation and the opposite condition for curves with different time orientations, whereas in the aforementioned paper there is no difference, as it is ``absorbed'' by the sign.

\begin{thm}[{\cite[Theorem~4.13]{beran-saemann2023}}]\label{thm:monotonicity of comparison angles}
Let $Y$ be a locally strictly timelike geodesically connected Lorentzian pre-length space. Then $Y$ satisfies $\CBA(0)$ if and only if it satisfies non-positive timelike monotonicity comparison.
\end{thm}

\begin{thm}[{\cite[Lemma~4.10]{beran-saemann2023}}]\label{thm:monotonicity implies existence of angles}
Let $Y$ be a Lorentzian pre-length space. If $Y$ satisfies $\CBA(K)$ for some $K\in\mathbb{R}$, then the angle between any two timelike geodesics starting at the same point $x$ which are both future-directed or both past-directed, exists.
\end{thm}

The following statement is a summary of parts of \cite[Theorems~3.1 and 4.5]{beran-saemann2023} and \cite[Lemma~4.14]{beran-kunzinger-rott2024} that will be useful in the following sections.
\begin{thm}\label{thm:triangle inequality for angles}
Let $Y$ be a strongly causal and locally causally closed Lorentzian pre-length space with $\tau$ locally finite valued and locally continuous. Let $\alpha, \beta, \gamma \colon [0,\varepsilon) \to Y$ be timelike curves starting at 
$x := \alpha(0) = \beta(0) = \gamma(0)$.
\begin{enumerate}
\item\label{item:same orientation triangle inequality} If $\alpha,\beta,\gamma$ have the same time orientation, then 
    \[
    \ma_{x}(\alpha,\gamma) \;\leq\; 
    \ma_{x}(\alpha,\beta) + \ma_{x}(\beta,\gamma).
    \]
\item If $\alpha,\beta$ are both future-directed, $\gamma$ is past-directed, and $\ma_x(\alpha,\gamma)$ exists, then 
    \[
    \ma_{x}(\alpha,\gamma)\leq \ma_{x}(\alpha,\beta) + \ma_x(\beta,\gamma)\, .    
    \]
\item\label{item:special case of triangle inequality} If $Y$ satisfies $\CBA(K)$ for some $K\in \R$, $\alpha,\beta$ are future-directed, $\gamma$ is past-directed, and the concatenation of $\gamma$ and $\beta$ is a geodesic, then
    \[
    \ma_{x}(\alpha,\gamma)\leq \ma_{x}(\alpha,\beta).    
    \]
\end{enumerate}
\end{thm}

\begin{thm}[{\cite[Proposition~3.12]{beran-saemann2023}}]\label{thm:continuity of angles}
Let $Y$ be a Lorentzian pre-length space satisfying $\CBA(K)$ for some $K\in\mathbb{R}$. Then angles are continuous for geodesics, that
is, for $x\in Y$ and $\alpha,\alpha_n,\beta,\beta_n\colon [0,b]\to Y$ all future or all past-directed timelike geodesics starting
at $x$ with $\alpha_n\to \alpha$, $\beta_n\to \beta$ pointwise, then
\[
\ma_x(\alpha,\beta) = \lim_{n\to\infty} \ma_x(\alpha_n,\beta_n).
\]
\end{thm}

For several results in Sections~\ref{sec:ideal-boundary}-\ref{sec:angular-metric} we will need to assume that the starting space satisfies a timelike curvature from above by $0$ globally. Recent developments provide conditions under which local curvature bounds extend to global bounds. We summarise these results here.

\begin{thm}[{\cite[Theorem~4.5]{beran-napper-rott2025}}]
Let $Y$ be a strongly causal, non-timelike locally isolating (see \cite[Definition~3.2.3]{beran_gluing_2024}) and regular \LpLS with (local) timelike curvature bounded above by $0$. Suppose that $\tau$ is continuous on $Y\times Y$ and that for every pair of timelike related points $x\ll y$ there is a unique geodesic $\gamma_{xy}$ joining them. Assume moreover that the geodesic map $G$ of $Y$ given by
\[
G\colon \ll \times [0,1] \to Y, \qquad G(x,y;t)=\gamma_{xy}(t)
\]
is continuous. Then $Y$ satisfies $\CBA(0)$ globally. 
\end{thm}

In this result, apart from the usual causality conditions, one needs to require two global results: namely (1) unique timelike geodesic connectedness and (2) geodesics depend continuously on their endpoints. These conditions are guaranteed by a geometric assumption: the local curvature bound; together with a topological assumption: future one-connectedness.

\begin{thm}[{\cite[Theorem~4.8]{eros-gieger-2025}}]
    Let $Y$ be a globally hyperbolic, strongly causal and locally causally closed regular Lorentzian pre-length space. If $Y$ is locally concave and future one-connected, then (1) every pair of timelike related points is connected by a unique geodesic, and (2) geodesics vary continuously with their endpoints.
\end{thm}

\begin{cor}[{\cite[Corollary~4.9]{eros-gieger-2025}}]\label{cor:globalization}
    Let $Y$ be a globally hyperbolic, strongly causal, and locally causally closed regular Lorentzian pre-length space in which $\tau$ is continuous. If $Y$ satisfies $\CBA(0)$ and is future one-connected, then it satisfies $\CBA(0)$ globally.
\end{cor}

Notice that unique timelike geodesic connectedness is not too strong of a condition to assume. In fact, every \LpLS satisfying $\CBA(0)$ globally will satisfy it:

\begin{thm}[{\cite[Theorem~4.7]{beran-napper-rott2025}}]\label{thm:uniqueness of geodesics in cba}
Let $Y$ be a strongly causal and regular Lorentzian pre-length space satisfying $\CBA(0)$ globally. Then geodesics between timelike related points in $Y$ are unique.
\end{thm}

On another note, we reproduce here a result that ensures that under a global curvature bound from above, timelike geodesics that meet forming an angle $0$ can be concatenated to form a geodesic.

\begin{thm}[{\cite[Lemma~5.10]{barton-beran-che-gieger-roehrig-rott}}]\label{thm:concatenation with angle zero}
Let $Y$ be a \LpLS satisfying $\CBA(K)$ globally and let $\alpha:[a,b]\to Y,\beta:[b,c]\to Y$ be two future-directed timelike geodesics with $x=\alpha(b)=\beta(b)$, $\ma_x(\alpha,\beta)=0$ and $\tau(\alpha(a),\alpha(b))+\tau(\beta(b),\beta(c))\leq D_K$, where $D_K$ is the finite diameter of $\mathbb{L}^2(K)$ as in \cite{beran-napper-rott2025}. Then the concatenation of $\alpha,\beta$ is a geodesic.
\end{thm}

The following version of the limit curve theorem is a summary of \cite[Theorems~3.7 and 3.14]{kunzinger-saemann2018}.

\begin{thm}\label{thm:limit curve theorem}
Let $Y$ be a locally causally closed Lorentzian pre-length space. Let $(\gamma_n)_{n \in \mathbb{N}}$ be a sequence of future-directed causal curves
$\gamma_n \colon [a,b] \to Y$
that are uniformly Lipschitz continuous, i.e., there exists $L > 0$ such that 
\[
\mathrm{Lip}(\gamma_n) \leq L \quad \text{for all } n \in \mathbb{N}.
\]
Suppose that either there exists a compact set that contains every $\gamma_n([a,b])$, or $d$ is proper (i.e., all closed and bounded sets are compact) and that the curves $(\gamma_n)_n$ accumulate at some point, i.e., there exists $t_0 \in [a,b]$ such that 
\[
\gamma_n(t_0) \to x_0 \in Y.
\]  

Then there exists a subsequence $(\gamma_{n_k})_{k}$ of $(\gamma_n)_n$ and a Lipschitz continuous curve 
$\gamma \colon [a,b] \to Y$
such that $\gamma_{n_k} \to \gamma$ uniformly. Moreover, if $\gamma$ is non-constant, then $\gamma$ is a future-directed causal curve. In particular, if $\gamma_n(a) = p$ and $\gamma_n(b) = q$ for all $n \in \mathbb{N}$, with $p \neq q$, then $\gamma$ is a future-directed causal curve connecting $p$ and $q$.
\end{thm}

\subsection{Metric and Lorentzian generalized products}

In this section we focus on a special kind of spaces, namely \textit{generalized} (or \textit{warped}) products. These have been developed both in the Riemannian \cite{bishop_manifolds_1969} and semi-Riemannian settings \cite{oneill_semi-riemannian_1983}, and further extended to the corresponding metric \cite{chen_warped_1999} and Lorentzian synthetic \cite{alexander_generalized_2023} settings. In the latter case, due to the Lorentzian signature it will be required that the base (i.e., the first factor) be one-dimensional. These objects are called generalized \textit{cones} in \cite{alexander_generalized_2023}, naming that we will follow for the Lorentzian case to distinguish it from the metric case. The Lorentzian ideal boundary of generalized cones will be studied in \autoref{sec:cones} in terms of the ideal boundary of its fiber. For brevity, we introduce here only the synthetic versions.

\begin{dfn}
    Let $(B,d_B)$ and $(F,d_F)$ be length spaces and let $f\colon B\to (0,\infty)$ be continuous. Consider the product $B\times F$ and a curve $\gamma=(\gamma_B,\gamma_F)\colon I\to B\times F$ such that $\gamma_B$ and $\gamma_F$ are $d_B$- and $d_F$-rectifiable, respectively. We define its \textit{length} as 
    \begin{equation}\label{eq:length-structure-cones}
    L(\gamma)=\int_I \sqrt{v_B^2+\left(f\circ \gamma_B\right)^2 v_F^2\,},
    \end{equation}
    where $v_B$ and $v_F$ are the \textit{metric speeds} of $\gamma_B$ and $\gamma_F$, respectively (which exist almost everywhere \cite[Theorem~2.7.6]{BBI}). We define the \textit{warped product} with \textit{base} $B$, \textit{fiber} $F$ and \textit{warping function} $f$, which we denote $B\times_f F$, as the length space $(B\times F, d)$ whose distance is defined by the infimum of lengths of curves.
\end{dfn}

In the case in which the warping function $f$ is allowed to attain the value $0$ some modifications are in order \cite{alexander_warped_2016}. As a topological space, the warped product $B\times_f F$ is the quotient $Z$ of $B\times F$ under the equivalence relation that identifies all the elements in $\{p\}\times F$ whenever $f(p)=0$. To define the length structure in this setting one considers curves $\gamma=(\gamma_B,\gamma_F)\colon I\to Z$ and divides their domain $J$ into $J_0:=(f\circ \gamma_B)^{-1}(0)$ and $J_+=J\backslash J_0$. The length is then defined to be as in Equation~\eqref{eq:length-structure-cones} but defining $v_F=0$ in $J_0$.

We will be particularly interested in the case in which $B\subset \R$ is an interval and $F\in \CAT(0)$, as stated in the following result.
\begin{thm}{\cite[Theorem~1.1]{alexander_curvature_2004}}
    Let $I\subset \R$ be a closed interval and let $X$ be a $\CAT(0)$ space. Assume that $f\colon I \to [0,\infty)$ is $K$-convex (i.e., $f''-Kf\geq 0$). If $f^{-1}(0)=\varnothing$, assume that $K\inf f\geq 0$. Then $I\times_f X$ is a $\CAT(K)$ space.
\end{thm}

In the Lorentzian framework, the analogue construction involves the machinery of Lorentzian pre-length spaces introduced by \cite{kunzinger-saemann2018} (see \autoref{subs:LLS}). Now, one starts with a metric space $X$ (the fiber) and an interval (the base). As mentioned before, the base is assumed to be one-dimensional in order to achieve the Lorentzian signature. Moreover, one needs to check under which conditions the resulting generalized cone is a Lorentzian pre-length space. This was all done in \cite{alexander_generalized_2023}, whose main results and definitions we briefly present for an easier reading.

\begin{dfn}\label{def:generalized-cones}
    Let $(X,d)$ be a metric space, $I\subset \R$ be an interval and $f\colon I\to \R$ be continuous. Consider the product $Y=I\times X$ endowed with the product distance. We say that an absolutely continuous curve $\gamma=(\alpha,\beta)\colon J\to Y$ is \textit{timelike} (resp.\ \textit{null}, or \textit{causal}) if $\smash{-\dot{\alpha}^2+(f\circ \alpha)^2\, v_{\beta}^2<0}$ (resp.\ $=0$, or $\leq 0$) almost everywhere. Again $v_{\beta}$, which exists almost everywhere \cite[Theorem~2.7.6]{BBI}, is the metric speed of $\beta$. The curve is called \textit{future-directed} (resp.\ \textit{past-directed}) if $\alpha$ is strictly increasing (resp.\ strictly decreasing). One can define a \textit{timelike} and a \textit{causal relation} in the usual way. We denote $\prescript{-}{}{I}\times_f X$ the space $Y$ endowed with this \textit{causal structure} and call it \textit{generalized cone} with \textit{warping function} $f$. The superscript $^-$ before the product is meant to indicate the Lorentzian signature.
\end{dfn}

\begin{dfn}\label{def:lorentzian-length-cones}
    Let $Y$ be a generalized cone and let $\gamma=(\alpha,\beta)\colon J\to Y$ be a causal curve. Its \textit{Lorentzian length} is defined as 
    \[
    L(\gamma):=\int_J \sqrt{\dot{\alpha}^2-\left(f\circ\alpha\right)^2 v_{\beta}^2\,}.
    \]
    One can define a \textit{time separation function} $\tau\colon Y\times Y\to[0,\infty)$ in the usual sense, with the convention $\tau(y_1,y_2)=0$ whenever there is no future-directed causal curve from $y_1$ to $y_2$. A causal curve whose length agrees with the time separation between its endpoints is called \textit{maximal}.
\end{dfn}

We summarize now a few properties of generalized cones.

\begin{thm}{\cite[Proposition~3.26, Corollary~3.30, Theorem~4.10]{alexander_generalized_2023}}\label{thm:structure-generalizedcones}
    Let $Y=\prescript{-}{}{I}\times_f X$ be a generalized cone, where $(X,d)$ is a length space. Then $(Y,D,\ll,\leq,\tau)$, where $D$ is the product distance, is a Lorentzian pre-length space. In particular, $\tau$ is lower semicontinuous, satisfies the reverse triangle inequality, and $\tau(y_1,y_2)>0$ implies $y_1\ll y_2$. If $(X,d)$ is also proper and geodesic, then $Y$ is a globally hyperbolic and regular Lorentzian length space. 
\end{thm}

\begin{thm}{\cite[Theorem~3.29]{alexander_generalized_2023}}\label{thm:properties-geodesics-generalizedcones}
    Let $(X,d)$ be a geodesic space and let $\gamma=(\alpha,\beta)\colon J\to Y=\prescript{-}{}{I}\times_f X$ be a future directed maximal geodesic. Then
    \begin{enumerate}
        \item $\beta$ is minimizing in $X$.
        \item If $\gamma$ is timelike, then it admits an absolutely continuous parametrization by arclength, and in such parametrization $v_\beta$ is proportional to $(f\circ \alpha)^{-2}$.
    \end{enumerate}
\end{thm}

\begin{thm}{\cite[Corollary~5.4]{alexander_generalized_2023}}\label{thm:curvature_cones}
    Let $Y=\prescript{-}{}{I}\times_f X$ be a generalized cone with $f$ smooth. Assume that $f$ is $K$-convex (i.e., that $f''-Kf\geq 0$), and that $(X,d)$ is a geodesic space with curvature bounded above by $\inf\{Kf^2-(f')^2\}$. Then $Y$ satisfies $\CBA(K)$.
\end{thm}

This result connects an upper curvature bound in the metric sense on the fibre $X$ and a convexity condition on the warping function $f$, with a local upper timelike curvature bound on the generalized cone $Y$. In our case, we will be mostly interested in obtaining an upper timelike curvature bound by $0$ in $Y$. Therefore, it will suffice to assume that $f$ is convex and that $X$ has curvature bounded above by $-\sup\{f'\}^2$ (see also \ref{pop:restricted-cone-2}).

\begin{rem}
    If in \autoref{thm:curvature_cones} one assumes moreover that $(X,d)$ is locally compact and complete, the generalized cone $Y$ is globally hyperbolic and regular Lorentzian length space by \autoref{thm:structure-generalizedcones}. In particular, $\tau$ is finite and continuous. If $X$ has curvature bounded above by $-\sup\{f'\}^2$, then it will be uniquely geodesic. As a consequence (see \autoref{rem:uniquely-geodesic}) the generalized cone is also uniquely geodesic. By \cite[Theorems~4.5, 4.8]{beran-napper-rott2025} we deduce that the generalized cone satisfies $\CBA(0)$ \textit{globally}.
\end{rem}

\begin{rem}\label{rem:cones-global-curvature}
    Let $Y:=\prescript{-}{}{\mathbb{R}}\times_f X$ be a generalized cone with smooth and convex warping function $f$ converging to $L>0$, and where $X$ is proper and geodesic. Then $Y$ is a globally hyperbolic and regular Lorentzian length space (\autoref{thm:structure-generalizedcones}). In particular, $\tau$ is finite and continuous \cite[Theorem~3.28]{kunzinger-saemann2018}. Assume that there exists $\varepsilon>0$ such that $X\in\CAT(-\varepsilon)$. Then there exists $a\in \R$ such that the generalized cone $Y_a:=\smash{\prescript{-}{}{(a,\infty)}\times_f X}$ satisfies $\CBA(0)$ locally (\autoref{thm:curvature_cones}). If $Y_a$ is future one-connected, then it satisfies a global curvature bound from above by $0$ (\autoref{cor:globalization}).
\end{rem}

\section{Timelike ideal boundary\label{sec:ideal-boundary}}
In this section we introduce the notion of timelike ideal boundary of a Lorentzian pre-length space. We first recall a couple of definitions.

\begin{dfn}[{\cite[Definition~4.1]{beran-ohanyan-rott-solis2023}}]\label{def:timelike geodesic rays}
Let $Y$ be Lorentzian pre-length space. A \textit{timelike geodesic ray} is a $\tau$-arclength parametrized (thus timelike) curve $\gamma\colon [0,\infty)\to Y$ which is a maximal geodesic between any of its points.
\end{dfn}

\begin{dfn}[{\cite[Definition~5.4]{barton-beran-che-gieger-roehrig-rott}}]\label{def:asymptotic}
Let $Y$ be a Lorentzian pre-length space and $\alpha,\beta\colon [0,\infty)\to Y$ be two future-directed timelike geodesic rays in $Y$. We say that \textit{$\alpha$ is in the past of $\beta$}, and write $\alpha\ll \beta$, if there exists $c>0$ such that for all $t\in\mathbb{R}$, 
\[
\alpha(t)\ll \beta(t+c).
\]
If $\alpha\ll\beta$ and $\beta\ll \alpha$ we say that they are \textit{asymptotic} or \textit{weakly parallel}, which we denote $\alpha\sim\beta$. 
\end{dfn}

It is immediate that being asymptotic defines an equivalence relation in the set of all future-directed timelike geodesic rays. This motivates the following Lorentzian version of \autoref{def:metric-ideal-boundary}.

\begin{dfn}
The \textit{future timelike ideal boundary} of a \LpLS $Y$ is the set $\bd^+Y$ of equivalence classes of future-directed timelike geodesic rays under the relation of being asymptotic. The \textit{future timelike ideal completion} of $Y$ is the set $Y^*=Y\cup \bd^+Y$. Given a future-directed timelike geodesic ray $\alpha$, we denote by $\alpha(\infty)$ the equivalence class of $\alpha$ in $\bd^+Y$, and we call elements of $\bd^+Y$ \textit{future timelike boundary points}. By reversing causal and timelike relations we obtain the notion of \textit{past} timelike ideal boundary, denoted by $\bd^-Y$, and so on.
\end{dfn}

For brevity, we will occasionally omit the word ``timelike'' when referring to the ideal boundary. Furthermore, in what follows we will formulate some of the results and definitions for the future case, as the past case is completely analogous. Accordingly, the past ideal boundary will be mentioned almost exclusively when both the future and past ideal boundaries are treated simultaneously. In those instances, we will call $\bd^+Y\cup \bd^-Y$ the \textit{ideal boundary}, and refer to either the future or the past ideal boundary as a \textit{partial ideal boundary}.

\begin{rem}
    Let $\alpha\ll \beta$. Then for any $\lambda>0$, one can choose the parameter shift $c>0$ in such a way that $\tau\bigl(\alpha(t),\beta(t+c)\bigr)\geq \lambda$ for all $t$.
\end{rem}

\begin{rem}
    If $\alpha\ll \beta$, then $I^-(\alpha)\subset I^-(\beta)$. In particular, the elements of a given equivalence class of asymptotic curves have all the same past. The converse is not true, as one can easily see in the Minkowski space.
\end{rem}

The next proposition asserts the uniqueness of asymptotic rays to a given future-directed timelike ray, starting from a given point in the past of that ray. This is the first part of \cite[Proposition~5.9]{barton-beran-che-gieger-roehrig-rott}. We provide here a more explicit proof.

\begin{pop}\label{pop:uniqueness of pointwise rep}
Let $Y$ be a strongly causal, regular Lorentzian pre-length space satisfying $\CBA(0)$ globally. Let $\beta\colon[0,\infty)\to Y$ be a future-directed timelike ray and $x\in Y$. Then the there is at most one future-directed timelike ray $\alpha\colon [0,\infty)\to Y$ such that $\alpha(0)=x$ and $\alpha(\infty) = \beta(\infty)$.
\end{pop}

\begin{proof}
Let $\alpha,\alpha'\colon [0,\infty)\to Y$ be future-directed timelike rays such that $\alpha(\infty)=\alpha'(\infty)=\beta(\infty)$ and $\alpha(0)=\alpha'(0)=x$. Then, for any fixed $\lambda>0$, there exists $c>0$ such that 
$\tau(\alpha(t),\alpha'(t+c))\geq \lambda$
for all $t\in [0,\infty)$.

Since $Y$ satisfies $\CBA(0)$ globally, by \autoref{thm:monotonicity of comparison angles}, it follows that
\[
\widetilde{\ma}_x(\alpha(t),\alpha'(t+c)) \leq \widetilde{\ma}_x(\alpha(t'),\alpha'(t'+c)),
\]
for any $0\leq t\leq t'$. However,
\begin{align*}
\widetilde{\ma}_x(\alpha(t'),\alpha'(t'+c)) &= \cosh^{-1}\left(\frac{t'^2+(t'+c)^2-\tau(\alpha(t'),\alpha'(t'+c))^2}{2t'(t'+c)}\right)\\
&\leq \cosh^{-1}\left(\frac{t'^2+(t'+c)^2-\lambda^2}{2t'(t'+c)}\right)\\
&= \cosh^{-1}\left(1+\frac{c^2-\lambda^2}{2t'(t'+c)}\right) \xrightarrow{t'\to \infty} 0.
\end{align*}
Therefore $\widetilde{\ma}_x(\alpha(t),\alpha'(t+c)) = 0$, which in turn implies that $x$, $\alpha(t)$ and $\alpha'(t+c)$ lie on a timelike geodesic. By \autoref{thm:uniqueness of geodesics in cba},  $\alpha(t) = \alpha'(t)$ for any $t\in[0,\infty)$.
\end{proof}

The following lemma will be instrumental to guarantee the existence of asymptotic rays in Lorentzian pre-length spaces with non-positive timelike curvature.

\begin{lem}\label{lem:minkowski-situation-bounded-angles}
Let $p,q\in \mathbb{R}^{1,1}$ be such that $p\ll q$, and $\alpha_o \geq 0$. Then there is $c_o\geq 1$ such that, if $r\in \mathbb{R}^{1,1}$ satisfies  $\ma_p(q,r)\leq \alpha_o$ and $\tau(p,r)>c_o\tau(p,q)$, then $q\ll r$.
\end{lem}

\begin{proof}
Without loss of generality we can assume that $p = (0,0)$ and $q = (\tau(p,q),0)$ where, as usual, the first coordinate is the time coordinate and the second one is the spacial one. Let $0\leq \alpha\leq \alpha_o$ and $\gamma_{\alpha}\colon [0,\infty)\to \mathbb{R}^{1,1}$ be given by $\gamma_{\alpha}(t) = (t\cosh\alpha, t\sinh\alpha)$. Such a curve forms an angle $\alpha$ at $p$ with the segment $\overline{pq}$. One has $q\ll \gamma_{\alpha}(t)$ if and only if 
\[
t(\cosh\alpha - \sinh\alpha) > \tau(p,q).
\]
Moreover,
\begin{equation}\label{eq:bound1}
t(\cosh\alpha - \sinh\alpha) = \frac{t}{\cosh\alpha + \sinh\alpha} \geq \frac{t}{\cosh\alpha_o + \sinh\alpha_o}
\end{equation}
for any $0\leq \alpha\leq \alpha_o$ and $t>0$. Therefore, by setting $c_o = \cosh\alpha_o + \sinh\alpha_o\geq 1$, the claim follows. Indeed, if $t=\tau(p,r)>c_o\tau(p,q)$, then the timelike condition above is satisfied.
\end{proof}

The following lemma can be proved along the same lines as \cite[Lemma~3.33]{beran-saemann2023}.
\begin{lem}\label{lem:bounded comparison angles imply timelike limit}
Let $Y$ be a strongly causal and regular Lorentzian pre-length space satisfying $\CBA(0)$ globally. Let $\gamma_n\colon [0,\varepsilon]\to Y$ be future-directed timelike geodesics (non-necessarily parametrized by $\tau$-length), starting at $\gamma_n(0)=x$ and converging uniformly to a causal curve $\gamma\colon [0,\varepsilon]\to Y$. If there exists $C>0$ such that
\[
\sup\bigl\{\widetilde\ma_x(\gamma_n(t),\gamma_1(t')):\gamma_n(t)\ll\gamma_1(t')\ \text{ or }\ \gamma_1(t')\ll\gamma_n(t)\bigr\}<C
\]
for all $n$, then $\gamma$ is timelike.
\end{lem}

We now prove the existence of an asymptotic ray to a given future-directed timelike ray, starting from a point in the past of the given ray. This is the remaining part of \cite[Proposition~5.9]{barton-beran-che-gieger-roehrig-rott}. To that end we prove that under our assumptions the \textit{timelike co-ray condition} of Beem-Ehrlich-Markvorsen-Galloway \cite[Definition~3.1]{beem-ehrlich-markvorsen-galloway} is satisfied for every timelike ray.

\begin{pop}\label{pop:existence of pointwise rep}
Let $Y$ be a proper, strongly causal, locally
causally closed, regular Lorentzian pre-length space that satisfies $\CBA(0)$ globally. Let $\gamma\colon [0,\infty)\to Y$ be a future-directed timelike ray and $p\ll \gamma(0)$. Then there is exactly one future-directed timelike ray, $\gamma_p$, starting at $p$ and weakly parallel to $\gamma$. Furthermore, $\gamma_p$ is the pointwise limit of arclength-parametrized maximal geodesics of the form $[p,\gamma(t_n)]$ where $t_n\to\infty$. 
\end{pop}

\begin{proof}
First observe that, by \autoref{thm:monotonicity of comparison angles}, the function
\[
[1,\infty) \ni t\mapsto \widetilde{\ma}_{\gamma(0)}(p,\gamma(t)) 
\]
is non-increasing and, in particular it is bounded. Let $\alpha_o > 0$ be such that 
\begin{equation}\label{eq:alpha_o}
\widetilde{\ma}_{\gamma(0)}(p,\gamma(t)) \leq \alpha_o \quad \text{ for all }\quad t\in [1,\infty).
\end{equation}
\autoref{lem:sum of angles of triangle} combined with \eqref{eq:alpha_o} yield 
\begin{equation}\label{eq:comparison angle at p bounded}
\widetilde{\ma}_p(\gamma(0),\gamma(t)) \leq \alpha_o
\end{equation}
for all $t\in [1,\infty)$.

Now, for any sequence $t_n\to \infty$ the reverse triangle inequality for $\tau$ implies $\tau(p,\gamma(t_n))\to \infty$. By \cite[Proposition 4.3]{beran-ohanyan-rott-solis2023}, if $\sigma_n$ is a timelike maximal geodesic from $p$ to $\gamma(t_n)$, parametrized by $d$-arclength, then there is a causal ray $\gamma_p$ starting at $p$ such that $\sigma_n \to \gamma_p$ locally uniformly (see \autoref{subfig:sigma_n-existencerep}). Moreover, since $I^-(\gamma(0))$ is open and $\gamma_p$ is continuous, it follows that $\sigma_n([0,\epsilon]),\gamma_p([0,\epsilon])\subset I^-(\gamma(0))$ for some $\epsilon>0$ and sufficiently large $n$. In particular, the sequence of timelike maximizers $\sigma_n|_{[0,\epsilon]}$ converges uniformly to $\gamma_p|_{[0,\epsilon]}$ and by \autoref{thm:monotonicity of comparison angles} applied to $\triangle p\sigma_n(\epsilon)\gamma(0)$,
\[
\widetilde{\ma}_p(\gamma(0),\sigma_n(\epsilon)) \leq \widetilde{\ma}_p(\gamma(0),\gamma(t_n)) \leq \alpha_o.
\] 
\autoref{lem:bounded comparison angles imply timelike limit} implies that $\gamma_p|_{[0,\epsilon]}$ is timelike, and since maximal geodesics in regular Lorentzian pre-length spaces have causal character, it follows that $\gamma_p$ is timelike. One can also prove that $\gamma_p$ is a timelike ray and that $\sigma_n$ converges pointwise to $\gamma_p$ in $\tau$-arclength parametrizations \cite[Lemmata~5.7, 5.8]{barton-beran-che-gieger-roehrig-rott}. Hence, we can abuse the notation and denote by $\gamma_p$ and $\sigma_n$ their corresponding $\tau$-arclength parametrizations.

We now proceed to show that $\gamma_p$ is parallel to $\gamma$. Indeed, observe that for any $s\in [0,\infty)$, since $t_n$ is unbounded, the global $\CBA(0)$ condition on the triangle $\triangle p\gamma(0)\gamma(t_n)$, for sufficiently large $n$, implies
\begin{align*}
\tau(\sigma_n(s),\gamma(s))&\geq \tau\bigl(\overline{\sigma_n(s)},\overline{\gamma(s)}\bigr)\\
&= \sqrt{\tau(p,\gamma(0))^2+2s\bigl(t_n+\tau(p,\gamma(t_n))-s\bigr)\bigl(\cosh\widetilde{\ma}_{\gamma(t_n)}(\gamma(0),p)-1\bigr)},
\end{align*}
where to pass from the first line to the second one we used that, working in the comparison space $\R^{1,1}$, the $\cosh$ on the previous equation can be expressed in the two following ways (see \autoref{subfig:triangle-existencerep}: the angle is the one on the top vertex):
\[
\frac{t_n^2+\tau(p,\gamma(t_n))^2-\tau(p,\gamma(0))^2}{2t_n\tau(p,\gamma(t_n))}=
\frac{(t_n-s)^2+\tau\bigl(\overline{\sigma_n(s)},\overline{\gamma(t_n)}\bigr)^2-\tau\bigl(\overline{\sigma_n(s)},\overline{\gamma(s)}\bigr)^2}{2(t_n-s)\tau\bigl(\overline{\sigma_n(s)},\overline{\gamma(s)}\bigr)},
\]
and that $\tau\bigl(\overline{\sigma_n(s)},\overline{\gamma(s)}\bigr)=\tau(p,\gamma(t_n))-s$.

This implies
\begin{equation}\label{eq:gamma_p in the past of gamma}
\tau(\sigma_n(s),\gamma(s))\geq \tau(p,\gamma(0))
\end{equation}
and, letting $n \to \infty$, we get
\begin{equation}\label{eq:distance between gamma and gamma_p bounded below}
     \tau(\gamma_p(s),\gamma(s)) \geq \tau(p,\gamma(0)) > 0.
\end{equation}
In particular, $\gamma_p(s)\ll \gamma(s)$ for any $s\in [0,\infty)$.

We now prove that there is some $c_o$ such that $\gamma(s)\ll \gamma_p(s+c_o)$ for all $s \in [0,\infty)$.  Indeed, 
we can apply \autoref{lem:minkowski-situation-bounded-angles} to the comparison triangles $\triangle \bar{p}\overline{\gamma(0)}\,\overline{\gamma(t_n)}$ in $\mathbb{R}^{1,1}$, which satisfy the following: 
\begin{itemize}
\item they have a fixed side-length $\tau(p,\gamma(0))$; 
\item the angle $\ma_{\bar{p}}(\overline{\gamma(0)},\overline{\gamma(t_n)})=\widetilde{\ma}_{p}(\gamma(0),\gamma(t_n))$ is bounded above by $\alpha_o$, by inequality \eqref{eq:comparison angle at p bounded};
\item they have an unbounded side-length $\tau(\overline{\gamma(0)},\overline{\gamma(t_n)})=t_n$. 
\end{itemize}

\begin{figure}
\centering
\begin{subfigure}{0.45\textwidth}
    \centering
    \tikzset{every picture/.style={line width=0.75pt}}    
    \begin{tikzpicture}[x=0.75pt,y=0.75pt,yscale=-0.9,xscale=0.9]
    \draw [color={rgb, 255:red, 74; green, 144; blue, 226 }  ,draw opacity=1 ][line width=1.5]    (204.2,38.6) .. controls (246.6,85.8) and (334.2,210.2) .. (351.8,260.6) ;
    \draw  [fill={rgb, 255:red, 0; green, 0; blue, 0 }  ,fill opacity=1 ] (368.1,165) .. controls (368.1,163.84) and (369.04,162.9) .. (370.2,162.9) .. controls (371.36,162.9) and (372.3,163.84) .. (372.3,165) .. controls (372.3,166.16) and (371.36,167.1) .. (370.2,167.1) .. controls (369.04,167.1) and (368.1,166.16) .. (368.1,165) -- cycle ; 
    \draw [color={rgb, 255:red, 128; green, 128; blue, 128 }  ,draw opacity=1 ]   (351.8,260.6) .. controls (356.75,210.5) and (357.82,172.85) .. (353.02,123.25) ;
    \draw [color={rgb, 255:red, 128; green, 128; blue, 128 }  ,draw opacity=1 ]   (351.8,260.6) .. controls (345.5,195.5) and (326.65,97.1) .. (299.4,34.6) ;
    \draw [color={rgb, 255:red, 128; green, 128; blue, 128 }  ,draw opacity=1 ]   (351.8,260.6) .. controls (337.5,205.25) and (288,82.25) .. (239.25,14) ;
    \draw [line width=1.5]    (277.8,10.6) .. controls (314.6,44.6) and (352.6,114.6) .. (370.2,165) ;
    \draw  [fill={rgb, 255:red, 0; green, 0; blue, 0 }  ,fill opacity=1 ] (349.7,260.6) .. controls (349.7,259.44) and (350.64,258.5) .. (351.8,258.5) .. controls (352.96,258.5) and (353.9,259.44) .. (353.9,260.6) .. controls (353.9,261.76) and (352.96,262.7) .. (351.8,262.7) .. controls (350.64,262.7) and (349.7,261.76) .. (349.7,260.6) -- cycle ;
    \draw  [fill={rgb, 255:red, 0; green, 0; blue, 0 }  ,fill opacity=1 ] (350.92,123.25) .. controls (350.92,122.09) and (351.86,121.15) .. (353.02,121.15) .. controls (354.18,121.15) and (355.12,122.09) .. (355.12,123.25) .. controls (355.12,124.41) and (354.18,125.35) .. (353.02,125.35) .. controls (351.86,125.35) and (350.92,124.41) .. (350.92,123.25) -- cycle ;
    \draw  [fill={rgb, 255:red, 0; green, 0; blue, 0 }  ,fill opacity=1 ] (297.3,34.6) .. controls (297.3,33.44) and (298.24,32.5) .. (299.4,32.5) .. controls (300.56,32.5) and (301.5,33.44) .. (301.5,34.6) .. controls (301.5,35.76) and (300.56,36.7) .. (299.4,36.7) .. controls (298.24,36.7) and (297.3,35.76) .. (297.3,34.6) -- cycle ;
    
    \draw (300,107.7) node [anchor=north west][inner sep=0.75pt]  [color={rgb, 255:red, 128; green, 128; blue, 128 }  ,opacity=1 ]  {$\cdots $};
    \draw (377,156.25) node [anchor=north west][inner sep=0.75pt]  [font=\small]  {$\gamma ( 0)$};
    \draw (358.05,110.45) node [anchor=north west][inner sep=0.75pt]  [font=\small]  {$\gamma ( t_{1})$};
    \draw (303.35,15.95) node [anchor=north west][inner sep=0.75pt]  [font=\small]  {$\gamma ( t_{2})$};
    \draw (360.55,199.95) node [anchor=north west][inner sep=0.75pt]  [color={rgb, 255:red, 128; green, 128; blue, 128 }  ,opacity=1 ]  {$\sigma _{1}$};
    \draw (294.45,74.15) node [anchor=north west][inner sep=0.75pt]  [color={rgb, 255:red, 128; green, 128; blue, 128 }  ,opacity=1 ]  {$\sigma _{2}$};
    \draw (236,37.9) node [anchor=north west][inner sep=0.75pt]  [color={rgb, 255:red, 128; green, 128; blue, 128 }  ,opacity=1 ]  {$\sigma _{n}$};
    \draw (258.45,133.65) node [anchor=north west][inner sep=0.75pt]  [font=\small,color={rgb, 255:red, 74; green, 144; blue, 226 }  ,opacity=1 ]  {$\gamma _{p}$};
    \draw (337,65.05) node [anchor=north west][inner sep=0.75pt]  [font=\normalsize]  {$\gamma $};
    \draw (359.4,255.65) node [anchor=north west][inner sep=0.75pt]  [font=\small]  {$p$};    
    \end{tikzpicture}
    \caption{}
    \label{subfig:sigma_n-existencerep}
\end{subfigure}\hspace{0.5cm}
\begin{subfigure}{0.45\textwidth}
    \centering
    \tikzset{every picture/.style={line width=0.75pt}}       
    \begin{tikzpicture}[x=0.75pt,y=0.75pt,yscale=-0.9,xscale=0.9]
    \draw    (279.5,20) -- (217,228.5) ;
    \draw    (279.5,20) -- (273.5,158) ;
    \draw    (273.5,158) -- (217,228.5) ;
    \draw  [fill={rgb, 255:red, 0; green, 0; blue, 0 }  ,fill opacity=1 ] (214.25,228.5) .. controls (214.25,226.98) and (215.48,225.75) .. (217,225.75) .. controls (218.52,225.75) and (219.75,226.98) .. (219.75,228.5) .. controls (219.75,230.02) and (218.52,231.25) .. (217,231.25) .. controls (215.48,231.25) and (214.25,230.02) .. (214.25,228.5) -- cycle ; 
    \draw  [fill={rgb, 255:red, 0; green, 0; blue, 0 }  ,fill opacity=1 ] (276.75,20) .. controls (276.75,18.48) and (277.98,17.25) .. (279.5,17.25) .. controls (281.02,17.25) and (282.25,18.48) .. (282.25,20) .. controls (282.25,21.52) and (281.02,22.75) .. (279.5,22.75) .. controls (277.98,22.75) and (276.75,21.52) .. (276.75,20) -- cycle ;
    \draw  [fill={rgb, 255:red, 0; green, 0; blue, 0 }  ,fill opacity=1 ] (270.75,158) .. controls (270.75,156.48) and (271.98,155.25) .. (273.5,155.25) .. controls (275.02,155.25) and (276.25,156.48) .. (276.25,158) .. controls (276.25,159.52) and (275.02,160.75) .. (273.5,160.75) .. controls (271.98,160.75) and (270.75,159.52) .. (270.75,158) -- cycle ;
    \draw  [fill={rgb, 255:red, 0; green, 0; blue, 0 }  ,fill opacity=1 ] (245.5,124.25) .. controls (245.5,122.73) and (246.73,121.5) .. (248.25,121.5) .. controls (249.77,121.5) and (251,122.73) .. (251,124.25) .. controls (251,125.77) and (249.77,127) .. (248.25,127) .. controls (246.73,127) and (245.5,125.77) .. (245.5,124.25) -- cycle ;
    \draw  [fill={rgb, 255:red, 0; green, 0; blue, 0 }  ,fill opacity=1 ] (274,84.25) .. controls (274,82.73) and (275.23,81.5) .. (276.75,81.5) .. controls (278.27,81.5) and (279.5,82.73) .. (279.5,84.25) .. controls (279.5,85.77) and (278.27,87) .. (276.75,87) .. controls (275.23,87) and (274,85.77) .. (274,84.25) -- cycle ;
    \draw    (276.75,84.25) -- (248.25,124.25) ;
    
    \draw (208,109.9) node [anchor=north west][inner sep=0.75pt]  [font=\small]  {$\overline{\sigma _{n}( s)}$};
    \draw (202,230.4) node [anchor=north west][inner sep=0.75pt]  [font=\small]  {$\overline{p}$};
    \draw (284.5,9.9) node [anchor=north west][inner sep=0.75pt]  [font=\small]  {$\overline{\gamma ( t_{n})}$};
    \draw (279.5,149.4) node [anchor=north west][inner sep=0.75pt]  [font=\small]  {$\overline{\gamma ( 0)}$};
    \draw (284,72.9) node [anchor=north west][inner sep=0.75pt]  [font=\small]  {$\overline{\gamma ( s)}$};
    \end{tikzpicture}
    \caption{}
    \label{subfig:triangle-existencerep}
\end{subfigure}
    \caption{Some situations described in the proof of \autoref{pop:existence of pointwise rep}.}
    \label{fig:existence-representative}
\end{figure}
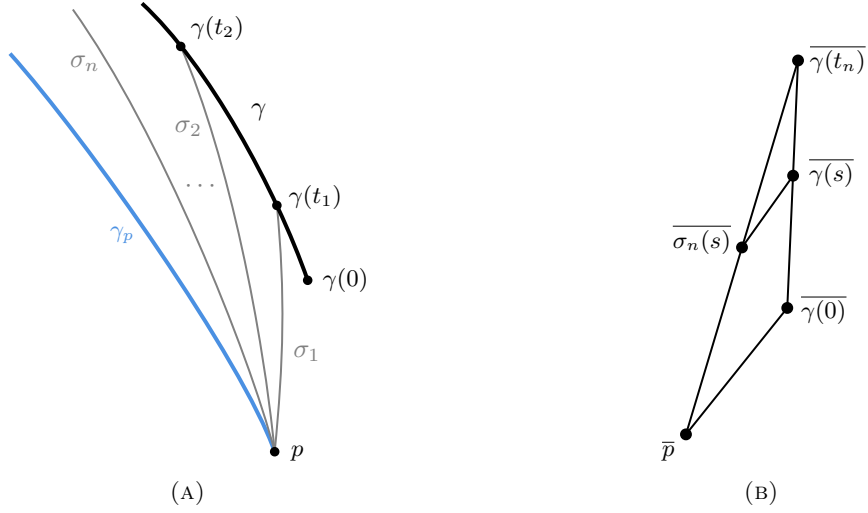

This way, we obtain $\tilde{c}_o=\cosh\alpha_o+\sinh\alpha_o$ such that whenever $r\in\R^{1,1}$ satisfies $\tau(p,r)>\tilde{c}_o\tau(p,\gamma(0))$ and $\ma_p(\gamma(0),r)\leq \alpha_o$, then $\overline{\gamma(0)}\ll r$. Let us call $c_o:=2\tilde{c}_o\tau(p,\gamma(0))$, for future convenience.

Consider, for $n$ sufficiently large, a point $x_n\in [p,\gamma(t_n)]$ such that $\tau(p,x_n) \geq c_o/2$. Then $\overline{\gamma(0)}\ll\bar{x}_n$ by the previous reasonings, where $\bar{x}_n$ is the comparison point for $x_n$ in $[\bar{p},\overline{\gamma(t_n)}]$ (see \autoref{fig:existence-representative-bis}). By the global $\CBA(0)$ condition, we get $\tau(\gamma(0),x_n)\geq \tau(\overline{\gamma(0)},\bar{x}_n)>0$, i.e., $\gamma(0)\ll x_n$. In particular, we have
\[
\tau(p,{\sigma}_n(c_o))=c_o>c_o/2,
\]
and therefore $\gamma(0)\ll {\sigma}_n(c_o)$. Analogously to \eqref{eq:gamma_p in the past of gamma}, we obtain (see again \autoref{subfig:triangle-existencerep} with the points changed accordingly):
\[
\tau(\gamma(s),\sigma_n(s+c_o)) \geq \tau(\gamma(0),\sigma_n(c_o)),
\]
which implies 
\[
\tau(\gamma(s),\gamma_p(s+c_o)) \geq \tau(\gamma(0),\gamma_p(c_o))
\]
by letting $n\to \infty$.
Finally, observe that
\begin{align*}
\tau(\gamma(0),{\sigma}_n(c_o))^2 &= c_o^2 + \tau(p,\gamma(0))^2 - 2\tau(p,\gamma(0))c_o\cosh\widetilde{\ma}_p(\gamma(0),\sigma_n(c_o))\\
&\geq c_o^2 + \tau(p,\gamma(0))^2 - 2\tau(p,\gamma(0))c_o\cosh\widetilde{\ma}_p(\gamma(0),\gamma(t_n))\\
&\geq  c_o^2 + \tau(p,\gamma(0))^2 - 2\tau(p,\gamma(0))c_o\cosh \alpha_o =: \lambda_o,
\end{align*}
where $\lambda_o>0$ by substituting $c_o=2(\cosh\alpha_o+\sinh\alpha_o)\tau(p,\gamma(0))$.

In turn, this implies 
\[
\tau(\gamma(s),\gamma_p(s+c_o))\geq \tau(\gamma(0),{\gamma}_p(c_o))  \geq \sqrt{\lambda_o}>0,
\]
and therefore, $\gamma(s)\ll \gamma_p(s+c_o)$, for all $s\in[0,\infty)$.\qedhere

\begin{figure}
\def \globalscale {1.000000}
\begin{tikzpicture}[y=1.1cm, x=1.1cm, yscale=0.9,xscale=0.9, every node/.append style={scale=\globalscale}, inner sep=0pt, outer sep=0pt]
  \path[draw=black,line cap=butt,line join=miter,line width=0.02cm] (5.0, 26.7) -- (10.0, 21.7);
  \path[draw=black,line cap=butt,line join=miter,line width=0.02cm] (10.0, 21.7) -- (15.0, 26.7);
  \path[draw=black,line cap=butt,line join=miter,line width=0.02cm] (10.0, 21.7) -- (10.0, 23.7);
  \path[draw=black,line cap=butt,line join=miter,line width=0.02cm,miter limit=4.0,dash pattern=on 0.1cm off 0.1cm] (10.0, 21.7) -- (7.1, 26.7);
  \path[draw=black,line cap=butt,line join=miter,line width=0.02cm] (10.0, 23.7) -- (8.4, 26.3) -- (10.0, 21.7);
  \path[draw=black,line cap=butt,line join=miter,line width=0.02cm,miter limit=4.0,dash pattern=on 0.1cm off 0.1cm] (10.0, 21.7) -- (12.9, 26.7);
  \path[draw=black,line cap=butt,line join=miter,line width=0.02cm] (10.0, 23.7) -- (8.9, 24.8);
  \path[draw=black,line width=0.02cm,xscale=-1.0,yscale=1.0] (-10.6, 22.7)arc(240.4:270.0:1.2 and -1.2);
  \path[fill=black,line width=0.02cm] (10.0, 21.7) circle (2pt);
  \path[fill=black,line width=0.02cm] (10.0, 23.7) circle (2pt);
  \path[fill=black,line width=0.02cm] (8.4, 26.3) circle (2pt);
  \path[fill=black,line width=0.02cm] (8.925, 24.8) circle (2pt);
  \node[line width=0.0cm,anchor=south west] (text10651) at (9.8, 21.2){$\bar{p}$};
  \node[line width=0.0cm,anchor=south west] (text13019) at (10.1, 23.9  ){$\overline{\gamma(0)}$};
  \node[line width=0.0cm,anchor=south west] (text16653) at (8.5, 26.4){$\overline{\gamma(t_n)}$};
  \node[line width=0.0cm,anchor=south west] (text28651) at (8.45, 24.8){$\bar{x}$};
  \node[line width=0.0cm,anchor=south west] (text43450) at (10.2, 22.95){$\alpha_0$};
\end{tikzpicture}
    \caption{Situation described in the proof of \autoref{pop:existence of pointwise rep}.}
    \label{fig:existence-representative-bis}
\end{figure}
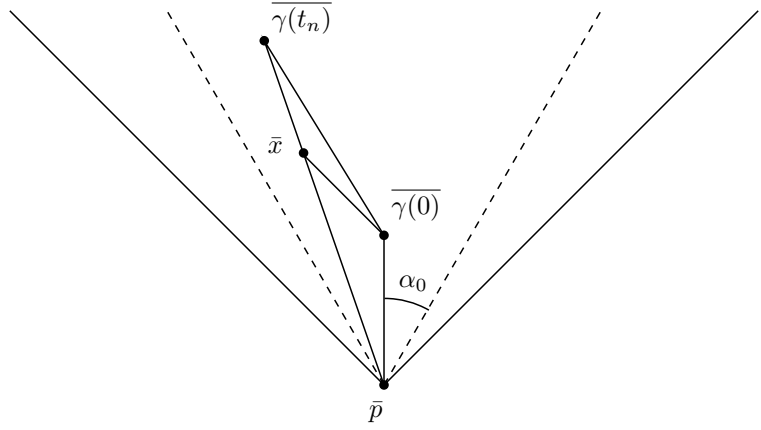
\end{proof}

\begin{lem}\label{lem:angles between asymptotic rays}
Let $Y$ be a Lorentzian pre-length space satisfying the hypotheses of Propositions~\ref{pop:uniqueness of pointwise rep} and \ref{pop:existence of pointwise rep}. Let $\gamma,\gamma'\colon [0,\infty)\to Y$ be future-directed timelike rays such that $\gamma(0)\ll \gamma'(0)$ and $\gamma\sim\gamma'$. Then (see \autoref{subfig:comparison-angles-asymptotic}):
\[
\ma_{\gamma(0)}(\gamma,[\gamma(0),\gamma'(0)]) \leq \ma_{\gamma'(0)}(\gamma',[\gamma'(0),\gamma(0)]).
\]
\end{lem}
\begin{proof}
Let $t_n\geq 0$ be such that $t_n\to\infty$ as $n\to\infty$. By Propositions~\ref{pop:uniqueness of pointwise rep} and \ref{pop:existence of pointwise rep}, $[\gamma(0),\gamma'(t_n)]\to \gamma$ pointwise as $n\to \infty$. Therefore,
\begin{align}
\ma_{\gamma(0)}(\gamma,[\gamma(0),\gamma'(0)]) &= \lim_{n\to\infty} \ma_{\gamma(0)}(\gamma'(t_n),\gamma'(0)), \tag*{by \autoref{thm:continuity of angles},}\\
&\leq \lim_{n\to\infty} \widetilde\ma_{\gamma(0)}(\gamma'(t_n),\gamma'(0)), \tag*{by \autoref{thm:monotonicity of comparison angles},}\\
&\leq \lim_{n\to\infty} \widetilde\ma_{\gamma'(0)}(\gamma'(t_n),\gamma(0)), \tag*{by \autoref{lem:sum of angles of triangle},}\\
&\leq \ma_{\gamma'(0)}(\gamma',[\gamma'(0),\gamma(0)]). \tag*{by \autoref{thm:monotonicity of comparison angles}.}
\end{align}
\end{proof}

\begin{lem}\label{lem:points_seen_from_far_seem_close}
    Let $Y$ be a \LpLS satisfying the hypotheses of Propositions~\ref{pop:uniqueness of pointwise rep} and \ref{pop:existence of pointwise rep}. Let $\xi$ be a future ideal point and let $p,q\in I^-(\xi)$ be such that $p\ll q$. Consider the representative $\gamma$ of $\xi$ at $p$. Then (see \autoref{subfig:angle-at-infinity-is-0}):
    \[
    \ma_{\gamma(t)}(p,q)\xrightarrow{t\to\infty} 0.
    \] 
\end{lem}
\begin{proof}
    Consider $t>0$ such that $q\ll \gamma(t)$, and the timelike triangle $\triangle pq\gamma(t)$. The curvature bound from above ensures (see ~\autoref{thm:monotonicity of comparison angles}) that $\ma_{\gamma(t)}(p,q)\leq  \smash{\widetilde{\ma}_{\gamma(t)}(p,q)}$. Now, for $s>0$ we have
    \[
\begin{aligned}
    \widetilde{\ma}_{\gamma_{t+s}} (p,q)&=
    \cosh^{-1} \left(
    \frac{\tau(p,\gamma_{t+s})^2 + \tau(q,\gamma_{t+s})^2 -\tau(p,q)^2}{2\tau(p,\gamma_{t+s})\tau(q,\gamma_{t+s})}
    \right) \\
    &\leq\cosh^{-1} \left(
    \frac{2(t+s)^2}{2(t+s)(\tau(q,\gamma_t)+s)}
    \right) =\cosh^{-1} \left(
    \frac{t+s}{\tau(q,\gamma_t)+s}
    \right) 
    \xrightarrow{s\to\infty} 0,
\end{aligned}
    \]
    and the previous inequality gives the result.
\end{proof}

\begin{figure}
\centering
\begin{subfigure}{0.45\textwidth}
    \centering
    \tikzset{every picture/.style={line width=0.75pt}}      
    \begin{tikzpicture}[x=0.75pt,y=0.75pt,yscale=-0.8,xscale=0.8]
    \draw [color={rgb, 255:red, 0; green, 0; blue, 0 }  ,draw opacity=1 ][line width=1]  (351.8,260.6) .. controls (354.5,219.75) and (360.5,193.25) .. (370.2,165) ;
    \draw [color={rgb, 255:red, 0; green, 0; blue, 0 }  ,draw opacity=1 ][line width=1]    (204.2,38.6) .. controls (246.6,85.8) and (334.2,210.2) .. (351.8,260.6) ;
    \draw  [fill={rgb, 255:red, 0; green, 0; blue, 0 }  ,fill opacity=1 ] (368.1,165) .. controls (368.1,163.84) and (369.04,162.9) .. (370.2,162.9) .. controls (371.36,162.9) and (372.3,163.84) .. (372.3,165) .. controls (372.3,166.16) and (371.36,167.1) .. (370.2,167.1) .. controls (369.04,167.1) and (368.1,166.16) .. (368.1,165) -- cycle ;
    \draw [line width=1]    (277.8,10.6) .. controls (314.6,44.6) and (352.6,114.6) .. (370.2,165) ;
    \draw  [fill={rgb, 255:red, 0; green, 0; blue, 0 }  ,fill opacity=1 ] (349.7,260.6) .. controls (349.7,259.44) and (350.64,258.5) .. (351.8,258.5) .. controls (352.96,258.5) and (353.9,259.44) .. (353.9,260.6) .. controls (353.9,261.76) and (352.96,262.7) .. (351.8,262.7) .. controls (350.64,262.7) and (349.7,261.76) .. (349.7,260.6) -- cycle ;
    \draw  [draw opacity=0] (334.99,225.2) .. controls (340.1,222.89) and (345.8,221.6) .. (351.8,221.6) .. controls (353.03,221.6) and (354.25,221.65) .. (355.45,221.76) -- (351.8,260.6) -- cycle ; \draw   (334.99,225.2) .. controls (340.1,222.89) and (345.8,221.6) .. (351.8,221.6) .. controls (353.03,221.6) and (354.25,221.65) .. (355.45,221.76) ;  
    \draw  [draw opacity=0] (363.65,187.01) .. controls (356.15,183.78) and (350.79,175.14) .. (350.79,165) .. controls (350.79,155.57) and (355.43,147.44) .. (362.11,143.75) -- (370.2,165) -- cycle ; \draw   (363.65,187.01) .. controls (356.15,183.78) and (350.79,175.14) .. (350.79,165) .. controls (350.79,155.57) and (355.43,147.44) .. (362.11,143.75) ;  
    \draw (359.5,254.25) node [anchor=north west][inner sep=0.75pt]  [font=\small]  {$\gamma ( 0)$};
    \draw (337,65.05) node [anchor=north west][inner sep=0.75pt]  [font=\normalsize]  {$\gamma '$};
    \draw (375.25,165.7) node [anchor=north west][inner sep=0.75pt]  [font=\small]  {$\gamma '( 0)$};
    \draw (254.5,127.15) node [anchor=north west][inner sep=0.75pt]  [font=\normalsize]  {$\gamma $};
    \end{tikzpicture}
    \caption{}
    \label{subfig:comparison-angles-asymptotic}
\end{subfigure}\hspace{0.5cm}
\begin{subfigure}{0.45\textwidth}
    \centering
    \tikzset{every picture/.style={line width=0.75pt}}       
    \begin{tikzpicture}[x=0.75pt,y=0.75pt,yscale=-1,xscale=1]
    
    \draw [color={rgb, 255:red, 0; green, 0; blue, 0 }  ,draw opacity=1 ]   (351.8,260.6) .. controls (354.5,219.75) and (357.97,205.92) .. (367.67,177.67) ;
    \draw  [fill={rgb, 255:red, 0; green, 0; blue, 0 }  ,fill opacity=1 ] (349.7,260.6) .. controls (349.7,259.44) and (350.64,258.5) .. (351.8,258.5) .. controls (352.96,258.5) and (353.9,259.44) .. (353.9,260.6) .. controls (353.9,261.76) and (352.96,262.7) .. (351.8,262.7) .. controls (350.64,262.7) and (349.7,261.76) .. (349.7,260.6) -- cycle ;
    \draw [color={rgb, 255:red, 128; green, 128; blue, 128 }  ,draw opacity=1 ]   (274.33,130) .. controls (321.67,164.33) and (334.67,168) .. (367.67,177.67) ;
    \draw [color={rgb, 255:red, 128; green, 128; blue, 128 }  ,draw opacity=1 ]   (243.5,87.75) .. controls (287.25,128.92) and (323,155.25) .. (367.67,177.67) ;
    \draw [color={rgb, 255:red, 128; green, 128; blue, 128 }  ,draw opacity=1 ]   (210,44.75) .. controls (254.25,90.25) and (337,160.33) .. (367.67,177.67) ;
    \draw  [fill={rgb, 255:red, 0; green, 0; blue, 0 }  ,fill opacity=1 ] (365.57,177.67) .. controls (365.57,176.51) and (366.51,175.57) .. (367.67,175.57) .. controls (368.83,175.57) and (369.77,176.51) .. (369.77,177.67) .. controls (369.77,178.83) and (368.83,179.77) .. (367.67,179.77) .. controls (366.51,179.77) and (365.57,178.83) .. (365.57,177.67) -- cycle ;
    \draw [color={rgb, 255:red, 0; green, 0; blue, 0 }  ,draw opacity=1 ][line width=1]    (204.2,38.6) .. controls (246.6,85.8) and (334.2,210.2) .. (351.8,260.6) ;
    \draw  [draw opacity=0] (236.63,70.61) .. controls (235.36,71.8) and (234,72.91) .. (232.56,73.93) -- (210,44.75) -- cycle ; \draw  [color={rgb, 255:red, 128; green, 128; blue, 128 }  ,draw opacity=1 ] (236.63,70.61) .. controls (235.36,71.8) and (234,72.91) .. (232.56,73.93) ;  
    \draw  [draw opacity=0] (271.09,112.68) .. controls (269.55,114.23) and (267.86,115.65) .. (266.06,116.93) -- (243.5,87.75) -- cycle ; \draw  [color={rgb, 255:red, 128; green, 128; blue, 128 }  ,draw opacity=1 ] (271.09,112.68) .. controls (269.55,114.23) and (267.86,115.65) .. (266.06,116.93) ;  
    \draw  [draw opacity=0] (304.88,151.56) .. controls (302.28,154.91) and (299.09,157.82) .. (295.45,160.14) -- (274.33,130) -- cycle ; \draw  [color={rgb, 255:red, 128; green, 128; blue, 128 }  ,draw opacity=1 ] (304.88,151.56) .. controls (302.28,154.91) and (299.09,157.82) .. (295.45,160.14) ; 
    \draw  [fill={rgb, 255:red, 0; green, 0; blue, 0 }  ,fill opacity=1 ] (272.23,130) .. controls (272.23,128.84) and (273.17,127.9) .. (274.33,127.9) .. controls (275.49,127.9) and (276.43,128.84) .. (276.43,130) .. controls (276.43,131.16) and (275.49,132.1) .. (274.33,132.1) .. controls (273.17,132.1) and (272.23,131.16) .. (272.23,130) -- cycle ; 
    \draw  [fill={rgb, 255:red, 0; green, 0; blue, 0 }  ,fill opacity=1 ] (241.4,87.75) .. controls (241.4,86.59) and (242.34,85.65) .. (243.5,85.65) .. controls (244.66,85.65) and (245.6,86.59) .. (245.6,87.75) .. controls (245.6,88.91) and (244.66,89.85) .. (243.5,89.85) .. controls (242.34,89.85) and (241.4,88.91) .. (241.4,87.75) -- cycle ; 
    \draw  [fill={rgb, 255:red, 0; green, 0; blue, 0 }  ,fill opacity=1 ] (207.9,44.75) .. controls (207.9,43.59) and (208.84,42.65) .. (210,42.65) .. controls (211.16,42.65) and (212.1,43.59) .. (212.1,44.75) .. controls (212.1,45.91) and (211.16,46.85) .. (210,46.85) .. controls (208.84,46.85) and (207.9,45.91) .. (207.9,44.75) -- cycle ;
    
    \draw (357.75,251.75) node [anchor=north west][inner sep=0.75pt]  [font=\small]  {$p$};
    \draw (296.5,179.4) node [anchor=north west][inner sep=0.75pt]  [font=\normalsize]  {$\gamma $};
    \draw (372.17,174.53) node [anchor=north west][inner sep=0.75pt]  [font=\small]  {$q$};
    \end{tikzpicture}
    \caption{}
    \label{subfig:angle-at-infinity-is-0}
\end{subfigure}
    \caption{Situations described in Lemmata~\ref{lem:angles between asymptotic rays} and \ref{lem:points_seen_from_far_seem_close}.}
    \label{fig:auxiliary-angles}
\end{figure}
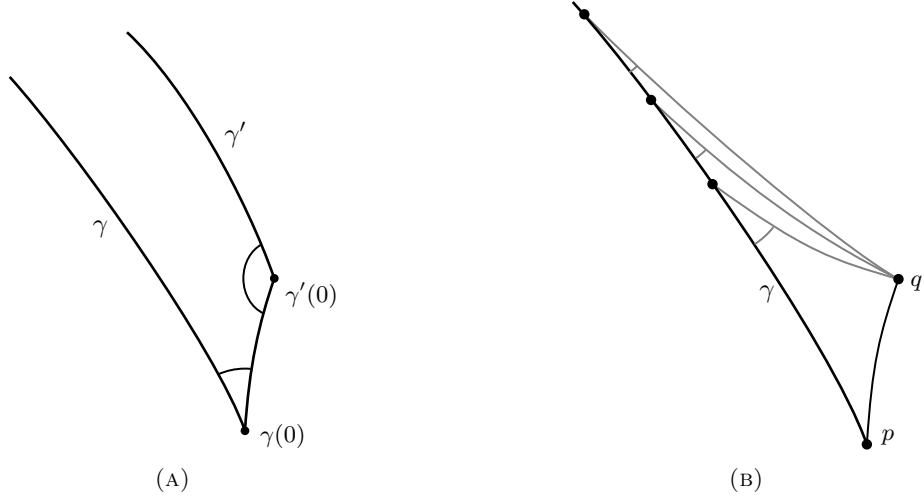

\section{Cone topology on the future timelike ideal completion}
In this section we introduce a natural topology on $Y^*:=Y\cup \bd^+Y$, namely the \textit{cone topology}, in the same spirit as the one defined in the metric case \cite{eberlein_visibility_1973}. Of course, an analogous topology can be defined for $Y_*:=Y\cup \bd^-Y$.

Let $Y$ be a \LpLS satisfying $\CBA(0)$ globally. Let $x\in Y$ and $\xi,\xi'\in \bd^+Y$. Then the angle
\begin{equation}
\ma_x(\xi,\xi') := \ma_x(\alpha,\alpha')  = \limsup_{s,t\to 0^+} \cosh^{-1}\left(\frac{s^2+t^2-\tau( \alpha(s),\alpha'(t))^2}{2st}\right)
\end{equation}
is well-defined, by Theorems~\ref{thm:monotonicity implies existence of angles} and \ref{pop:uniqueness of pointwise rep}, whenever there exist $\alpha,\alpha'$ future-directed timelike geodesic rays such that $\alpha(0)=\alpha'(0)=x$ and $\alpha(\infty) = \xi$, $\alpha'(\infty)=\xi'$. Otherwise, we declare that $\ma_x(\xi,\xi')$ does not exist. Note that if $Y$ is proper, strongly causal, locally causally closed and regular, the representatives $\alpha=\xi_x$ and $\alpha'=\xi'_x$ exist whenever $x\in I^-(\xi)\cap I^-(\xi')$, by \autoref{pop:existence of pointwise rep}.

\begin{dfn}\label{def:basic nbhd cone topology}
    Let $Y$ be a proper, strongly causal, locally causally closed, regular Lorentzian pre-length space satisfying $\CBA(0)$ globally. Let $\xi\in \bd^+Y$, $p\in I^-(\xi)$ and $R,\varepsilon>0$. Define the sets
    \[
    \begin{aligned}
    V_{p,R,\varepsilon}(\xi)&=\left\{
    x\in Y^* \relmiddle| \ma_p\bigl(\xi_p,[p,x]\bigr)<\varepsilon \text{ and } \tau(p,x)>R
    \right\},
    \end{aligned}
    \]
    where $[p,x]$ denotes the unique geodesic from $p$ to $x$ (see \autoref{thm:uniqueness of geodesics in cba}). We call such a set a \textit{truncated cone at $p$ with axis $\xi$}.
\end{dfn}

\begin{lem}\label{lem:cones_open}
    Let $Y$ be a proper, strongly causal, locally causally closed, regular Lorentzian pre-length space satisfying $\CBA(0)$ globally. The intersection of a truncated cone with $Y$ is open in the metric topology.
\end{lem}
\begin{proof}
    Consider 
    \[
    V:=V_{p,\varepsilon,R}(\xi)\cap Y=
    \left\{
    x\in Y \relmiddle| \ma_p\bigl(\xi_p,[p,x]\bigr)<\varepsilon \text{ and } \tau(p,x)>R
    \right\},
    \]
    and take a point $x\in V$ and a sequence $x_n$ in $Y$ converging to $x$ with the metric topology. Because the timelike cones are open and $x\in I^+(p)$, there is some $m\in\N$ such that $x_n\in I^+(p)$ for every $n\geq m$. What is more, the continuity of $\tau$ (given by the global $\CBA(0)$ condition) yields $\lim_{n} \tau(p,x_n) = \tau(p,x)>R$, therefore there is some $m'\in\N$ such that $\tau(p,x_n)>R$ for every $n\geq m'$.

    For each $n\geq m'$, take $\alpha_n\colon [0,1]\to Y$ to be the unique future-directed maximal geodesic from $p$ to $x_n$ (see \autoref{thm:uniqueness of geodesics in cba}) of constant $d$-speed. This will in fact be timelike by regularity, and the unit $d$-speed parametrization exists because each $\alpha_n$ is $d$-rectifiable \cite[Proposition~2.19]{kunzinger-saemann2018}. 

    Assume that $L_d(\alpha_n)$ is bounded. Then the curves $\alpha_n$, which are Lipschitz, have uniformly bounded Lipschitz constants (in fact, the Lipschitz constant of all of them is $1$). Thus, as $Y$ is locally causally closed with proper metric $d$, we can use \autoref{thm:limit curve theorem} and obtain a subsequence $\alpha_{n_k}$ which converges uniformly to a Lipschitz curve $\alpha\colon [0,1]\to Y$ which is future-directed causal from $p$ to $x$. Moreover, $\alpha$ is maximizing. Indeed,
    \[
    L_\tau(\alpha)\leq \tau(p,x) = \lim_n \tau(p,x_n) = \lim_n L_\tau(\alpha_n) \leq L_\tau(\alpha),
    \] 
    where we used the continuity of $\tau$, the upper semicontinuity of $L_\tau$ (ensured by strong causality and timelike geodesic connectedness; see \cite[Proposition~3.17]{kunzinger-saemann2018} and \cite[Remark~2.5]{eros-gieger-2025}), and the fact that $\alpha_n$ are maximal geodesics. If, on the other hand, $L_d(\alpha_n)$ diverges, we can similarly deduce local uniform convergence \cite[Theorem~2.23]{beran-ohanyan-rott-solis2023}.
    
    By \autoref{thm:continuity of angles}, the upper curvature bound ensures that the map $\ma_p(\xi_p,\cdot)$ is continuous in the sense that for $\alpha_n,\alpha$ future-directed geodesics (non necessarily parametrized by $\tau$-length) starting at $p$ and such that $\alpha_n\to \alpha$ pointwise, one has
    \[
    \lim_n \ma_p(\xi_p,\alpha_n)=\ma_p(\xi_p,\alpha).
    \]
    Therefore, after reparametrization, we have
    \[
    \lim_n \ma_p\bigl(\xi_p,[p,x_n]\bigr)=\ma_p\bigl(\xi_p,[p,x]\bigr)<\varepsilon.
    \]
    Thus, for $n$ large enough, $\ma_p\bigl(\xi_p,[p,x_n]\bigr)<\varepsilon$. In other words, there is some $\tilde{m}\in \N$ such that $x_n\in V$, for all $n\geq \tilde{m}$. From the generality of the sequence $x_n$ and the point $x\in V$ we deduce that $V$ is open in the metric topology.
\end{proof}

Now, we will show that the truncated cones induce some topology in $Y^*$ such that the set of cones with axis $\xi \in \bd^+Y$ is a neighborhood basis of $\xi$. To that end, consider first the following topological result \cite[Lemma~1.3]{eberlein_visibility_1973}:
\begin{lem}\label{lem:topological}
    Let $(X,\newtau)$ be a topological space, $S$ a set and $X^*:=X\sqcup S$. For each $y\in S$, let $\mathcal{N}(y)$ be a family of subsets of $X^*$ such that: 
    \begin{enumerate}
        \item $\mathcal{N}(y)\neq \varnothing$, for each $y\in S$
        \item $V\in \mathcal{N}(y)\implies y\in V$ and $V\cap X\in \newtau$,
        \item if $V_1\in \mathcal{N}(y_1)$, $V_2\in \mathcal{N}(y_2)$ and $y_3\in V_1\cap V_2 \cap S$, then there exists $V_3\in\mathcal{N}(y_3)$ such that $V_3\subset V_1\cap V_2$.
    \end{enumerate} 
    Then there is a unique topology $\newtau^*$ on $X^*$ such that its restriction to $X$ is $\newtau$ and such that for each $y\in S$, $\mathcal{N}(y)$ is a neighborhood basis of $y$ for $\newtau^*$.
\end{lem}

With this result, we can prove that there is a unique topology on $Y^*$ satisfying that its restriction to $Y$ is the metric topology and such that the truncated cones with a fixed axis $\xi$ form a neighborhood basis of $\xi$.

\begin{pop}\label{pop:cone_top}
    Let $Y$ be a proper, strongly causal, locally causally closed, regular Lorentzian pre-length space satisfying $\CBA(0)$ globally. For each $\xi\in\bd^+Y$, set 
    \[
    \mathcal{N}(\xi)=\bigl\{V_{p,\varepsilon,R}(\xi)\mid p\in I^-(\xi); \: \varepsilon,R>0\bigr\}.
    \]
    Then there is a unique topology on $Y^*$ whose restriction to $Y$ is the metric topology and such that for every $\xi\in\bd^+Y$, $\mathcal{N}(\xi)$ is a neighborhood basis of $\xi$.
\end{pop}
\begin{proof}
    Property (1) and the first part of property (2) in \autoref{lem:topological} are clearly satisfied. The second part of property (2) is precisely \autoref{lem:cones_open}.

    Let us check that property \textit{(3)} also holds. To that end consider boundary points $\xi_1,\xi_2\in\bd^+Y$ and truncated cones $V_j:=V_{p_j,\varepsilon_j,R_j}(\xi_j)$, for $j\in\{1,2\}$. Assume that the intersection of these is non-empty and, moreover, that there is a boundary point $\xi_3\in V_1\cap V_2\cap \bd^+Y$. We will show that there is a truncated cone $V_3:=V_{p_3,\varepsilon_3,R_3}(\xi_3)$ contained in the intersection of $V_1$ and $V_2$.
    
    For $i\in\{1,2,3\}$, $j\in\{1,2\}$, call $\xi_{i,p_j}$ the future-directed timelike geodesic ray starting at $p_j$ such that $\xi_{i,p_j}(\infty)=\xi_i$. As $\xi_3\in V_{p_j,\varepsilon_j,R_j}(\xi_j)$, we have $\theta_j:=\ma_{p_j}\bigl(\xi_{j,p_j},\xi_{3,p_j}\bigr)<\varepsilon_j$. The symmetry of this condition gives, conversely, $\xi_j\in V_{p_j,\varepsilon_j,R_j}(\xi_3)$. For clarity, we depict the situation in Figures~\ref{subfig:original_cones} and \ref{subfig:final_cones}. We will work with these two asymptotic rays $\xi_{3,p_j}$ starting at different points. Define $\smash{\varepsilon_j':=\varepsilon_j-\theta_j>0}$ and consider the thinner cones $\smash{\widetilde{V}_j:=V_{p_j,\varepsilon_j',R_j}(\xi_3)}$.

    \begin{figure}[ht]
    \begin{subfigure}{0.45\textwidth}
        \centering
        \tikzset{every picture/.style={line width=0.75pt}}       
        
        \begin{tikzpicture}[x=0.75pt,y=0.75pt,yscale=-0.9,xscale=0.9]

        \draw [color={rgb, 255:red, 208; green, 2; blue, 27 }  ,draw opacity=0.38 ]   (318.33,46.33) .. controls (328,92.33) and (321.67,153.33) .. (306.33,198) ;
        \draw [color={rgb, 255:red, 208; green, 2; blue, 27 }  ,draw opacity=0.38 ]   (427,103) .. controls (382.33,117) and (336.67,150.67) .. (306.33,198) ;
        \draw [color={rgb, 255:red, 74; green, 144; blue, 226 }  ,draw opacity=0.53 ]   (230,42.67) .. controls (253.67,83) and (257,140) .. (246.67,184.67) ;
        \draw [color={rgb, 255:red, 74; green, 144; blue, 226 }  ,draw opacity=0.53 ]   (349,62) .. controls (309,82.67) and (277,137.33) .. (246.67,184.67) ;
        \draw [color={rgb, 255:red, 208; green, 2; blue, 27 }  ,draw opacity=1 ]   (401.67,37.67) -- (306.33,198) ;
        \draw [color={rgb, 255:red, 74; green, 144; blue, 226 }  ,draw opacity=1 ]   (302,18.33) -- (246.67,184.67) ;
        \draw    (381,31) -- (306.33,198) ;
        \draw  [fill={rgb, 255:red, 0; green, 0; blue, 0 }  ,fill opacity=1 ] (245.67,184.67) .. controls (245.67,184.11) and (246.11,183.67) .. (246.67,183.67) .. controls (247.22,183.67) and (247.67,184.11) .. (247.67,184.67) .. controls (247.67,185.22) and (247.22,185.67) .. (246.67,185.67) .. controls (246.11,185.67) and (245.67,185.22) .. (245.67,184.67) -- cycle ;
        \draw  [fill={rgb, 255:red, 0; green, 0; blue, 0 }  ,fill opacity=1 ] (305.33,198) .. controls (305.33,197.45) and (305.78,197) .. (306.33,197) .. controls (306.89,197) and (307.33,197.45) .. (307.33,198) .. controls (307.33,198.55) and (306.89,199) .. (306.33,199) .. controls (305.78,199) and (305.33,198.55) .. (305.33,198) -- cycle ;
        \draw    (321.33,17.67) -- (246.67,184.67) ;
        
        \draw (235.08,185.23) node [anchor=north west][inner sep=0.75pt]  [font=\footnotesize]  {$p_{1}$};
        \draw (296.33,199.48) node [anchor=north west][inner sep=0.75pt]  [font=\footnotesize]  {$p_{2}$};
        \draw (269,13.4) node [anchor=north west][inner sep=0.75pt]  [font=\footnotesize]  {$\xi _{1,p_{1}}$};
        \draw (403,37.4) node [anchor=north west][inner sep=0.75pt]  [font=\footnotesize]  {$\xi _{2,p_{2}}$};
        \draw (318,22.73) node [anchor=north west][inner sep=0.75pt]  [font=\footnotesize]  {$\xi _{3,p_{1}}$};
        \draw (351,21.4) node [anchor=north west][inner sep=0.75pt]  [font=\footnotesize]  {$\xi _{3,p_{2}}$};
        \end{tikzpicture}
        \caption{Original cones $V_{p_j,\varepsilon_j}(\xi_j)$, containing both the ideal point $\xi_{3}$.}
        \label{subfig:original_cones}
    \end{subfigure}\hspace{0.5cm}
    \begin{subfigure}{0.45\textwidth}
        \centering
        \tikzset{every picture/.style={line width=0.75pt}}          
        \begin{tikzpicture}[x=0.75pt,y=0.75pt,yscale=-0.9,xscale=0.9]
 
        \draw [color={rgb, 255:red, 208; green, 2; blue, 27 }  ,draw opacity=1 ]   (399.33,31.17) -- (304,191.5) ;
        \draw [color={rgb, 255:red, 74; green, 144; blue, 226 }  ,draw opacity=1 ]   (299.67,11.83) -- (244.33,178.17) ;
        \draw [color={rgb, 255:red, 128; green, 128; blue, 128 }  ,draw opacity=1 ]   (248.33,37.5) .. controls (261.33,60.17) and (259.67,133.5) .. (244.33,178.17) ;
        \draw [color={rgb, 255:red, 128; green, 128; blue, 128 }  ,draw opacity=1 ]   (288.33,28.83) .. controls (310.33,64.17) and (319.33,146.83) .. (304,191.5) ;
        \draw [color={rgb, 255:red, 128; green, 128; blue, 128 }  ,draw opacity=1 ]   (374.33,66.17) .. controls (332,82.5) and (274.67,130.83) .. (244.33,178.17) ;
        \draw [color={rgb, 255:red, 128; green, 128; blue, 128 }  ,draw opacity=1 ]   (407.33,86.17) .. controls (376.33,97.17) and (334.33,144.17) .. (304,191.5) ;
        \draw    (378.67,24.5) -- (304,191.5) ;
        \draw  [fill={rgb, 255:red, 0; green, 0; blue, 0 }  ,fill opacity=1 ] (243.33,178.17) .. controls (243.33,177.61) and (243.78,177.17) .. (244.33,177.17) .. controls (244.89,177.17) and (245.33,177.61) .. (245.33,178.17) .. controls (245.33,178.72) and (244.89,179.17) .. (244.33,179.17) .. controls (243.78,179.17) and (243.33,178.72) .. (243.33,178.17) -- cycle ;
        \draw  [fill={rgb, 255:red, 0; green, 0; blue, 0 }  ,fill opacity=1 ] (303,191.5) .. controls (303,190.95) and (303.45,190.5) .. (304,190.5) .. controls (304.55,190.5) and (305,190.95) .. (305,191.5) .. controls (305,192.05) and (304.55,192.5) .. (304,192.5) .. controls (303.45,192.5) and (303,192.05) .. (303,191.5) -- cycle ;
        \draw    (319,11.17) -- (244.33,178.17) ;
        
        \draw (232.75,178.73) node [anchor=north west][inner sep=0.75pt]  [font=\footnotesize]  {$p_{1}$};
        \draw (294,192.98) node [anchor=north west][inner sep=0.75pt]  [font=\footnotesize]  {$p_{2}$};
        \draw (266.67,6.9) node [anchor=north west][inner sep=0.75pt]  [font=\footnotesize]  {$\xi _{1,p_{1}}$};
        \draw (400.67,30.9) node [anchor=north west][inner sep=0.75pt]  [font=\footnotesize]  {$\xi _{2,p_{2}}$};
        \draw (315.67,16.23) node [anchor=north west][inner sep=0.75pt]  [font=\footnotesize]  {$\xi _{3,p_{1}}$};
        \draw (348.67,14.9) node [anchor=north west][inner sep=0.75pt]  [font=\footnotesize]  {$\xi _{3,p_{2}}$};
        \end{tikzpicture}
        \caption{New cones $V_{p_j,\varepsilon_j}(\xi_3)$, containing resp.\ the ideal point $\xi_{j}$.}
        \label{subfig:final_cones}
    \end{subfigure}
        \caption{Situations in the proof of \autoref{pop:cone_top}. The cones are depicted without truncation for clarity.}
    \end{figure}
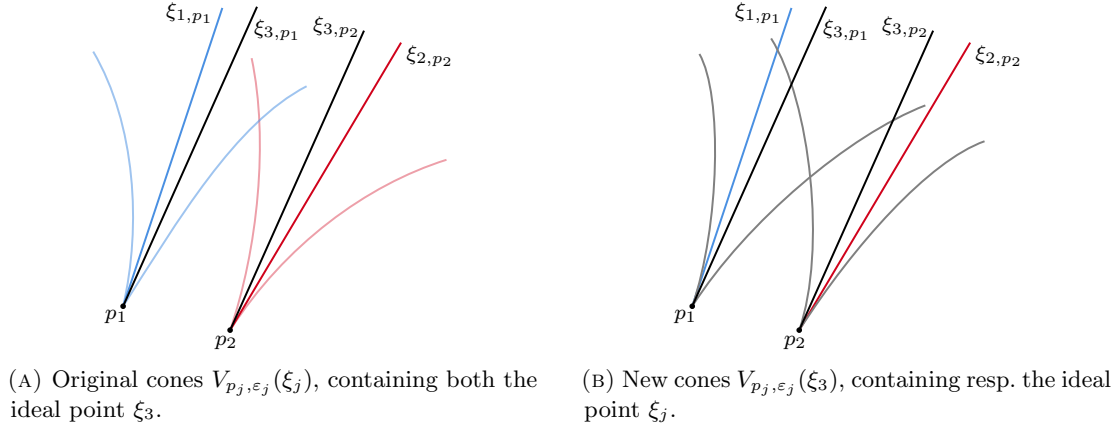

    \begin{claim}
    There exists a value $T>0$ such that $\xi_{3,p_j}(t)\in V_{p_k,\varepsilon_k',R_k}(\xi_3)$, for all $t\geq T$ and $j,k\in\{1,2\}$. In other words, at some point each of these asymptotic rays  enters the (thinner) truncated cone of the other one and remains within thereafter.
    \end{claim}
    \begin{claimproof}
    Indeed, following a similar argument as in the proof of \autoref{pop:uniqueness of pointwise rep} one has, for $j,k\in\{1,2\}$,
    \[
    0\leq
    \ma_{p_k}\bigl(
    [p_k,\xi_{3,p_j}(t)],\xi_{3,p_k}
    \bigr)
    \leq \widetilde{\ma}_{p_k}\bigl(
    \xi_{3,p_j}(t),\xi_{3,p_k}(t+c)
    \bigr)\xrightarrow{t\to\infty} 0,
    \]
    where $c$ is the value given by \autoref{def:asymptotic}. As a consequence, there is some value $T'>0$ for which for all $t>T'$ and $k,j\in\{1,2\}$,
    \[
    \ma_{p_k}\bigl([p_k,\xi_{3,p_j}(t)],\xi_{3,p_k}\bigr)<\varepsilon_k'.
    \]
    To achieve the inequality $\tau(p_k,\xi_{3,p_j}(t))>R_k$, it suffices to take $t>R_k+c$. In conclusion, taking $T>\max(T',R_1+c,R_2+c)$ one obtains $\xi_{3,p_j}(t)\in V_{p_k,\varepsilon_k,R_k}(\xi_3)$, for all $t\geq T$.
    \end{claimproof}

    This claim, together with \autoref{lem:points_seen_from_far_seem_close}, implies that one can take a parameter value $S$ sufficiently large such that $S\geq T$ and $\ma_{\gamma(t)}\bigl([\gamma(t), p],[\gamma(t),q]\bigr)<\varepsilon_3:=\min\{\varepsilon_1',\frac{\varepsilon_2'}{3}\}$, for every $t\geq S$. Define $p_3=\xi_{3,p_1}(S)$ and take $R_3>0$ to be any positive number.
    
    \begin{claim}
        $V_3:=V_{p_3,\varepsilon_3,R_3}(\xi_3)\subset V_1\cap V_2$.
    \end{claim}
    \begin{claimproof}
        That $V_3\subset \widetilde{V}_1$ is clear from the curvature bound. Indeed, take $q\in V_3$. We have a situation as in \autoref{subfig:cone_inclusion1}, where the angles represent the angle between geodesics. The curvature bound gives the inequalities $a\geq b\geq c$, hence $c=\ma_{p_1}\bigl([p_1,q],\xi_{3,p_1}\bigr)\leq a=
        \ma_{p_3}\bigl([p_3,q],\xi_{3,p_3}\bigr)<\varepsilon_3\leq \varepsilon_1'=\varepsilon_1-\theta_1$. Moreover, $d=\ma_{p_1}(\xi_{3,p_1}, \xi_{1,p_1})= \theta_1$. \autoref{thm:triangle inequality for angles} gives therefore
        \[
        \ma_{p_1}\bigl([p_1,q],\xi_{1,p_1}\bigr)\leq \ma_{p_1}\bigl([p_1,q],\xi_{3,p_1}\bigr) + \ma_{p_1}(\xi_{3,p_1},\xi_{1,p_1})<\varepsilon_1.
        \]

        \begin{figure}
        \centering
        \begin{subfigure}{0.45\textwidth}
        \centering
        \tikzset{every picture/.style={line width=0.75pt}}     
        \begin{tikzpicture}[x=0.75pt,y=0.75pt,yscale=-1,xscale=1]
        \draw  [fill={rgb, 255:red, 0; green, 0; blue, 0 }  ,fill opacity=1 ] (234.17,212.23) .. controls (234.17,211.68) and (234.61,211.23) .. (235.17,211.23) .. controls (235.72,211.23) and (236.17,211.68) .. (236.17,212.23) .. controls (236.17,212.78) and (235.72,213.23) .. (235.17,213.23) .. controls (234.61,213.23) and (234.17,212.78) .. (234.17,212.23) -- cycle ;
        \draw  [fill={rgb, 255:red, 0; green, 0; blue, 0 }  ,fill opacity=1 ] (271.25,110.31) .. controls (271.25,109.76) and (271.7,109.31) .. (272.25,109.31) .. controls (272.8,109.31) and (273.25,109.76) .. (273.25,110.31) .. controls (273.25,110.86) and (272.8,111.31) .. (272.25,111.31) .. controls (271.7,111.31) and (271.25,110.86) .. (271.25,110.31) -- cycle ;
        \draw    (296,45.33) -- (235.17,212.23) ;
        \draw  [fill={rgb, 255:red, 0; green, 0; blue, 0 }  ,fill opacity=1 ] (334.25,51.75) .. controls (334.25,51.2) and (334.7,50.75) .. (335.25,50.75) .. controls (335.8,50.75) and (336.25,51.2) .. (336.25,51.75) .. controls (336.25,52.3) and (335.8,52.75) .. (335.25,52.75) .. controls (334.7,52.75) and (334.25,52.3) .. (334.25,51.75) -- cycle ;
        \draw    (272.25,110.31) .. controls (281.5,97.56) and (316.5,75.69) .. (335.25,50.75) ;
        \draw  [draw opacity=0] (287.88,67.72) .. controls (294.76,70.47) and (300.81,74.97) .. (305.52,80.7) -- (272.25,110.31) -- cycle ; \draw   (287.88,67.72) .. controls (294.76,70.47) and (300.81,74.97) .. (305.52,80.7) ;  
        \draw    (258.67,46.33) -- (235.17,212.23) ;
        \draw    (335.25,51.75) -- (235.17,212.23) ;
        \draw  [draw opacity=0] (282.57,100.05) .. controls (282.85,101.36) and (283,102.72) .. (283,104.13) .. controls (283,113.34) and (276.58,120.94) .. (268.26,122.1) -- (266.06,104.13) -- cycle ; \draw   (282.57,100.05) .. controls (282.85,101.36) and (283,102.72) .. (283,104.13) .. controls (283,113.34) and (276.58,120.94) .. (268.26,122.1) ;  
        \draw  [draw opacity=0] (250.8,169.64) .. controls (253.65,170.78) and (256.36,172.22) .. (258.89,173.92) -- (235.17,212.23) -- cycle ; \draw   (250.8,169.64) .. controls (253.65,170.78) and (256.36,172.22) .. (258.89,173.92) ;  
        \draw  [draw opacity=0] (243.49,156.66) .. controls (247.21,157.17) and (250.82,158.01) .. (254.28,159.15) -- (235.17,212.23) -- cycle ; \draw   (243.49,156.66) .. controls (247.21,157.17) and (250.82,158.01) .. (254.28,159.15) ;  
        
        \draw (223.58,212.8) node [anchor=north west][inner sep=0.75pt]  [font=\footnotesize]  {$p_{1}$};
        \draw (279.58,28.13) node [anchor=north west][inner sep=0.75pt]  [font=\footnotesize]  {$\xi _{3,p_{1}}$};
        \draw (256.42,98.3) node [anchor=north west][inner sep=0.75pt]  [font=\footnotesize]  {$p_{3}$};
        \draw (338.25,39.71) node [anchor=north west][inner sep=0.75pt]  [font=\footnotesize]  {$q$};
        \draw (239.83,28.65) node [anchor=north west][inner sep=0.75pt]  [font=\footnotesize]  {$\xi _{1,p_{1}}$};
        \draw (296.75,60.9) node [anchor=north west][inner sep=0.75pt]  [font=\scriptsize]  {$a$};
        \draw (284.57,103.45) node [anchor=north west][inner sep=0.75pt]  [font=\scriptsize]  {$b$};
        \draw (256.5,155.15) node [anchor=north west][inner sep=0.75pt]  [font=\scriptsize]  {$c$};
        \draw (246.5,143.9) node [anchor=north west][inner sep=0.75pt]  [font=\scriptsize]  {$d$};
        \end{tikzpicture}
        \caption{}
        \label{subfig:cone_inclusion1}
        \end{subfigure}
        \begin{subfigure}{0.45\textwidth}
        \centering
        \tikzset{every picture/.style={line width=0.75pt}}        
        \begin{tikzpicture}[x=0.75pt,y=0.75pt,yscale=-1,xscale=1]
        \draw    (370,90) -- (295.33,257) ;
        \draw  [fill={rgb, 255:red, 0; green, 0; blue, 0 }  ,fill opacity=1 ] (214.17,238.67) .. controls (214.17,238.11) and (214.61,237.67) .. (215.17,237.67) .. controls (215.72,237.67) and (216.17,238.11) .. (216.17,238.67) .. controls (216.17,239.22) and (215.72,239.67) .. (215.17,239.67) .. controls (214.61,239.67) and (214.17,239.22) .. (214.17,238.67) -- cycle ;
        \draw  [fill={rgb, 255:red, 0; green, 0; blue, 0 }  ,fill opacity=1 ] (294.33,257) .. controls (294.33,256.45) and (294.78,256) .. (295.33,256) .. controls (295.89,256) and (296.33,256.45) .. (296.33,257) .. controls (296.33,257.55) and (295.89,258) .. (295.33,258) .. controls (294.78,258) and (294.33,257.55) .. (294.33,257) -- cycle ;
        \draw [color={rgb, 255:red, 0; green, 0; blue, 0 }  ,draw opacity=1 ]   (259.75,138.75) .. controls (278,157.75) and (302,218) .. (295.33,257) ;
        \draw  [fill={rgb, 255:red, 0; green, 0; blue, 0 }  ,fill opacity=1 ] (258.75,138.75) .. controls (258.75,138.2) and (259.2,137.75) .. (259.75,137.75) .. controls (260.3,137.75) and (260.75,138.2) .. (260.75,138.75) .. controls (260.75,139.3) and (260.3,139.75) .. (259.75,139.75) .. controls (259.2,139.75) and (258.75,139.3) .. (258.75,138.75) -- cycle ;
        \draw    (289.83,71.67) -- (215.17,238.67) ;
        \draw  [fill={rgb, 255:red, 0; green, 0; blue, 0 }  ,fill opacity=1 ] (305.5,72.5) .. controls (305.5,71.95) and (305.95,71.5) .. (306.5,71.5) .. controls (307.05,71.5) and (307.5,71.95) .. (307.5,72.5) .. controls (307.5,73.05) and (307.05,73.5) .. (306.5,73.5) .. controls (305.95,73.5) and (305.5,73.05) .. (305.5,72.5) -- cycle ;
        \draw    (295.33,257) .. controls (315.75,188) and (319.5,136) .. (306.5,72.5) ; 
        \draw    (259.75,138.75) .. controls (269,126) and (291.75,98.75) .. (306.5,72.5) ;
        \draw  [draw opacity=0] (295.52,227) .. controls (298.09,227.02) and (300.58,227.35) .. (302.95,227.98) -- (295.33,257) -- cycle ; \draw   (295.52,227) .. controls (298.09,227.02) and (300.58,227.35) .. (302.95,227.98) ;  
        \draw  [draw opacity=0] (304.62,220.27) .. controls (306.82,220.77) and (308.95,221.44) .. (311,222.26) -- (295.33,257) -- cycle ; \draw   (304.62,220.27) .. controls (306.82,220.77) and (308.95,221.44) .. (311,222.26) ;  
        \draw  [draw opacity=0] (289.54,197.82) .. controls (291.45,197.65) and (293.38,197.56) .. (295.33,197.56) .. controls (303.91,197.56) and (312.08,199.26) .. (319.49,202.32) -- (295.33,257) -- cycle ; \draw   (289.54,197.82) .. controls (291.45,197.65) and (293.38,197.56) .. (295.33,197.56) .. controls (303.91,197.56) and (312.08,199.26) .. (319.49,202.32) ;  
        \draw  [draw opacity=0] (278.75,97.68) .. controls (281.48,99.05) and (284.06,100.71) .. (286.43,102.62) -- (259.75,138.75) -- cycle ; \draw   (278.75,97.68) .. controls (281.48,99.05) and (284.06,100.71) .. (286.43,102.62) ;  
        \draw  [draw opacity=0] (272.91,121.56) .. controls (275.88,126.07) and (277.7,132.11) .. (277.7,138.75) .. controls (277.7,146.02) and (275.52,152.57) .. (272.03,157.17) -- (259.75,138.75) -- cycle ; \draw   (272.91,121.56) .. controls (275.88,126.07) and (277.7,132.11) .. (277.7,138.75) .. controls (277.7,146.02) and (275.52,152.57) .. (272.03,157.17) ;  
        \draw  [draw opacity=0] (273.78,108.16) .. controls (282.15,114.27) and (287.78,125.68) .. (287.78,138.75) .. controls (287.78,150.01) and (283.61,160.03) .. (277.1,166.5) -- (259.75,138.75) -- cycle ; \draw   (273.78,108.16) .. controls (282.15,114.27) and (287.78,125.68) .. (287.78,138.75) .. controls (287.78,150.01) and (283.61,160.03) .. (277.1,166.5) ;   
        \draw  [draw opacity=0] (267.73,149.47) .. controls (265.39,150.56) and (262.66,151.19) .. (259.75,151.19) .. controls (258.04,151.19) and (256.4,150.98) .. (254.86,150.58) -- (259.75,138.75) -- cycle ; \draw   (267.73,149.47) .. controls (265.39,150.56) and (262.66,151.19) .. (259.75,151.19) .. controls (258.04,151.19) and (256.4,150.98) .. (254.86,150.58) ;  
        
        \draw (203.58,239.23) node [anchor=north west][inner sep=0.75pt]  [font=\footnotesize]  {$p_{1}$};
        \draw (281.83,256.73) node [anchor=north west][inner sep=0.75pt]  [font=\footnotesize]  {$p_{2}$};
        \draw (209.25,175.73) node [anchor=north west][inner sep=0.75pt]  [font=\footnotesize]  {$\xi _{3,p_{1}}$};
        \draw (345.25,151.65) node [anchor=north west][inner sep=0.75pt]  [font=\footnotesize]  {$\xi _{3,p_{2}}$};
        \draw (243.75,126.4) node [anchor=north west][inner sep=0.75pt]  [font=\footnotesize]  {$p_{3}$};
        \draw (302,58) node [anchor=north west][inner sep=0.75pt]  [font=\footnotesize]  {$q$};
        \draw (295.75,214.15) node [anchor=north west][inner sep=0.75pt]  [font=\scriptsize]  {$a$};
        \draw (307,207.15) node [anchor=north west][inner sep=0.75pt]  [font=\scriptsize]  {$b$};
        \draw (295.75,188) node [anchor=north west][inner sep=0.75pt]  [font=\scriptsize]  {$c$};
        \draw (287.5,148.4) node [anchor=north west][inner sep=0.75pt]  [font=\scriptsize]  {$d$};
        \draw (283,85.65) node [anchor=north west][inner sep=0.75pt]  [font=\scriptsize]  {$f$};
        \draw (257.25,155) node [anchor=north west][inner sep=0.75pt]  [font=\scriptsize]  {$g$};
        \draw (267,132.98) node [anchor=north west][inner sep=0.75pt]  [font=\scriptsize]  {$e$};
        \end{tikzpicture}
        \caption{}
        \label{subfig:cone_inclusion2}
        \end{subfigure}
        \caption{Situations of the last claim of \autoref{pop:cone_top}. Same letters in different images have different meanings.}
        \end{figure}
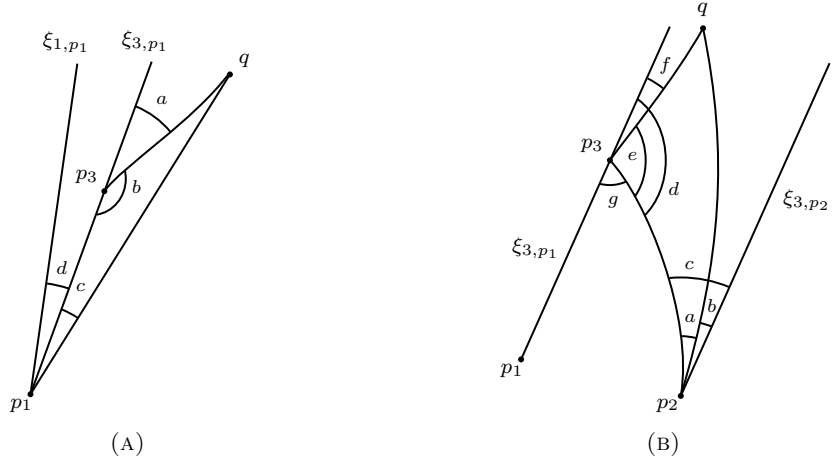

        For the other inclusion, consider again some $q\in V_3$. We have the situation of \autoref{subfig:cone_inclusion2}, where again angles represent angles between corresponding geodesics. From \autoref{thm:triangle inequality for angles}, we have $b\leq a+c$ and $e\leq d+f$. Also, from the curvature bounded above we deduce that $c\leq d$, $a\leq e$ and $d\leq g$. By definition of $q$, we have $f:=\ma_{p_3}\bigl([p_3,q],\xi_{3,p_3}\bigr)<\varepsilon_3$. And by the choice of $S$, we also have $g:=\ma_{p_3}\bigl([p_3,p_1],[p_3,p_2]\bigr)<\varepsilon_3$. So
        \[
        b:=\ma_{p_2}\bigl([p_2,q],\xi_{3,p_2}\bigr)\leq a+c \leq e+d \leq 2d+f < \varepsilon'_2=\varepsilon_2-\theta_2.
        \]
        In other words, $V_3\subset \widetilde{V}_2$. Finally, analogously to the previous situation, \autoref{thm:triangle inequality for angles} gives
        \[
        \ma_{p_2}\bigl([p_2,q],\xi_{2,p_2}\bigr) \leq
        \ma_{p_2}\bigl([p_2,q],\xi_{3,p_2}\bigr) + 
        \ma_{p_2}(\xi_{3,p_2},\xi_{2,p_2})<\varepsilon_2.
        \]
        Thus $V_3\subset V_1\cap V_2$.
    \end{claimproof}
    
    Using now \autoref{lem:topological} we deduce the result.    
\end{proof}

\begin{cor}\label{cor:nice basis of nbhds}
    Let $\xi\in\bd^+Y$ and $p\in I^-(\xi)$. The set 
    \[
    \mathcal{N}_p(\xi):=\bigl\{
    V_{\xi_p(t),\varepsilon,R}(\xi) \mid t\in [0,\infty),\, \varepsilon,R>0
    \bigr\}
    \]
    is a neighborhood basis of $\xi$ for the cone topology.
\end{cor}
\begin{proof}
    Along the proof of \autoref{pop:cone_top} we proved that for any truncated cone $V:=V_{q,\varepsilon,R}(\xi)$ there exists some $t$ large enough and values $\smash{\tilde{\varepsilon},\tilde{R}>0}$ such that $V_{\tilde{p},\tilde{\varepsilon},\tilde{R}}(\xi)\subset V$, where $\smash{\tilde{p}=\xi_p(t)}$.
\end{proof}

\section{Angular metric in the timelike ideal boundary\label{sec:angular-metric}}

In this section we endow the future timelike ideal boundary with an extended distance, in a similar spirit as in the metric case \cite[Chapter~II.9]{BH}.

\begin{dfn}\label{def:angular-distance}
Let $Y$ be a proper, strongly causal, locally causally closed, regular \LpLS satisfying $\CBA(0)$ globally and let $\xi,\xi'\in \bd^+Y$ be such that $I^-(\xi)\cap I^-(\xi')\neq \varnothing$. Then the \textit{angular distance} $\ma(\xi,\xi')$ is defined by 
\begin{equation}\label{eq:def-angular distance}
\ma(\xi,\xi') = 
\sup\{\ma_x(\xi,\xi'): x\in I^-(\xi)\cap I^-(\xi')\}.
\end{equation}
If $I^-(\xi)\cap I^-(\xi')=\varnothing$, we set $\ma(\xi,\xi')=\infty$. We define $\ma$ analogously for any $\xi,\xi'\in \bd^+Y\cup \bd^-Y$, by considering the appropriate intersection $I^\pm(\xi)\cap I^\pm(\xi')$ in \eqref{eq:def-angular distance} if that intersection is non-empty, and setting $\ma(\xi,\xi')=\infty$ otherwise. 
\end{dfn}

\begin{pop}\label{pop:different-pasts-infinite-angle}
    Let $Y$ be a proper, globally hyperbolic, strongly causal, locally causally closed, regular \LpLS satisfying $\CBA(0)$ globally. Let $\xi,\xi'\in \bd^+Y$ (resp. $\xi,\xi'\in \bd^-Y$) be such that $I^-(\xi)\neq I^-(\xi')$ (resp. $I^+(\xi) \neq I^+(\xi')$). Then $\ma(\xi,\xi')=\infty$. 
    
    Additionally, if $Y$ is regularly localizable, the following implication holds: if $\xi\in \bd^+Y$, $\xi'\in \bd^-Y$ are such that $I^-(\xi)\neq I^+(\xi')$, then $\ma(\xi,\xi') = \infty$.
\end{pop}

\begin{rem}
Observe that, under strong causality, being regularly localizable (see \cite[Definition~3.16]{kunzinger-saemann2018}) is equivalent to being regular and localizable, by \cite[Lemma~3.6]{beran-kunzinger-rott2024}.
\end{rem}

\begin{proof}
    Let $\xi,\xi'\in \bd^+Y$ be such that $I^-(\xi)\neq I^-(\xi')$. If $S:=I^-(\xi)\cap I^-(\xi')=\varnothing$ there is nothing to prove. Thus, let us assume that $S\neq\varnothing$ and let $x\in S$. Consider the representatives $\xi_x,\xi'_x$ of $\xi,\xi'$, respectively, starting at $x$ (see~\autoref{pop:existence of pointwise rep}). As the pasts of these geodesic rays are not equal, there exists a parameter $t\in[0,\infty)$ such that either $\xi_x(t)\notin I^-(\xi')$ or $\xi'_x(t)\notin I^-(\xi)$. Without loss of generality, let us assume the former (see \autoref{subfig:infinite-angle-same-direction}).

    By a standard topological argument, being $\xi_x$ continuous, $I^-(\xi')$ open, $\xi_x(0)=x\in I^-(\xi')$ and $\xi_x(t)\notin I^-(\xi')$, there exists some parameter $s\in (0,t]$ such that $p_\infty:=\xi_x(s)$ is in the topological boundary of $I^-(\xi')$.

    \begin{figure}
    \centering
    \begin{subfigure}{0.45\textwidth}
    \centering
    \tikzset{every picture/.style={line width=0.75pt}}
    \begin{tikzpicture}[x=0.75pt,y=0.75pt,yscale=-1,xscale=1]
    \draw [color={rgb, 255:red, 128; green, 128; blue, 128 }  ,draw opacity=1 ]   (320.62,172.37) .. controls (339.92,143.85) and (355.81,116.52) .. (373.03,55.73) ;
    \draw [color={rgb, 255:red, 128; green, 128; blue, 128 }  ,draw opacity=1 ]   (332.73,201.28) .. controls (353.92,162.47) and (359.98,158.51) .. (373.03,55.73) ;
    \draw [color={rgb, 255:red, 20; green, 0; blue, 255 }  ,draw opacity=1 ]   (350.9,238.11) .. controls (363.95,192.96) and (371.9,156.72) .. (373.03,55.73) ;
    \draw [shift={(369.41,141.13)}, rotate = 95.33] [color={rgb, 255:red, 20; green, 0; blue, 255 }  ,draw opacity=1 ][line width=0.75]    (10.93,-4.9) .. controls (6.95,-2.3) and (3.31,-0.67) .. (0,0) .. controls (3.31,0.67) and (6.95,2.3) .. (10.93,4.9)   ;
    \draw  [dash pattern={on 4.5pt off 4.5pt}]  (287.32,27.21) -- (403.69,148.41) ;
    \draw  [dash pattern={on 4.5pt off 4.5pt}]  (295.83,84.24) -- (412.2,205.44) ;
    \draw  [dash pattern={on 4.5pt off 4.5pt}]  (373.03,55.73) -- (425.4,111.8) ;
    \draw  [dash pattern={on 4.5pt off 4.5pt}]  (287.32,27.21) -- (229,87.8) ;
    \draw  [dash pattern={on 4.5pt off 4.5pt}]  (373.03,55.73) -- (260.6,171) ;
    \draw  [dash pattern={on 4.5pt off 4.5pt}]  (295.83,84.24) -- (246.2,135) ;
    \draw [color={rgb, 255:red, 208; green, 2; blue, 27 }  ,draw opacity=1 ]   (350.9,238.11) .. controls (319.67,179.89) and (298.1,121.08) .. (287.32,27.21) ;
    \draw [shift={(306.53,129.49)}, rotate = 74.07] [color={rgb, 255:red, 208; green, 2; blue, 27 }  ,draw opacity=1 ][line width=0.75]    (10.93,-4.9) .. controls (6.95,-2.3) and (3.31,-0.67) .. (0,0) .. controls (3.31,0.67) and (6.95,2.3) .. (10.93,4.9)   ;
    \draw  [fill={rgb, 255:red, 0; green, 0; blue, 0 }  ,fill opacity=1 ] (293,84.24) .. controls (293,82.6) and (294.27,81.27) .. (295.83,81.27) .. controls (297.4,81.27) and (298.67,82.6) .. (298.67,84.24) .. controls (298.67,85.88) and (297.4,87.21) .. (295.83,87.21) .. controls (294.27,87.21) and (293,85.88) .. (293,84.24) -- cycle ;
    \draw  [fill={rgb, 255:red, 0; green, 0; blue, 0 }  ,fill opacity=1 ] (348.06,238.11) .. controls (348.06,236.47) and (349.33,235.14) .. (350.9,235.14) .. controls (352.46,235.14) and (353.73,236.47) .. (353.73,238.11) .. controls (353.73,239.75) and (352.46,241.08) .. (350.9,241.08) .. controls (349.33,241.08) and (348.06,239.75) .. (348.06,238.11) -- cycle ;
    \draw  [fill={rgb, 255:red, 0; green, 0; blue, 0 }  ,fill opacity=1 ] (302.65,127.02) .. controls (302.65,125.38) and (303.92,124.05) .. (305.48,124.05) .. controls (307.05,124.05) and (308.32,125.38) .. (308.32,127.02) .. controls (308.32,128.66) and (307.05,129.99) .. (305.48,129.99) .. controls (303.92,129.99) and (302.65,128.66) .. (302.65,127.02) -- cycle ;
    \draw  [fill={rgb, 255:red, 0; green, 0; blue, 0 }  ,fill opacity=1 ] (317.97,110.38) .. controls (317.97,108.74) and (319.24,107.41) .. (320.81,107.41) .. controls (322.38,107.41) and (323.65,108.74) .. (323.65,110.38) .. controls (323.65,112.02) and (322.38,113.35) .. (320.81,113.35) .. controls (319.24,113.35) and (317.97,112.02) .. (317.97,110.38) -- cycle ;
    \draw  [fill={rgb, 255:red, 0; green, 0; blue, 0 }  ,fill opacity=1 ] (319.11,172.37) .. controls (319.11,171.49) and (319.78,170.78) .. (320.62,170.78) .. controls (321.46,170.78) and (322.13,171.49) .. (322.13,172.37) .. controls (322.13,173.24) and (321.46,173.95) .. (320.62,173.95) .. controls (319.78,173.95) and (319.11,173.24) .. (319.11,172.37) -- cycle ; 
    \draw  [fill={rgb, 255:red, 0; green, 0; blue, 0 }  ,fill opacity=1 ] (331.22,201.28) .. controls (331.22,200.4) and (331.89,199.69) .. (332.73,199.69) .. controls (333.57,199.69) and (334.24,200.4) .. (334.24,201.28) .. controls (334.24,202.15) and (333.57,202.86) .. (332.73,202.86) .. controls (331.89,202.86) and (331.22,202.15) .. (331.22,201.28) -- cycle ;
    \draw  [fill={rgb, 255:red, 0; green, 0; blue, 0 }  ,fill opacity=1 ] (355.72,148.01) .. controls (355.72,147.13) and (356.4,146.43) .. (357.23,146.43) .. controls (358.07,146.43) and (358.75,147.13) .. (358.75,148.01) .. controls (358.75,148.88) and (358.07,149.59) .. (357.23,149.59) .. controls (356.4,149.59) and (355.72,148.88) .. (355.72,148.01) -- cycle ;
    \draw  [fill={rgb, 255:red, 0; green, 0; blue, 0 }  ,fill opacity=1 ] (342,134.15) .. controls (342,133.27) and (342.68,132.56) .. (343.52,132.56) .. controls (344.35,132.56) and (345.03,133.27) .. (345.03,134.15) .. controls (345.03,135.02) and (344.35,135.73) .. (343.52,135.73) .. controls (342.68,135.73) and (342,135.02) .. (342,134.15) -- cycle ;
    \draw  [fill={rgb, 255:red, 0; green, 0; blue, 0 }  ,fill opacity=1 ] (287.6,53.35) .. controls (287.6,51.71) and (288.87,50.38) .. (290.44,50.38) .. controls (292.01,50.38) and (293.28,51.71) .. (293.28,53.35) .. controls (293.28,54.99) and (292.01,56.32) .. (290.44,56.32) .. controls (288.87,56.32) and (287.6,54.99) .. (287.6,53.35) -- cycle ;
    
    \draw (358.52,232.11) node [anchor=north west][inner sep=0.75pt]  [font=\small]  {$x$};
    \draw (272.24,5.95) node [anchor=north west][inner sep=0.75pt]    {$\xi $};
    \draw (357.09,36.84) node [anchor=north west][inner sep=0.75pt]    {$\xi '$};
    \draw (269,48.53) node [anchor=north west][inner sep=0.75pt]  [font=\footnotesize]  {$r_{\infty }$};
    \draw (281.22,116.85) node [anchor=north west][inner sep=0.75pt]  [font=\footnotesize]  {$p_{\infty }$};
    \draw (313,89) node [anchor=north west][inner sep=0.75pt]  [font=\footnotesize]  {$q_{\infty }$};
    \draw (304.73,168.55) node [anchor=north west][inner sep=0.75pt]  [font=\scriptsize]  {$p_{n}$};
    \draw (326.51,126.07) node [anchor=north west][inner sep=0.75pt]  [font=\scriptsize]  {$q_{n}$};
    \draw (277.98,72.89) node [anchor=north west][inner sep=0.75pt]  [font=\footnotesize]  {$r_{1}$};
    \draw (321.01,217.44) node [anchor=north west][inner sep=0.75pt]  [font=\small,color={rgb, 255:red, 208; green, 2; blue, 27 }  ,opacity=1 ]  {$\xi _{x}$};
    \draw (364.53,204.77) node [anchor=north west][inner sep=0.75pt]  [font=\small,color={rgb, 255:red, 20; green, 0; blue, 255 }  ,opacity=1 ]  {$\xi '_{x}$};
    \end{tikzpicture}
    \caption{Case $\xi,\xi'\in\bd^+Y$.\label{subfig:infinite-angle-same-direction}}
    \end{subfigure}\quad
    \begin{subfigure}{0.45\textwidth}
    \centering
    \tikzset{every picture/.style={line width=0.75pt}}
    \begin{tikzpicture}[x=0.75pt,y=0.75pt,yscale=-1,xscale=1]
    \draw [color={rgb, 255:red, 20; green, 0; blue, 255 }  ,draw opacity=1 ]   (355.4,256) .. controls (353.32,222.45) and (355.38,191.95) .. (325.53,144.9) ;
    \draw [shift={(350.66,204.45)}, rotate = 257.37] [color={rgb, 255:red, 20; green, 0; blue, 255 }  ,draw opacity=1 ][line width=0.75]    (10.93,-3.29) .. controls (6.95,-1.4) and (3.31,-0.3) .. (0,0) .. controls (3.31,0.3) and (6.95,1.4) .. (10.93,3.29)   ;
    \draw [color={rgb, 255:red, 20; green, 0; blue, 255 }  ,draw opacity=1 ]   (355.4,256) .. controls (348.92,222.05) and (339.32,203.66) .. (314.33,170.92) ;
    \draw [shift={(343.64,216.89)}, rotate = 244.43] [color={rgb, 255:red, 20; green, 0; blue, 255 }  ,draw opacity=1 ][line width=0.75]    (10.93,-3.29) .. controls (6.95,-1.4) and (3.31,-0.3) .. (0,0) .. controls (3.31,0.3) and (6.95,1.4) .. (10.93,3.29)   ;
    \draw  [dash pattern={on 4.5pt off 4.5pt}]  (241.51,144.49) -- (355.4,256) ;
    \draw  [dash pattern={on 4.5pt off 4.5pt}]  (365.72,12.17) -- (432.9,79.03) ;
    \draw  [dash pattern={on 4.5pt off 4.5pt}]  (365.72,12.17) -- (262.07,120.13) ;
    \draw [color={rgb, 255:red, 208; green, 2; blue, 27 }  ,draw opacity=1 ]   (279.32,240.19) .. controls (317.73,173.39) and (356.52,78.04) .. (365.72,12.17) ;
    \draw [shift={(334.41,123.41)}, rotate = 111.21] [color={rgb, 255:red, 208; green, 2; blue, 27 }  ,draw opacity=1 ][line width=0.75]    (10.93,-4.9) .. controls (6.95,-2.3) and (3.31,-0.67) .. (0,0) .. controls (3.31,0.67) and (6.95,2.3) .. (10.93,4.9)   ;
    \draw  [fill={rgb, 255:red, 0; green, 0; blue, 0 }  ,fill opacity=1 ] (276.32,240.19) .. controls (276.32,238.56) and (277.66,237.23) .. (279.32,237.23) .. controls (280.97,237.23) and (282.32,238.56) .. (282.32,240.19) .. controls (282.32,241.83) and (280.97,243.16) .. (279.32,243.16) .. controls (277.66,243.16) and (276.32,241.83) .. (276.32,240.19) -- cycle ;
    \draw  [fill={rgb, 255:red, 0; green, 0; blue, 0 }  ,fill opacity=1 ] (297.19,201.6) .. controls (297.19,199.96) and (298.54,198.63) .. (300.19,198.63) .. controls (301.85,198.63) and (303.19,199.96) .. (303.19,201.6) .. controls (303.19,203.24) and (301.85,204.57) .. (300.19,204.57) .. controls (298.54,204.57) and (297.19,203.24) .. (297.19,201.6) -- cycle ;
    \draw  [fill={rgb, 255:red, 0; green, 0; blue, 0 }  ,fill opacity=1 ] (323.93,144.9) .. controls (323.93,144.03) and (324.64,143.32) .. (325.53,143.32) .. controls (326.41,143.32) and (327.13,144.03) .. (327.13,144.9) .. controls (327.13,145.78) and (326.41,146.49) .. (325.53,146.49) .. controls (324.64,146.49) and (323.93,145.78) .. (323.93,144.9) -- cycle ;
    \draw  [fill={rgb, 255:red, 0; green, 0; blue, 0 }  ,fill opacity=1 ] (351.02,71.02) .. controls (351.02,70.14) and (351.74,69.44) .. (352.62,69.44) .. controls (353.5,69.44) and (354.22,70.14) .. (354.22,71.02) .. controls (354.22,71.89) and (353.5,72.6) .. (352.62,72.6) .. controls (351.74,72.6) and (351.02,71.89) .. (351.02,71.02) -- cycle ;
    \draw  [fill={rgb, 255:red, 0; green, 0; blue, 0 }  ,fill opacity=1 ] (341.42,94.66) .. controls (341.42,93.02) and (342.77,91.69) .. (344.42,91.69) .. controls (346.08,91.69) and (347.42,93.02) .. (347.42,94.66) .. controls (347.42,96.3) and (346.08,97.63) .. (344.42,97.63) .. controls (342.77,97.63) and (341.42,96.3) .. (341.42,94.66) -- cycle ;
    \draw  [dash pattern={on 4.5pt off 4.5pt}]  (438,174.87) -- (355.4,256) ; 
    \draw  [fill={rgb, 255:red, 0; green, 0; blue, 0 }  ,fill opacity=1 ] (312.73,170.92) .. controls (312.73,170.04) and (313.44,169.33) .. (314.33,169.33) .. controls (315.21,169.33) and (315.93,170.04) .. (315.93,170.92) .. controls (315.93,171.79) and (315.21,172.5) .. (314.33,172.5) .. controls (313.44,172.5) and (312.73,171.79) .. (312.73,170.92) -- cycle ;
    \draw  [fill={rgb, 255:red, 0; green, 0; blue, 0 }  ,fill opacity=1 ] (357.32,47.38) .. controls (357.32,46.5) and (358.03,45.8) .. (358.92,45.8) .. controls (359.8,45.8) and (360.52,46.5) .. (360.52,47.38) .. controls (360.52,48.25) and (359.8,48.96) .. (358.92,48.96) .. controls (358.03,48.96) and (357.32,48.25) .. (357.32,47.38) -- cycle ;
    \draw  [dash pattern={on 4.5pt off 4.5pt}]  (412.97,102.23) -- (279.32,240.19) ;
    \draw  [fill={rgb, 255:red, 0; green, 0; blue, 0 }  ,fill opacity=1 ] (340.32,176.26) .. controls (340.32,175.38) and (341.04,174.68) .. (341.92,174.68) .. controls (342.81,174.68) and (343.52,175.38) .. (343.52,176.26) .. controls (343.52,177.13) and (342.81,177.84) .. (341.92,177.84) .. controls (341.04,177.84) and (340.32,177.13) .. (340.32,176.26) -- cycle ;
    \draw  [fill={rgb, 255:red, 0; green, 0; blue, 0 }  ,fill opacity=1 ] (326.53,189.91) .. controls (326.53,189.03) and (327.24,188.33) .. (328.13,188.33) .. controls (329.01,188.33) and (329.72,189.03) .. (329.72,189.91) .. controls (329.72,190.78) and (329.01,191.49) .. (328.13,191.49) .. controls (327.24,191.49) and (326.53,190.78) .. (326.53,189.91) -- cycle ;
    \draw  [dash pattern={on 4.5pt off 4.5pt}]  (220.25,186.74) -- (279.32,240.19) ;
    \draw  [fill={rgb, 255:red, 0; green, 0; blue, 0 }  ,fill opacity=1 ] (305.89,210.2) .. controls (305.89,208.57) and (307.23,207.24) .. (308.89,207.24) .. controls (310.55,207.24) and (311.89,208.57) .. (311.89,210.2) .. controls (311.89,211.84) and (310.55,213.17) .. (308.89,213.17) .. controls (307.23,213.17) and (305.89,211.84) .. (305.89,210.2) -- cycle ;

    \draw (263.26,231.62) node [anchor=north west][inner sep=0.75pt]  [font=\small]  {$x$};
    \draw (321.22,82.72) node [anchor=north west][inner sep=0.75pt]  [font=\footnotesize]  {$r_{\infty }$};
    \draw (274.58,194.88) node [anchor=north west][inner sep=0.75pt]  [font=\footnotesize]  {$p_{\infty }$};
    \draw (302.93,147.08) node [anchor=north west][inner sep=0.75pt]  [font=\scriptsize]  {$p_{n}$};
    \draw (342.12,115.07) node [anchor=north west][inner sep=0.75pt]  [font=\small,color={rgb, 255:red, 208; green, 2; blue, 27 }  ,opacity=1 ]  {$\xi _{x}$};
    \draw (356.04,209.68) node [anchor=north west][inner sep=0.75pt]  [font=\small,color={rgb, 255:red, 20; green, 0; blue, 255 }  ,opacity=1 ]  {$\xi '_{p_{n}}$};
    \draw (339.32,51.04) node [anchor=north west][inner sep=0.75pt]  [font=\scriptsize]  {$r_{n}$};
    \draw (349.32,172.6) node [anchor=north west][inner sep=0.75pt]  [font=\scriptsize]  {$q_{n}$};
    \draw (304.17,217.43) node [anchor=north west][inner sep=0.75pt]  [font=\footnotesize]  {$q_{\infty }$};
    \end{tikzpicture}
        \caption{Case $\xi\in\bd^+Y$ and $\xi'\in\bd^-Y$.\label{subfig:infinite-angle-opposite-direction}}
    \end{subfigure}
    \caption{Situations in the proof of \autoref{pop:different-pasts-infinite-angle}.}
    \label{fig:infinite_angle}
    \end{figure}
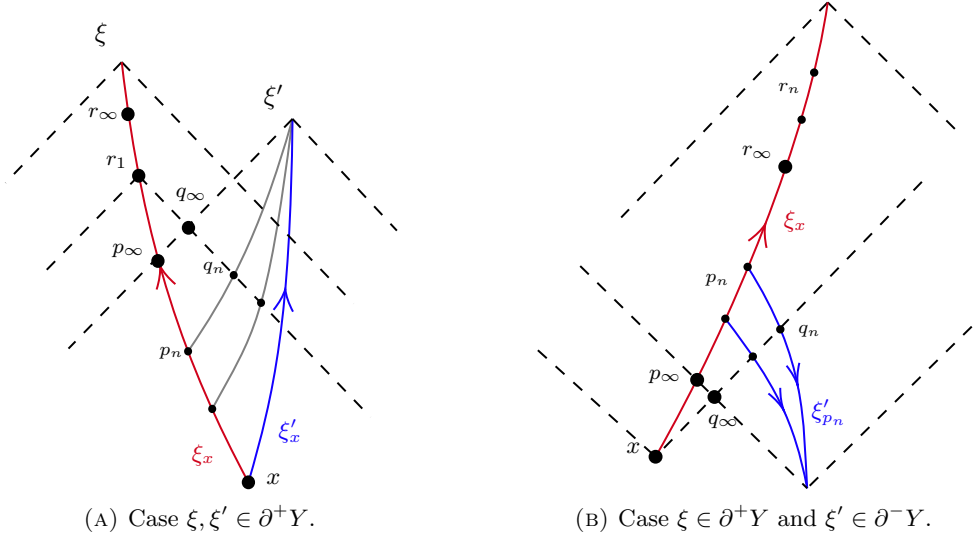

    Consider an increasing sequence $s_n\to s$ of positive real numbers and define $p_n:=\xi_x(s_n)\to p_\infty$. Define $r_n:=\xi_x(s_n+\lambda)\to r_\infty:=\xi_x(s+\lambda)$, where $\lambda$ is any positive real number. Notice that, as the topological boundary of the timelike past of any set is achronal \cite{minguzzi_lorentzian_2019}, all the points $r_n,r_\infty$ are in $Y\backslash \overline{I^-(\xi')}$. For each $n\in\N$, take $\smash{\xi'_{p_n}}$ the unique future directed timelike geodesic ray asymptotic to $\xi'$ and starting at $p_n$. By a similar argument as before, for each $n$ there is a parameter value $s'_n$ such that $q_n:=\xi'_{p_n}(s'_n)$ is in the topological boundary of $I^-(r_1)$. 

    As all the $q_n$'s are contained in the causal diamond $J^+(x)\cap J^-(r_1)$, which is compact by global hyperbolicity, there is a converging subsequence $q_{\sigma(n)}\to q_\infty$, which we denote again by $q_n$ for simplicity. Moreover, $q_\infty$ is also in the topological boundary of $I^-(r_1)$ and, in particular, $q_\infty\neq p_\infty\in I^-(r_1)$. In addition, by local causal closedness, the relations $p_n\ll q_n$ imply that $p_\infty\leq q_\infty$. 
    
    Let us notice that this implies that $q_\infty$ is also in the topological boundary of $I^-(\xi')$. Indeed, it has to be in the closure (being limit of a sequence in $I^-(\xi')$) and it cannot be in the interior because otherwise one would obtain that also $p_\infty\in I^-(\xi')$ via push-up. The reasonings in this paragraph are not really needed for the proof, but just to justify the drawing in \autoref{subfig:infinite-angle-same-direction}.

    Notice also that $p_\infty\not\ll q_\infty$ (nor vice-versa), as both are in the topological boundary of $I^-(\xi')$, which is achronal.

    Let us now prove that there is a causal (in fact null) segment joining $p_\infty$ with $q_\infty$, and a causal segment joining $q_\infty$ with $r_\infty$. By global hyperbolicity, we have non-total imprisonment, i.e., for every compact set $K\subset Y$ the $d$-length of the causal curves contained in $K$ is bounded. The restrictions of $\xi_{p_n}$ to $[0,s_n']$, i.e., the parts connecting $p_n$ to $q_n$, are all contained in the compact diamond $J^+(x)\cap J^-(r_1)$. As a consequence, they all have uniformly bounded $d$-length. In particular, their parametrization to some common domain $[a,b]$ is uniformly Lipschitz. \autoref{thm:limit curve theorem} gives now existence of a causal curve joining $p_\infty$ to $q_\infty$. For the other case, notice that $q_\infty\ll r_\infty$ and the global curvature bound gives that the whole space is (strictly) timelike geodesically connected (\autoref{def:comparison_neighb}).

    So, for a contradiction, assume that the angle $\ma(\xi,\xi')$ is finite. Consider the timelike triangles $\triangle p_n q_n r_n$ where the indices are understood to be taken up to the subsequence previously mentioned, and the comparison angles $\alpha_n:=\smash{\widetilde{\ma}_{p_n}(r_n,q_n)}$. Then,
    \begin{align}
        \alpha_n &= \widetilde{\ma}_{q_n}(p_n,r_n) - \widetilde{\ma}_{r_n}(p_n,q_n), \tag*{by \autoref{lem:sum of angles of triangle}}\\
        &\leq \ma_{q_n}(p_n,r_n) - \ma_{r_n}(p_n,q_n), \tag*{by \autoref{thm:monotonicity of comparison angles}}\\
        &\leq \ma_{q_n}(\xi'_{q_n},[q_n,r_n]) - \ma_{r_n}(\xi_{r_n},[r_n,q_n]), \tag*{by \autoref{thm:triangle inequality for angles}(\ref{item:special case of triangle inequality})}\\
        &\leq \ma_{q_n}(\xi'_{q_n},\xi_{q_n})+\ma_{q_n}(\xi_{q_n},[q_n,r_n]) - \ma_{r_n}(\xi_{r_n},[r_n,q_n]), \tag*{by \autoref{thm:triangle inequality for angles}(\ref{item:same orientation triangle inequality})}\\
        &\leq \ma (\xi,\xi')+\ma_{q_n}(\xi_{q_n},[q_n,r_n])-\ma_{r_n}(\xi_{r_n},[r_n,q_n]), \tag*{by \autoref{def:angular-distance}.}
    \end{align}
    By \autoref{lem:angles between asymptotic rays}, 
    \[
    \ma_{q_n}(\xi_{q_n},[q_n,r_n])\leq \ma_{r_n}(\xi_{r_n},[r_n,q_n]),
    \]
    thus $\alpha_n \leq \ma(\xi,\xi')$.
    
    As a consequence,  
    \[
    \alpha_n=\cosh^{-1}\left(\frac{\tau(p_n,r_n)^2+\tau(p_n,q_n)^2-\tau(q_n,r_n)^2}{2\tau(p_n,r_n)\tau(p_n,q_n)}\right)
    \]
    is bounded.
    
    Now, as $\tau(p_n,q_n)\to \tau(p_\infty,q_\infty)=0$ whereas $\tau(p_n,r_n)=\lambda$ is constant, one deduces that the numerator must also converge to $0$. In other words, $\tau(q_n,r_n)\to \tau(q_\infty,r_\infty)=\lambda$. We used here that the whole space is a comparison neighborhood where $\tau$ is continuous.

    To sum up, we have a triangle $\triangle p_\infty q_\infty r_\infty$ such that $\tau(p_\infty,r_\infty)=\tau(q_\infty,r_\infty)$ and $\tau(p_\infty,q_\infty)=0$. Concatenating the segments $[p_\infty,q_\infty]$ with $[q_\infty,r_\infty]$ gives a causal curve whose $\tau$-length is $\lambda=\tau(p_\infty,r_\infty)$. In particular, this is a maximal geodesic, but it does not have a causal character as the first segment is timelike whereas the second one is null, thus contradicting regularity. The case $\xi,\xi'\in\bd^-Y$ and $I^{+}(\xi)\neq I^{+}(\xi')$ is completely analogous. 
    
    Assuming further that $Y$ is regularly localizable, the case $\xi\in \bd^+Y$, $\xi'\in\bd^-Y$ and $I^-(\xi)\neq I^+(\xi')$ is also very similar (see now \autoref{subfig:infinite-angle-opposite-direction}). Indeed, if  $\ma(\xi,\xi')<\infty$ then $I^-(\xi)\cap I^+(\xi')\neq \varnothing$ by definition. Since $I^-(\xi)\cup I^+(\xi')\setminus I^-(\xi)\cap I^+(\xi')\neq \varnothing$, without loss of generality we can assume there exists $x\in I^-(\xi)\setminus I^+(\xi')$. Moreover, since $Y$ is regularly localizable, we can further assume that $x\not\in \overline{I^+(\xi')}$. This is because, if $x\in\overline{I^+(\xi')}$, by regular localizability we can find $x'\ll x$, and such $x'$ cannot be in $\overline{I^+(\xi')}$; otherwise $I^-(x)$, which is a neighborhood of $x'$, would intersect $I^+(\xi')$, and this implies $x\in I^+(\xi')$, i.e.,a contradiction.
    
    Therefore, we can argue as above to obtain a minimal $t_\infty>0$ such that $\xi_x(t)\in I^-(\xi)\cap I^+(\xi')$ for all $t>t_\infty$. Then $\xi_x(t_\infty)$ is in the topological boundary of $I^+(\xi')$. Let $t_n>t_\infty$ be such that $t_n\to t_\infty$ and, for $n\in\mathbb{N}\cup\{\infty\}$, define $p_n=\xi_p(t_n)$, $r_n= \xi_x(t_n+\lambda)$ for some fixed constant $\lambda>0$. For $n\in\mathbb{N}$, let $q_n$ be the unique point in the intersection of $\xi'_{p_n}$ with the topological boundary of the causal diamond $J(x,r_1)$. The latter set is compact by global hyperbolicity, thus, up to a subsequence, $q_n \to q_\infty$ for some $q_\infty\in J(x,r_1)$. Arguing as in the case $\xi,\xi'\in \bd^+Y$, we obtain a triangle $\triangle p_\infty q_\infty r_\infty$ where $\tau(p_\infty,r_\infty) = \tau(q_\infty,r_\infty)$ and $\tau(q_\infty,p_\infty) = 0$, which gives the same contradiction as before.
\end{proof} 

\begin{pop}\label{pop:main1}
Let $Y$ be a proper, globally hyperbolic, strongly causal, locally causally closed and regular \LpLS satisfying $\CBA(0)$ globally. Then the function $(\xi,\xi')\mapsto \ma(\xi,\xi')$ defines an extended metric (i.e.\ $\ma$ might take infinite values) on $\bd^+Y$. Analogously with $\ma$ restricted to $\bd^-Y$.

Furthermore, for any $\xi\in\bd^+Y$, $\xi'\in \bd^-Y$, $\ma(\xi,\xi') = 0$ if and only if for every $x\in I^-(\xi)=I^+(\xi')$ there exists a geodesic line $\gamma\colon \mathbb{R}\to Y$ (which is unique) such that $\gamma(\infty) = \xi$, $\bar{\gamma}(\infty) = \xi'$ and $\gamma(0) = x$, where $\bar{\gamma}\colon \mathbb{R}\to Y$ is given by $\bar{\gamma}(t) = \gamma(-t)$.
\end{pop}

\begin{proof}
By \autoref{pop:uniqueness of pointwise rep}, $\ma$ is well-defined when restricted to either $\bd^+Y$ or $\bd^-Y$. Moreover, it is clear that $\ma$ is non-negative and symmetric, and that $\ma(\xi,\xi)=0$, for any $\xi\in\bd^+Y\cup\bd^-Y$.

Now consider $\xi,\xi'\in\bd^+Y$ and assume that $\ma(\xi,\xi')=0$. \autoref{pop:different-pasts-infinite-angle} ensures that $S:=I^-(\xi)=I^-(\xi')$, and by \autoref{pop:existence of pointwise rep}, for any $x\in S$ there are unique future-directed timelike rays $\xi_x$ and $\xi'_x$ with $\xi_x(0)=\xi'_x(0)=x$ and $\xi_x(\infty) = \xi$, $\xi'_x(\infty) = \xi'$. Now, by the definition of the angular distance, $\ma(\xi,\xi')=0$ implies that for any $x\in S$, one must have $\ma_x(\xi_x,\xi'_x) = 0$. 

In particular, if we fix one such $x\in S$ and define $y=\xi'_x(t)$ for some $t\in [0,\infty)$, we must have $\ma_y(\xi_y,\xi'_y)=0$. Therefore, by \autoref{thm:concatenation with angle zero}, the concatenation of $\xi'_x|_{[0,t]}$ with $\xi_y$ yields a future-directed timelike geodesic ray $\smash{\tilde{\xi}_x}$ such that $\smash{\tilde{\xi}_x(0)=x}$ and $\smash{\tilde{\xi}_x(\infty)=\xi}$. By \autoref{pop:uniqueness of pointwise rep}, it follows that $\smash{\tilde{\xi}_x = \xi_x}$ and in particular $\xi_x(t) = \tilde{\xi}_x(t) = \xi'_x(t)$. Since $t\in[0,\infty)$ was arbitrary, we obtain that $\xi_x=\xi'_x$ which implies $\xi=\xi'$. Similarly when $\xi,\xi'\in \bd^-Y$.

On the other hand, if $\xi\in \bd^+Y$ and $\xi'\in\bd^-Y$ and $\ma(\xi,\xi')=0$ one obtains similarly that $\ma_x(\xi_x,\xi'_x)=0$ for every $x\in I^-(\xi)=I^+(\xi')$. By \autoref{thm:concatenation with angle zero}, for any such $x$, $\xi_x$ and $\xi'_x$ fit together to a line. By definition, this line has $\xi$ and $\xi'$ for ideal points, and uniqueness follows from \autoref{pop:uniqueness of pointwise rep}.
 
Finally, we prove the triangle inequality as follows. Let $\xi,\xi',\xi''\in\bd^+Y$ be such that $\ma(\xi,\xi')<\infty$ and $\ma(\xi',\xi'')<\infty$. Then there is some $x\in Y$ and representatives $\xi_x,\xi'_x,\xi''_x$ of the corresponding ideal points at $x$. We know that the triangle inequality holds for the angle between timelike geodesic directions at $x\in Y$, by Theorems~\ref{thm:triangle inequality for angles} and \ref{thm:monotonicity implies existence of angles}. In particular,
\[
\ma_x(\xi_x,\xi'_x) + \ma_x(\xi'_x,\xi_x'') \geq \ma_x(\xi_x,\xi''_x).
\]
Taking the supremum over $x\in Y$ yields the claim. Analogously with $\xi,\xi',\xi''\in\bd^-Y$.
\end{proof} 

The following result provides a relation between the restriction of the cone topology to the future ideal boundary and the topology induced in such a boundary by the angular metric. This is a Lorentzian analogue to the metric result \cite[Proposition~II.9.7.(1)]{BH}.

\begin{pop}\label{pop:cone-top-coarser}
Let $Y$ be a proper, globally hyperbolic, strongly causal, locally causally closed and regular \LpLS satisfying $\CBA(0)$ globally. Then the cone topology on $\bd^+Y$ is coarser than the topology induced by the angular metric. Analogously with $\bd^-Y$. 
\end{pop}
\begin{proof}
If $\xi\in\bd^+Y$ and $V_{p,\varepsilon,R}(\xi)$ is a basic neighborhood of $\xi$ in the cone topology (\autoref{def:basic nbhd cone topology}), and $\xi''\in \bd^+Y$ is such that $\ma(\xi,\xi') < \varepsilon$ then, by definition of the angular metric, $\ma_p(\xi_p,\xi_p') < \varepsilon$, which implies that $\xi' \in V_{p,\varepsilon,R}(\xi)$. Therefore $\smash{B_\varepsilon^{\ma}}(\xi) \subset V_{p,\varepsilon,R}(\xi)$ and the claim follows.
\end{proof}

In general, the restriction of the cone topology to the ideal boundary need not coincide with the topology induced by the angular metric (see \autoref{rem:cone-topology-angular-topology}).

Now we relate this construction with the classical notion of causal boundary \cite{geroch_ideal_1972,flores_final_2011}. Recall that a subset $P\subset Y$ is called a \textit{past set} if $I^-(P)=P$. A a (non-empty) past set $P$ that cannot be written as the proper union of two past sets $P_i\subsetneq P$ is called \textit{indecomposable past set} (IP). If an IP is the past of a point in $Y$ then it is called a \textit{proper indecomposable past set} (PIP). Otherwise, it is a \textit{terminal indecomposable past set} (TIP). The future counterparts are defined analogously. As a consequence of \cite[Theorem~3]{harris1998}, for any future-directed timelike geodesic ray $\xi$, the set $I^-(\xi)$ is a TIP. This, combined with \autoref{pop:different-pasts-infinite-angle} and \autoref{pop:main1} yields the following.
\begin{cor}\label{cor:components of bd are inside TIFs}
Let $Y$ be a proper, globally hyperbolic, strongly causal, locally causally closed and regular \LpLS satisfying $\CBA(0)$ globally. Then each finiteness component of $(\bd^+Y,\ma)$ is contained in some TIP. Analogously, each finiteness component of $(\bd^-Y,\ma)$ is contained in some TIF.   
\end{cor} 

Intuitively, the (future or past) ideal boundary distinguishes points which are regarded as the same by the causal boundary. Notice, however, that not all TIPs can be obtained as the timelike past of a future-directed timelike geodesic ray and, therefore, they do not appear in the ideal boundary even though they appear in the causal one.
For instance, in $2$-dimensional Minkowski space all future-directed timelike rays have the whole space as their past (see \autoref{lem:cones-past-whole-space}), whereas a half-plane below any diagonal line (i.e., at angle $\pm45^\circ$ with the horizontal) is also a TIP.

The following proposition provides reformulations of the angular distance, analogous to the two first items of \cite[Proposition~II.9.8]{BH}. These reformulations will be useful to prove \autoref{pop:completeness} and \autoref{cor:geodesicity}, which establish that $(\bd^+Y,\ma)$ is complete and geodesic under some technical assumptions on $Y$.

\begin{pop}\label{pop:properties of angle}
Let $Y$ be a proper, globally hyperbolic, strongly causal, locally causally closed and regular \LpLS satisfying $\CBA(0)$ globally. Let $x \in Y$ and $\gamma, \gamma' \colon [0,\infty) \to Y$ be future-directed timelike geodesic rays with $\gamma(0) = \gamma'(0) = x$,  $\gamma(\infty) = \xi$, and $\gamma'(\infty) = \xi'$. Then: 
\begin{enumerate}
    \item{\label{item:increasing-angle}}  The function (see \autoref{subfig:increasing-angle-ray})
    \[
        t \;\mapsto\; \ma_{\gamma(t)}(\xi,\xi')
    \]
    is non-decreasing.
    
    \item{\label{item:alternative-definition-angular}}  Assuming that $\ma(\xi,\xi') < \infty$ and calling $A = \{(t,t') \in \mathbb{R}^2: \gamma(t)\ll \gamma'(t')\ \text{ or }\ \gamma'(t')\ll \gamma(t)\}$, one has
    \[\ma(\xi, \xi') = \overline{\ma}_x(\xi,\xi') := \sup_{(t,t')\in A}
    \widetilde{\ma}_{x}(\gamma(t),\gamma'(t'))
    = \lim_{\substack{(t,t')\in A \\ t,t'\to \infty}} \widetilde{\ma}_{x}(\gamma(t),\gamma'(t')),
    \]
    and in particular, $\overline{\ma}_x(\xi,\xi')$ does not depend on $x\in S=I^-(\xi)=I^-(\xi')$. Moreover,
    \[
        \lim_{t \to \infty} \ma_{\gamma(t)}(\xi,\xi') = \ma(\xi,\xi').
    \]
\end{enumerate}
\end{pop}

\begin{figure}
    \centering
    \begin{subfigure}{0.45\textwidth}
    \centering
    \tikzset{every picture/.style={line width=0.75pt}}
    \begin{tikzpicture}[x=0.75pt,y=0.75pt,yscale=-0.8,xscale=0.8] 
    \draw    (228,244.33) .. controls (236,194) and (329.33,84.33) .. (378.33,54.33) ;
    \draw    (228,244.33) .. controls (231.25,193.75) and (185.67,91) .. (147,48.33) ;
    \draw    (218.5,184.83) .. controls (234.5,135.83) and (300.08,56) .. (340.75,29) ;
    \draw    (201.58,138.92) .. controls (223.92,92.92) and (249.33,56.75) .. (291,19.75) ;
    \draw  [draw opacity=0] (191.1,117.95) .. controls (194.25,116.33) and (197.81,115.42) .. (201.58,115.42) .. controls (205.43,115.42) and (209.06,116.37) .. (212.25,118.05) -- (201.58,138.92) -- cycle ; \draw   (191.1,117.95) .. controls (194.25,116.33) and (197.81,115.42) .. (201.58,115.42) .. controls (205.43,115.42) and (209.06,116.37) .. (212.25,118.05) ;  
    \draw  [draw opacity=0] (211.03,162.58) .. controls (213.37,161.77) and (215.88,161.33) .. (218.5,161.33) .. controls (221.62,161.33) and (224.6,161.96) .. (227.31,163.09) -- (218.5,184.83) -- cycle ; \draw   (211.03,162.58) .. controls (213.37,161.77) and (215.88,161.33) .. (218.5,161.33) .. controls (221.62,161.33) and (224.6,161.96) .. (227.31,163.09) ;  
    \draw  [draw opacity=0] (227,220.85) .. controls (227.33,220.84) and (227.67,220.83) .. (228,220.83) .. controls (230.33,220.83) and (232.58,221.18) .. (234.7,221.83) -- (228,244.33) -- cycle ; \draw   (227,220.85) .. controls (227.33,220.84) and (227.67,220.83) .. (228,220.83) .. controls (230.33,220.83) and (232.58,221.18) .. (234.7,221.83) ;   
    \draw  [fill={rgb, 255:red, 0; green, 0; blue, 0 }  ,fill opacity=1 ] (226.25,244.08) .. controls (226.25,243.25) and (226.92,242.58) .. (227.75,242.58) .. controls (228.58,242.58) and (229.25,243.25) .. (229.25,244.08) .. controls (229.25,244.91) and (228.58,245.58) .. (227.75,245.58) .. controls (226.92,245.58) and (226.25,244.91) .. (226.25,244.08) -- cycle ;
    \draw  [fill={rgb, 255:red, 0; green, 0; blue, 0 }  ,fill opacity=1 ] (217,184.83) .. controls (217,184) and (217.67,183.33) .. (218.5,183.33) .. controls (219.33,183.33) and (220,184) .. (220,184.83) .. controls (220,185.66) and (219.33,186.33) .. (218.5,186.33) .. controls (217.67,186.33) and (217,185.66) .. (217,184.83) -- cycle ;
    \draw  [fill={rgb, 255:red, 0; green, 0; blue, 0 }  ,fill opacity=1 ] (200.08,138.92) .. controls (200.08,138.09) and (200.75,137.42) .. (201.58,137.42) .. controls (202.41,137.42) and (203.08,138.09) .. (203.08,138.92) .. controls (203.08,139.75) and (202.41,140.42) .. (201.58,140.42) .. controls (200.75,140.42) and (200.08,139.75) .. (200.08,138.92) -- cycle ;
    
    \draw (210.75,243.15) node [anchor=north west][inner sep=0.75pt]  [font=\small]  {$x$};
    \draw (184.25,181.4) node [anchor=north west][inner sep=0.75pt]  [font=\small]  {$\gamma ( t)$};
    \draw (167.5,138.65) node [anchor=north west][inner sep=0.75pt]  [font=\small]  {$\gamma ( s)$};
    \draw (283.5,152.4) node [anchor=north west][inner sep=0.75pt]  [font=\small]  {$\gamma '$};
    \draw (223.75,34.9) node [anchor=north west][inner sep=0.75pt]  [font=\small]  {$\xi '_{\gamma (s)}$};
    \draw (243.25,72.65) node [anchor=north west][inner sep=0.75pt]  [font=\small]  {$\xi '_{\gamma (t)}$};
    \draw (148.75,72.65) node [anchor=north west][inner sep=0.75pt]  [font=\small]  {$\gamma $};
    \end{tikzpicture}
    \caption{Situation described in \autoref{pop:properties of angle}(\ref{item:increasing-angle}).\label{subfig:increasing-angle-ray}}
    \end{subfigure}\quad
    \begin{subfigure}{0.45\textwidth}
    \centering
    \tikzset{every picture/.style={line width=0.75pt}}
    \begin{tikzpicture}[x=0.75pt,y=0.75pt,yscale=-0.85,xscale=0.85]
    \draw [line width=0.9]    (311,195.33) .. controls (319,145) and (384,71.5) .. (423.25,38) ;
    \draw [line width=0.9]    (311,195.33) .. controls (314.25,144.75) and (273.2,61.6) .. (253.2,31.2) ;
    \draw [line width=0.9]    (301.5,135.83) .. controls (318,87.25) and (351,44.25) .. (376.75,20.25) ;
    \draw  [draw opacity=0] (294.03,113.58) .. controls (296.37,112.77) and (298.88,112.33) .. (301.5,112.33) .. controls (304.62,112.33) and (307.6,112.96) .. (310.31,114.09) -- (301.5,135.83) -- cycle ; \draw   (294.03,113.58) .. controls (296.37,112.77) and (298.88,112.33) .. (301.5,112.33) .. controls (304.62,112.33) and (307.6,112.96) .. (310.31,114.09) ;   
    \draw  [fill={rgb, 255:red, 0; green, 0; blue, 0 }  ,fill opacity=1 ] (309.25,195.08) .. controls (309.25,194.25) and (309.92,193.58) .. (310.75,193.58) .. controls (311.58,193.58) and (312.25,194.25) .. (312.25,195.08) .. controls (312.25,195.91) and (311.58,196.58) .. (310.75,196.58) .. controls (309.92,196.58) and (309.25,195.91) .. (309.25,195.08) -- cycle ;
    \draw  [fill={rgb, 255:red, 0; green, 0; blue, 0 }  ,fill opacity=1 ] (300,135.83) .. controls (300,135) and (300.67,134.33) .. (301.5,134.33) .. controls (302.33,134.33) and (303,135) .. (303,135.83) .. controls (303,136.66) and (302.33,137.33) .. (301.5,137.33) .. controls (300.67,137.33) and (300,136.66) .. (300,135.83) -- cycle ;
    \draw  [fill={rgb, 255:red, 0; green, 0; blue, 0 }  ,fill opacity=1 ] (388.28,70.32) .. controls (388.28,69.49) and (388.95,68.82) .. (389.78,68.82) .. controls (390.61,68.82) and (391.28,69.49) .. (391.28,70.32) .. controls (391.28,71.15) and (390.61,71.82) .. (389.78,71.82) .. controls (388.95,71.82) and (388.28,71.15) .. (388.28,70.32) -- cycle ;
    \draw    (301.5,135.83) .. controls (342.8,109.6) and (343.2,107.7) .. (389.78,70.32) ;
    \draw  [draw opacity=0] (308,120.16) .. controls (311.16,121.57) and (313.8,123.99) .. (315.56,127.02) -- (301.5,135.83) -- cycle ; \draw   (308,120.16) .. controls (311.16,121.57) and (313.8,123.99) .. (315.56,127.02) ;   
    \draw  [draw opacity=0] (287.6,96.55) .. controls (291.96,95.12) and (296.64,94.35) .. (301.5,94.35) .. controls (316.19,94.35) and (329.16,101.41) .. (336.94,112.2) -- (301.5,135.83) -- cycle ; \draw   (287.6,96.55) .. controls (291.96,95.12) and (296.64,94.35) .. (301.5,94.35) .. controls (316.19,94.35) and (329.16,101.41) .. (336.94,112.2) ;   
    \draw  [draw opacity=0] (309.15,131.3) .. controls (309.8,132.65) and (310.17,134.2) .. (310.17,135.83) .. controls (310.17,139.99) and (307.79,143.53) .. (304.47,144.87) -- (301.5,135.83) -- cycle ; \draw   (309.15,131.3) .. controls (309.8,132.65) and (310.17,134.2) .. (310.17,135.83) .. controls (310.17,139.99) and (307.79,143.53) .. (304.47,144.87) ;  
    \draw  [draw opacity=0] (349.92,116.81) .. controls (346.6,114.47) and (343.54,111.85) .. (340.78,108.98) -- (389.78,70.32) -- cycle ; \draw   (349.92,116.81) .. controls (346.6,114.47) and (343.54,111.85) .. (340.78,108.98) ;   
    \draw  [draw opacity=0] (380.57,77.2) .. controls (378.97,75.31) and (378.02,72.92) .. (378.02,70.32) .. controls (378.02,64.2) and (383.28,59.25) .. (389.78,59.25) .. controls (392.75,59.25) and (395.47,60.29) .. (397.54,61.99) -- (389.78,70.32) -- cycle ; \draw   (380.57,77.2) .. controls (378.97,75.31) and (378.02,72.92) .. (378.02,70.32) .. controls (378.02,64.2) and (383.28,59.25) .. (389.78,59.25) .. controls (392.75,59.25) and (395.47,60.29) .. (397.54,61.99) ;  
    
    \draw (298,196) node [anchor=north west][inner sep=0.75pt]  [font=\small]  {$x$};
    \draw (272.25,132.4) node [anchor=north west][inner sep=0.75pt]  [font=\small]  {$\gamma ( t)$};
    \draw (424.9,30.8) node [anchor=north west][inner sep=0.75pt]  [font=\small]  {$\gamma '$};
    \draw (339,10.45) node [anchor=north west][inner sep=0.75pt]  [font=\small]  {$\xi '_{\gamma ( t)}$};
    \draw (241.75,39.25) node [anchor=north west][inner sep=0.75pt]  [font=\small]  {$\gamma $};
    \draw (393,66.85) node [anchor=north west][inner sep=0.75pt]  [font=\small]  {$\gamma '( t')$};
    \draw (313.5,108.4) node [anchor=north west][inner sep=0.75pt]  [font=\small]  {$e$};
    \draw (311.15,134.7) node [anchor=north west][inner sep=0.75pt]  [font=\small]  {$a$};
    \draw (297.5,99) node [anchor=north west][inner sep=0.75pt]  [font=\small]  {$d$};
    \draw (303.25,76.9) node [anchor=north west][inner sep=0.75pt]  [font=\small]  {$b$};
    \draw (365.5,58.9) node [anchor=north west][inner sep=0.75pt]  [font=\small]  {$f$};
    \draw (335,113) node [anchor=north west][inner sep=0.75pt]  [font=\small]  {$g$};
    \end{tikzpicture}
    \caption{Situation in the proof of \autoref{pop:properties of angle}(\ref{item:alternative-definition-angular}).\label{subfig:alternative-def-angular-metric}}
    \end{subfigure}
    \caption{}
    \label{fig:properties-angular-metric}
\end{figure}

\begin{proof}
\hfill
\begin{enumerate}
    \item Let $0\leq t<s$. By \autoref{lem:angles between asymptotic rays} and \autoref{thm:triangle inequality for angles}(\ref{item:special case of triangle inequality}), respectively
    \[
    \ma_{\gamma(t)}(\xi,\xi') \leq 
    \ma_{\gamma(s)}\bigl([\gamma(s),\gamma(t)],\xi'_{\gamma(s)}\bigr) \leq 
    \ma_{\gamma(s)}(\xi,\xi').
    \]
    
    \item Let $t,t'$ be such that $\gamma(t) \ll \gamma'(t')$. By \autoref{lem:sum of angles of triangle} and  monotonicity of comparison angles, we have
    \[
    \widetilde{\ma}_{x}(\gamma(t),\gamma'(t')) = \widetilde{\ma}_{\gamma(t)}(x,\gamma'(t'))-\widetilde{\ma}_{\gamma'(t)}(x,\gamma(t)) 
    \leq 
    \underbrace{\ma_{\gamma(t)}(x,\gamma'(t'))}_{a}-\underbrace{\ma_{\gamma'(t)}(x,\gamma(t))}_{g}.
    \]
    Now, by \autoref{thm:triangle inequality for angles}(\ref{item:special case of triangle inequality}), for every $s>t$ one has (see \autoref{subfig:alternative-def-angular-metric}),
    \[
    \underbrace{\ma_{\gamma(t)}(x,\gamma'(t'))}_{a} \leq 
    \underbrace{\ma_{\gamma(t)}(\gamma(s),\gamma'(t'))}_{b} \leq
    \underbrace{\ma_{\gamma(t)}(\xi,\xi')}_{d}
    +\underbrace{\ma_{\gamma(t)}\bigl(\xi'_{\gamma(t)},[\gamma(t),\gamma'(t')]\bigr)}_{e}.
    \]
    By \autoref{lem:angles between asymptotic rays} and one more application of \autoref{thm:triangle inequality for angles}(\ref{item:special case of triangle inequality}), we have
    \[
    e=\ma_{\gamma(t)}\bigl(\xi'_{\gamma(t)},[\gamma(t),\gamma'(t')]\bigr)
    \leq \underbrace{\ma_{\gamma'(t')}\bigl([\gamma'(t'),\gamma(t)],\xi'_{\gamma'(t')}\bigr)}_{f}
    \leq \underbrace{\ma_{\gamma'(t')}(x,\gamma(t))}_{g}.
    \]
    Combining the inequalities above,
    \begin{equation}\label{eq:properties of angle 1}
    \widetilde{\ma}_{x}(\gamma(t),\gamma'(t')) \leq a-g \leq d+e-g \leq d= \ma_{\gamma(t)}(\xi,\xi') \leq \ma(\xi,\xi').
    \end{equation}
    An analogous argument shows the same inequality when $\gamma'(t')\ll \gamma(t)$. This immediately implies
    \[
    \overline{\ma}_x(\xi,\xi') \leq \ma(\xi,\xi').
    \]
    On the other hand, $S:=I^-(\xi) = I^-(\xi')$ by \autoref{pop:different-pasts-infinite-angle}. If $x,y\in S$ then there exist $c,\lambda>0$  such that 
    \[
    \tau(\xi_y(t-c),\xi_x(t))\geq \lambda \quad\text{and}\quad \tau(\xi'_x(t'),\xi'_y(t'+c)) \geq \lambda 
    \]
    for any $t,t'>0$ sufficiently large. In particular, if $\xi_x(t)\ll\xi'_x(t')$ then, by the previous inequalities, the transitivity of $\ll$, and the reverse triangle inequality for $\tau$,
    \[
    \tau(\xi_x(t),\xi'_x(t'))\leq  \tau(\xi_y(t-c),\xi'_y(t'+c)) - 2\lambda. 
    \]
    A direct computation using the Lorentzian law of cosines yields 
    \[
    \cosh\widetilde{\ma}_y(\xi_y(t-c),\xi'_y(t'+c))\leq \frac{tt'}{(t-c)(t'+c)}\cosh\widetilde{\ma}_x(\xi_x(t),\xi'_x(t')) + \frac{(t'-t)c+c^2}{(t-c)(t'+c)}.
    \] 
    An analogous inequality holds when $\xi'_x(t')\ll \xi_x(t)$. By monotonicity of the hyperbolic cosine and monotonicity of comparison angles, these inequalities imply
    \[
    \overline{\ma}_y(\xi,\xi') \leq \overline{\ma}_x(\xi,\xi')
    \]
    when we let $t,t'\to \infty$. Notice that to be able to let $t,t'$ diverge we are again using that $\xi$ and $\xi'$ generate the same past. The reverse inequality is obtained analogously, yielding that $\overline{\ma}(\xi,\xi'):=\overline{\ma}_x(\xi,\xi')$ is independent of $x$. 
    
    Finally, by monotonicity of comparison angles,
    \[
    \ma_x(\xi,\xi')\leq \overline{\ma}_x(\xi,\xi') = \overline{\ma}(\xi,\xi')
    \]
    for any $x$ such that $\xi_x$ and $\xi'_x$ exist, and the result follows.
    
    The last part of the claim follows from \eqref{eq:properties of angle 1} and this result.\qedhere
\end{enumerate}
\end{proof}

The following proposition can be interpreted as follows: if the angular distance between two future-directed timelike rays $\xi$, $\xi'$ is finite, then for any point $x$ in the common past of $\xi$ and $\xi'$, one can choose $t,t'$ large enough so that the timelike triangle $\triangle x\xi_x(t)\xi'_x(t')$ becomes ``similar'' to a timelike triangle in Minkowski space with past angle $\ma(\xi,\xi')$. It is analogous to \cite[Proposition~II.9.8(4)]{BH} and it will be instrumental in proving \autoref{thm:main2}.

\begin{pop}\label{pop:asymptotic time separation}
Let $Y$ be a proper, globally hyperbolic, strongly causal, locally causally closed and regular \LpLS satisfying $\CBA(0)$ globally. Let $\xi,\xi'\in\bd^+Y$ be such that $\ma(\xi,\xi') < \infty$. Then there exists $c_o\geq 1$ such that for any $x\in I^-(\xi)=I^-(\xi')$, any $c> c_o$, and any $s>0$, we have $\xi_x(s)\ll \xi'_x(cs)$ and 
\[
\lim_{s\to \infty} \frac{1}{s^2}\tau(\xi_x(s),\xi'_x(cs))^2 = 1 + c^2 - 2c\cosh\ma(\xi,\xi').
\]
\end{pop}

\begin{proof}
The hypothesis on $\xi$ and $\xi'$ is equivalent to $I^-(\xi)=I^-(\xi')$ and $\alpha_o:=\overline{\ma}_x(\xi,\xi') < \infty$ for any $x\in I^-(\xi)$, where $\overline{\ma}$ is defined as in \autoref{pop:properties of angle}. In particular, if we fix such an $x$ then $\smash{\widetilde{\ma}_x(\xi_x(s),\xi'_x(s'))}\leq \alpha_o$ for any $s,s'>0$ such that $\xi_x(s)\ll \xi'_x(s')$. By \autoref{lem:minkowski-situation-bounded-angles} there is a constant $c_o\geq 1$ (only depending on $\alpha_o$) such that we can assume $s' = cs$ for some $c>c_o$, and the first part of the claim follows. 

For the second part, simply observe that 
\[
\frac{1}{s^2}\tau(\xi_x(s),\xi'(cs))^2 = 1 + c^2 - 2c\cosh \widetilde{\ma}_x(\xi_x(s),\xi'_x(cs))
\]
by the definition of comparison angles. This, combined with \autoref{pop:properties of angle}, yields the result.
\end{proof}

\begin{pop}\label{pop:completeness}
Let $Y$ be a proper, globally hyperbolic, strongly causal, locally causally closed and regular \LpLS satisfying $\CBA(0)$ globally. Then the \textit{extended} metric space $(\bd^+Y,\ma)$ is complete. 
\end{pop}

\begin{proof}
Let $(\xi_n)_{n\in\mathbb{N}}$ be a Cauchy sequence in $\bd^+Y$. By \autoref{pop:different-pasts-infinite-angle}, there exists $N\in\N$ such that all ideal points $\xi_n$ with $n\geq N$ have the same timelike past, which we call $A$. Let us fix $x\in A$. By \autoref{pop:properties of angle}(\ref{item:alternative-definition-angular}), if $\varepsilon>0$ then for sufficiently large $m,n$,  
\begin{equation}\label{eq:angular metric is complete}
\overline{\ma}_x(\xi_{m},\xi_{n})<\varepsilon. 
\end{equation}

Let us denote $\gamma_n$ the representatives of $\xi_n$ starting at $x$ and their reparametrizations by $d$-arclength by $\tilde{\gamma}_n$. \autoref{thm:limit curve theorem} ensures that there is a subsequence $(\tilde{\gamma}_{n_k})$ and a future-directed causal curve $\tilde{\gamma}$ starting at $x$ such that $\tilde{\gamma}_{n_k}\to \tilde{\gamma}$ locally uniformly. Now, \autoref{lem:bounded comparison angles imply timelike limit} ensures that $\tilde{\gamma}$ is timelike. Moreover, $\tilde{\gamma}$ is maximizing by the continuity of $\tau$ (although only the lower semicontinuity is needed) and the upper semicontinuity of $L_\tau$ (see, for instance, \cite[Theorem~2.23]{beran-ohanyan-rott-solis2023}). 

Being $\tilde{\gamma}$ timelike, we can reparametrize it by $\tau$-length and call the new curve $\gamma$. Then \cite[Lemmata~5.7-8]{barton-beran-che-gieger-roehrig-rott} ensure that also $\gamma_{n_k}\to \gamma$, i.e., there is also pointwise convergence in the parametrization by $\tau$-arclength, and that $\gamma$ is a timelike geodesic ray. Let us call $\xi:=\gamma(\infty)$.

Let $m=n_k$ and fix $t'> t > 0$. Since $\gamma$ is a future-directed timelike curve, then 
\[I^-(\gamma(t'-\delta),\gamma(t'+\delta))\] 
is a neighborhood of $\gamma(t')$, and since $\gamma_m\to \gamma$ locally uniformly, it follows that 
\[\gamma_m(t') \in I^-(\gamma(t'-\delta),\gamma(t'+\delta))\] for sufficiently large $m$. Analogously, 
\[\gamma_n(t)\in I(\gamma(t-\delta),\gamma(t+\delta))\]
for sufficiently large $n$. Choosing $\delta>0$ small enough, it follows that $\gamma_n(t)\ll \gamma_m(t')$ and  $\gamma_n(t)\ll \gamma(t')$ for sufficiently large $m,n$. Thus, fixing such $n$ and letting $m\to\infty$, by continuity of $\tau$ and the Lorentzian law of cosines, we get
\[
\widetilde{\ma}_{x}(\gamma_n(t),\gamma(t')) = \lim_{m\to\infty} \widetilde{\ma}_{x}(\gamma_n(t),\gamma(t')).
\]
By \eqref{eq:angular metric is complete}, we can assume $\widetilde{\ma}_{x}(\gamma_n(t),\gamma(t))<\varepsilon$ for $m,n$ sufficiently large, yielding
\[
\widetilde{\ma}_{x}(\gamma_n(t),\gamma(t')) \leq \varepsilon
\]
for such $n$. Since $t'>t>0$ are arbitrary, it follows that
\[
\ma(\xi,\xi_n)=\overline{\ma}_x(\xi,\xi_n) \leq \varepsilon.
\]
Therefore $\xi_n\to \xi$ with respect to the angular metric.
\end{proof}

The following result is analogous to \cite[Lemma~II.9.16]{BH}.

\begin{pop}\label{pop:lower semicontinuity of angle}
Let $Y$ be a proper, globally hyperbolic, strongly causal, locally causally closed and regular \LpLS satisfying $\CBA(0)$ globally. Consider $\xi,\eta\in\bd^+Y$ such that $\ma (\xi, \eta) < \infty$ and a point $p$ in their (common) timelike past. Let $x_n,y_n\in Y$ be two sequences such that $p\ll x_n \ll y_n$ for all $n$, and $x_n\to \xi$ and $y_n\to \eta$ in the cone topology. Then
\[
\liminf_{n\to \infty} \widetilde{\ma}_{p}(x_n,y_n) \geq \ma(\xi,\eta).
\]
\end{pop}
\begin{proof}
Let $\varepsilon>0$, $R>0$ and $p\in I^-(\xi)=I^-(\eta)$. By \autoref{pop:asymptotic time separation}, we know that there is some $c>1$ such that $\xi_p(t)\ll\eta_p(ct)$ for all $t$. So consider a sequence $t_n\to \infty$ and define two sequences $x_n':=\xi_p(t_n)\ll \eta_p(2ct_n)=:y_n'$. Notice that the quotient between $\tau(p,y_n')$ and $\tau(p,x_n')$ is precisely $2c$.

By hypothesis, for each $n\in \N$ there exists some $m=m_n\in \N$ (depending on $x_n'$, $y_n'$, $\varepsilon$ and $R$) such that $x_m\in V_{x_n',\varepsilon,R}(\xi)$ and $y_m\in V_{y_n',\varepsilon,R}(\eta)$. By \autoref{thm:triangle inequality for angles} and \autoref{thm:monotonicity of comparison angles},
\[
\widetilde{\ma}_{x_n'}(p,x_m) \leq \ma_{x_n'}(p,x_m) \leq \ma_{x_n'}(x_m,\xi) < \varepsilon,
\]
and by \autoref{lem:sum of angles of triangle}, we obtain
\[
\widetilde{\ma}_{p}(x_m,x_n') < \varepsilon.
\]
Analogously, $\widetilde{\ma}_{p} (y_m,y_n') < \varepsilon$.

Straightforward computations involving the Lorentzian law of cosines in the comparison triangles {$\overline{\triangle} (p,x_m,x_n')$ and $\overline{\triangle} (p,y_m,y_n')$} show that, if $x_m''\in [p,x_m]$ and $y_m''\in [p,y_m]$ are such that 
\begin{equation}\label{eq:xm'' and xn'}
\tau(p,x_m'') = \tau(p,x_n')=t_n\quad \text{ and } \quad \tau(p,y_m'') = \tau(p, y_n')=2ct_n, 
\end{equation}
then there exist $p_m,q_m\in \xi_{p}$ and $r_m,s_m\in \eta_{p}$ such that (see \autoref{fig:lsc})
\[
x_m''\in J(p_m,q_m),\quad y_m''\in J(r_m,s_m), \quad p_m\ll x_n'\ll q_m, \quad r_m\ll y_n'\ll s_m,
\]
and moreover 
\[
\tau(p_m,q_m) < \delta\tau(p,x_n') \quad \text{and} \quad \tau(r_m,s_m) < \delta\tau(p,y_n'),
\]
for some $\delta=\delta(\varepsilon)>0$ such that $\lim_{\varepsilon\to 0} \delta(\varepsilon) = 0$. Furthermore, by taking $\varepsilon>0$ sufficiently small, we can guarantee that $q_m\leq r_m$.

\begin{figure}
\def \globalscale {0.80000}
\begin{tikzpicture}[y=0.9cm, x=0.9cm, yscale=\globalscale,xscale=\globalscale, every node/.append style={scale=\globalscale}, inner sep=0pt, outer sep=0pt]
  \path[draw=black,line width=0.02cm] (3.7, 25.5).. controls (6.0, 23.1) and (7.2, 15.2) .. (7.2, 15.2);
  \path[draw=black,line width=0.02cm] (7.2, 15.2) .. controls (9.0, 22.8) and (12.0, 25.3) .. (12.0, 25.3);
  \path[draw=black,line width=0.02cm] (7.2, 15.2).. controls (6.2, 18.6) and (4.9, 21.5) .. (4.9, 21.5);
  \path[draw=black,line width=0.02cm] (7.2, 15.2).. controls (9.0, 19.8) and (11.5, 23.7) .. (11.5, 23.7);
  \path[draw=black,line width=0.02cm,dashed] (9.2, 21.1).. controls (8.0, 20.5) and (6.2, 19.9) .. (6.2, 19.9);
  
  \path[draw=black,line width=0.02cm,rotate around={58.5:(0.0, 29.7)}] (-6.45, 19.3) rectangle (-5.1, 17.95);
  \path[draw=black,line width=0.02cm,rotate around={17.8:(0.0, 29.7)}] (6.1, 20.05) rectangle (7.5, 18.7);
  
  \path[fill=black,line width=0.02cm] (4.9, 21.5) circle (2pt);
  \path[fill=black,line width=0.02cm] (6.0, 18.8) circle (2pt);
  \path[fill=black,line width=0.02cm] (6.2, 19.9) circle (2pt);
  \path[fill=black,line width=0.02cm] (6.43, 19.0) circle (2pt);
  \path[fill=black,line width=0.02cm] (6.65, 18.1) circle (2pt);
  \path[fill=black,line width=0.02cm] (7.2, 15.2) circle (2pt);
  \path[fill=black,line width=0.02cm] (9.2, 21.1) circle (2pt);
  \path[fill=black,line width=0.02cm] (9.65, 22.0) circle (2pt);
  \path[fill=black,line width=0.02cm] (10.1, 22.8) circle (2pt);
  \path[fill=black,line width=0.02cm] (10.26, 21.6) circle (2pt);
  \path[fill=black,line width=0.02cm] (11.5, 23.7) circle (2pt);

  \node[text=black,anchor=south west,line width=0.02cm] (text16) at (4.6, 21.8){$x_m$};
  \node[text=black,anchor=south west,line width=0.02cm] (text17) at (11.7, 24.0){$y_m$};
  \node[text=black,anchor=south west,line width=0.02cm] (text18) at (9.6, 23.0){$s_m$};
  \node[text=black,anchor=south west,line width=0.02cm] (text19) at (8.6, 21.1){$r_m$};
  \node[text=black,anchor=south west,line width=0.02cm] (text20) at (10.6, 21.7){$y''_m$};
  \node[text=black,anchor=south west,line width=0.02cm] (text21) at (9.1, 22.1){$y'_n$};
  \node[text=black,anchor=south west,line width=0.02cm] (text22) at (6.6, 19.0){$x'_n$};
  \node[text=black,anchor=south west,line width=0.02cm] (text23) at (5.3, 18.2){$x''_m$};
  \node[text=black,anchor=south west,line width=0.02cm] (text24) at (6.9, 17.9){$p_m$};
  \node[text=black,anchor=south,line width=0.02cm] (text25) at (6.4, 20.1){$q_m$};
  \node[text=black,anchor=south west,line width=0.02cm] (text27) at (3.8, 25.8){$\xi_{p}$};
  \node[text=black,anchor=south west,line width=0.02cm] (text28) at (11.9, 25.7){$\eta_{p}$};
  \node[text=black,anchor=south west,line width=0.02cm] (text29) at (7.2, 14.6){$p$};
\end{tikzpicture}
\caption{}
\label{fig:lsc}
\end{figure}

Thus, by the reverse triangle inequality, 
\begin{equation}\label{eq:lsc-small diamonds}
\tau(x_n',y_n'), \tau(x_m'',y_m'') \in [\tau(q_m,r_m), \tau(p_m,s_m)].
\end{equation}
By the Lorentzian law of cosines and \eqref{eq:xm'' and xn'},
\begin{equation}\label{eq:lsc-comparing comparison angles1}
|\cosh\widetilde{\ma}_{p}(x_m'',y_m'') - \cosh\widetilde{\ma}_{p}(x_n',y_n')| = \frac{|\tau(x_n',y_n')^2-\tau(x_m'',y_m'')^2|}{2\tau(p,x_n')\tau(p,y_n')},
\end{equation}
and from \eqref{eq:xm'' and xn'}, \eqref{eq:lsc-small diamonds}, \eqref{eq:lsc-comparing comparison angles1}, it follows that
\begin{equation}\label{eq:lsc-comparing comparison angles2}
|\cosh\widetilde{\ma}_{p}(x_m'',y_m'') - \cosh\widetilde{\ma}_{p}(x_n',y_n')| \leq \frac{\tau(p_m,s_m)^2-\tau(q_m,r_m)^2}{2\tau(p,x_n')\tau(p,y_n')}.
\end{equation}
As $n\to \infty$, the right hand side of \eqref{eq:lsc-comparing comparison angles2} is asymptotically bounded by 
\[
2\delta\left(1-\frac{1}{c}\right)(1+\cosh\ma(\xi,\eta)).
\]
By letting $\varepsilon\to 0$, we obtain that
\[
\lim_{m\to\infty}\widetilde{\ma}_{p}(x_m'',y_m'') = \lim_{n\to\infty}\widetilde{\ma}_{p}(x_n',y_n') = \ma(\xi,\eta),
\]
where the last equation follows from \autoref{pop:properties of angle}.

Finally, by monotonicity of comparison angles, we have
\[
\widetilde{\ma}_{p}(x_m'',y_m'') \leq \widetilde{\ma}_{p}(x_m,y_m)
\]
which yields 
\[
\lim_{n\to\infty}\widetilde{\ma}_{p}(x_m'',y_m'') \leq \liminf_{m\to\infty}\widetilde{\ma}_{p}(x_m,y_m) 
\]
and the result follows.
\end{proof}

\begin{lem}\label{lem:angle bisectors converge to midpoint ray}
Let $Y$ be a proper, globally hyperbolic, strongly causal, locally causally closed and regular \LpLS satisfying $\CBA(0)$ globally. Let $\xi,\eta \in \bd^+Y$ such that $\ma(\xi,\eta) < \infty$. Let $p \in Y$ and representatives $\xi_p, \eta_p$ of $\xi, \eta$ with $\xi_p (0) = p = \eta_p (0)$. Let $\{x_n\} \subset \xi_p$ and $\{y_n\} \subset \eta_p$ with $x_n \ll x_{n+1}$, $y_n \ll y_{n+1}$, such that $x_n \to \xi, y_n \to \eta$ in the cone topology and $x_n \ll y_n$ for every $n$. Let $m_n \in [x_n, y_n]$ be such that $\widetilde{\ma}_{p} (x_n, m_n) = \widetilde{\ma}_{p} (m_n,y_n) \leq C$ for some constant $C>0$. Then, up to a subsequence, $m_n$ converges in the cone topology to a timelike ideal point $\mu$.
\end{lem}
\begin{proof}
Since $p\ll x_n \ll m_n$, we can take $\tilde{\alpha}_n$ to be the unique future-directed timelike geodesic segment from $p$ to $m_n$. Assume that the $\tilde{\alpha}_n$ are parametrized by $d$-arclength. Since $Y$ is a locally causally closed and proper \LpLS then \autoref{thm:limit curve theorem} lets us assume, up to a subsequence, that $\tilde{\alpha}_{n}$ converges locally uniformly to a future-directed causal curve $\tilde{\gamma}$. Moreover, by \autoref{lem:bounded comparison angles imply timelike limit} it follows that the limit curve is timelike and is also a maximal geodesic. As before, \cite[Lemma~5.7]{barton-beran-che-gieger-roehrig-rott} ensures that, calling $\gamma$ the reparametrization of $\tilde{\gamma}$ by $\tau$-length, there is also pointwise convergence in the $\tau$-parametrizations $\alpha_n$ of the curves $\tilde{\alpha}_n$ in the domain of $\gamma$.

Moreover, up to a subsequence, we can assume that $m_n\ll m_{n+1}$. Indeed, as $\ma(\xi,\eta)<\infty$ we know that $\xi$ and $\eta$ have the same timelike past (\autoref{pop:different-pasts-infinite-angle}). Therefore for each $m_n$ there exists some $x\in\xi_p$ such that $m_n\ll x$. But as $\tau(p,x_l)\to\infty$, all but finitely many terms of the sequence $\{x_l\}$ satisfy $m_n\ll x_l\ll m_l$. Consequently, we can pass to a subsequence (which we do not relabel for simplicity) to obtain that $m_n\ll m_{n+1}$ for every $n\in \N$.

In the following step we will prove that the sequence $L_n=d(p,m_n)$ is unbounded. Assume, by contradiction, that there exists $M>0$ such that $L_{n}\leq M$ for all $n\in \N$. This implies that $m_n\in \bar{B}_M(p)$ which, by properness of $d$, is compact. Extract a convergent subsequence $m_{n_k}\to z\in \bar{B}_M(p)$. As $\tau$ is finite and continuous because $Y$ is $\CBA(0)$ globally, and therefore the whole space is a comparison neighborhood, $\tau(p,z)$ is finite. Therefore, for every $k\in \N$ we have $x_{n_k}\ll m_{n_k}\ll z$, where we used that the sequence of the $m_n$'s is monotone increasing with respect to $\ll$. As a consequence, we have that all terms satisfy $x_{n_k}\in J(p,z)$ and, therefore, $\tau(p,x_{n_k})\leq \tau(p,z)$. However, this is a contradiction, as $\tau(p,x_{n_k})$ is unbounded. In particular, we deduce that $\tilde{\gamma}$ is inextendible and its $d$-length is infinite \cite[Lemma~3.12, Theorem~3.14]{kunzinger-saemann2018}.

We will now prove, following ideas from the proof of \cite[Lemma~5.8]{barton-beran-che-gieger-roehrig-rott}, that if the $\tau$-length of $\gamma$ is finite, then also its $d$-length must be finite, yielding a contradiction. To that end, assume that $L_\tau(\gamma)=a<\infty$. Under this assumption we will obtain a point $\eta_p(a')$ such that $\gamma(t)\leq\eta_p(a')$ for every $t\in[0,a)$. Therefore, the image of $\gamma$ will be contained in the compact diamond $J\bigl(p,\eta_p(a')\bigr)$ and will thus have finite $d$-length.

As $\tau(p,m_n)\to \infty$, there exists $N\in\N$ such that $L_\tau(\alpha_n)>a$, for all $n\geq N$. For $n$ large enough, 
take the point $r_n:=\alpha_n(a)$, i.e., $r_n\in [p,m_n]$ such that $\tau(p,r_n)=a$.  Now, for each $n$ large enough, consider the comparison triangles $\bar{\triangle}(\bar{p},\bar{r}_n,\overline{y_n})$
in $\R^{1,1}$. By monotonicity of comparison angles, $\widetilde{\ma}_p(r_n,y_n)\leq \widetilde{\ma}_p(m_n,y_n)$, and the right-hand side is uniformly bounded. Thus, by \autoref{lem:minkowski-situation-bounded-angles}, there exists some value $c_0\geq 1$ such that whenever $a'>c_0\,a$, one has $\bar{r}_n\ll\overline{\eta_p(a')}$ for all $n$ large enough, where $\overline{\eta_p(a')}$ is the comparison point for $\eta_p(a')$ in the corresponding side of the comparison triangle. Now define $b_n=\max\{\tau(p,y_n),a'\}$ and the new comparison triangles $\bar{\triangle}(\bar{p},\bar{r}_n,\overline{\eta_p(b_n)})$. The curvature bound ensures that $\tau(r_n,\eta_p(a'))\geq \tau_{\R^{1,1}}(\bar{r}_n,\overline{\eta_p(a')})>0$, from where we deduce that $r_n\ll\eta_p(a')$ for all $n$ large enough.

As a consequence, defining $a':=c_0\,a+1$, we have that all the segments $[p,r_n]$ are contained in the compact diamond $J\bigl(p,\eta_p(a')\bigr)$. By continuity of $\tau$ we deduce that also $\gamma$, which is the pointwise limit of the (half open) segments $[p,r_n)$, is contained in the diamond, arriving at the desired contradiction. To sum up, the curve $\gamma$ is defined on $[0,\infty)$ and therefore it is a timelike geodesic ray. Let us call $\mu:=\gamma(\infty)$ the corresponding ideal point.

Finally, to prove that $m_{n} \to \mu$ in the cone topology, take $t\geq 0$, $\varepsilon >0$, $R> 0$ and let $V_{\gamma(t),\varepsilon, R} (\mu)$ be the corresponding basic neighborhood of $\gamma$ in the cone topology (see \autoref{cor:nice basis of nbhds}). On the one hand, it is clear that $\tau(\gamma(t), m_n)\to \infty$ as $n\to \infty$. On the other hand, for each $s>0$ and $n$ large enough, consider the timelike triangle $\triangle\bigl(\gamma(t),\alpha_n(t+s),m_n\bigr)$. By the law of cosines and the fact that $\alpha_n(t+s)\to\gamma(t+s)$, a direct computation shows that $\widetilde{\ma}_{\gamma(t)}(\alpha_n(t+s),m_n)\to 0$ as $n\to\infty$. By the monotonicity of comparison angles, also $\ma_{\gamma(t)}\bigl(\alpha_n(t+s),m_n\bigr)\to 0$. By \cite[Corollary~4.3]{eros-gieger-2025}, $[\gamma(t),\alpha_n(t+s)]\to \gamma|_{[t,t+s]}$ uniformly as $n\to \infty$. As a consequence, \autoref{thm:continuity of angles} ensures that $\ma_{\gamma(t)}\bigl(\gamma,\alpha_n(t+s)\bigr)\to0$. Finally, by the triangle inequality for the angle between curves (\autoref{thm:triangle inequality for angles}) one deduces that $\ma_{\gamma(t)}\bigl(\gamma,m_n\bigr)\to0$. Consequently, a tail of the sequence $m_n$ is contained in the basic neighborhood $V_{\gamma(t),\varepsilon,R}(\mu)$, from where the result follows.
\end{proof}

\begin{pop}\label{pop:midpoints}
Let $Y$ be a proper, globally hyperbolic, strongly causal, locally causally closed and regular \LpLS satisfying $\CBA(0)$ globally. Then for any $\xi,\eta \in \bd^+Y$ such that $\ma(\xi,\eta) < \infty$ there exists $\mu \in \bd^+Y$ such that 
\[
\ma(\xi,\mu)=\ma(\eta,\mu)=\frac{1}{2}\ma(\xi,\eta).
\]
In other words, $\mu$ is a \emph{midpoint} between $\xi$ and $\eta$.
\end{pop}
\begin{proof}
By \autoref{pop:different-pasts-infinite-angle}, the ideal points $\xi,\eta$ generate both the same timelike past. Let us take a point $x_o$ in such a past and consider the representatives $\xi_o, \eta_o$ of $\xi, \eta$ with $\xi_o (0) = x_o = \eta_o (0)$. Take sequences $\{x_n\} \subset \xi_o$ and $\{y_n\} \subset \eta_o$ such that for every $n\in\N$, one has $x_n \ll x_{n+1}$, $y_n \ll y_{n+1}$ and $x_n \ll y_n$, and such that $x_n \to \xi$, $y_n \to \eta$ in the cone topology. This is possible by \autoref{pop:asymptotic time separation}, since  $\ma (\xi, \eta) < \infty$ implies there exists $c>0$ such that $\xi_o(t)\ll\eta_o(ct)$ for all $t$. Thus, one may take an increasing sequence $t_n\to\infty$ and define $x_n := \xi_o (t_n) \ll \eta_o (ct_n) =: y_n$.

For each $n$, \autoref{pop:properties of angle}(\ref{item:alternative-definition-angular}) ensures that $\widetilde{\ma}_{x_o} (x_n, y_n) \leq \ma(\xi,\eta)$. Additionally, following ideas from the proof of \cite[Theorem~3.8]{barton_space_2026}, we can show that there exists $m_n \in [x_n, y_n]$ such that
\begin{equation}\label{midpoints 1}
\widetilde{\ma}_{x_o} (x_n, m_n) = \widetilde{\ma}_{x_o} (m_n, y_n) \leq \frac{1}{2} \widetilde{\ma}_{x_o} (x_n, y_n).
\end{equation}
The argument for this last part is classical in triangle comparison; we reproduce it here for the sake of completeness. By continuity of $\tau$ (implied by the global $\CBA(0)$ condition) and the $\tau$-arclength parametrization of $[x_n,y_n]$ we can guarantee the existence of $m_n\in[x_n,y_n]$ realizing the first equation in \eqref{midpoints 1}. On the other hand, letting $\triangle \bar{x}_o\bar{x}_n\bar{m}_n$, $\triangle \bar{x}_o\bar{m}_n\bar{y}_n$ and $\triangle \tilde{x}_o\tilde{x}_n\tilde{y}_n$ be comparison triangles in $\mathbb{R}^{1,1}$ for $\triangle x_o x_n m_n$, $\triangle x_o m_n y_n$ and $\triangle x_o x_n y_n$, respectively,  such that $\bar{x}_n$ and $\bar{y}_n$ are in opposite sides of $\mathbb{R}^{1,1}$ with respect to $[\bar{x}_o,\bar{m}_n]$, we see that
\[
\overline{\tau}(\tilde{x}_n,\tilde{y}_n) = \overline{\tau}(\bar{x}_n,\bar{m}_n)+\overline{\tau}(\bar{m}_n,\bar{y}_n) \leq \overline{\tau}(\bar{x}_n,\bar{y}_n) 
\]
which yields
\[
\widetilde{\ma}_{x_o}(x_n,y_n) \geq \ma_{\bar{x}_o}(\bar{x}_n,\bar{y}_n)  = \ma_{\bar{x}_o}(\bar{x}_n,\bar{m}_n)  + \ma_{\bar{x}_o}(\bar{m}_n,\bar{y}_n) = 2\widetilde{\ma}_{x_o} (m_n, y_n) 
\]
and the inequality in \eqref{midpoints 1} follows.

By \autoref{lem:angle bisectors converge to midpoint ray}, a subsequence (which we do not relabel) of $m_n$ converges in the cone topology to a timelike ideal point $\mu$. By \eqref{midpoints 1} together with Propositions~\ref{pop:lower semicontinuity of angle} and \ref{pop:properties of angle},  
\[
\ma (\xi, \mu ) \leq \liminf \widetilde{\ma}_{x_o} (x_n, m_n) \leq \frac{1}{2} \lim_{n\to \infty} \widetilde{\ma}_{x_o} (x_n, y_n) = \frac{1}{2} \ma (\xi, \eta).
\]

Analogously,
\[
\ma (\eta, \mu) \leq \frac{1}{2} \ma(\xi, \eta),
\]
and the triangle inequality gives us
\[
\ma (\xi, \eta) \leq \ma (\xi, \mu) + \ma(\eta, \mu) \leq \ma (\xi, \eta),
\]
which implies that $\ma( \xi, \mu) = \frac{1}{2} \ma(\xi, \eta) = \ma (\eta, \mu)$.
\end{proof}

\begin{cor}\label{cor:geodesicity}
Let $Y$ be a proper, globally hyperbolic, strongly causal, locally causally closed and regular \LpLS satisfying $\CBA(0)$ globally. Then each component of finiteness of $\bd^+Y$ is geodesic.
\end{cor}
\begin{proof}
This follows from Propositions~\ref{pop:completeness} and \ref{pop:midpoints}, and the fact that complete metric spaces where midpoints exist between any pair of points are geodesic spaces (see, for example, \cite[Theorem~2.4.16]{BBI}).
\end{proof}

\begin{pop}\label{pop:almost intermediate points are close}
Let $Y$ be a proper, globally hyperbolic, strongly causal, locally causally closed and regular \LpLS satisfying $\CBA(0)$ globally. Let $x\ll m\ll y$ be such that $m$ is in a timelike geodesic $[x,y]$ and let $\lambda = \frac{\tau(x,m)}{\tau(x,y)}\in (0,1)$. Then, for all $0<\varepsilon<\min\{\lambda,1-\lambda\}$, there exists $\delta(\varepsilon,\lambda)>0$ such that $\lim_{\varepsilon\to 0}\delta(\varepsilon,\lambda) = 0$ and  the following implication holds: if $m'\in X$ is such that $x\ll m'\ll y$ and 
\begin{equation}\label{eq:almost intermediate points0}
\max\{|\tau(x,m')-\lambda\tau(x,y)|, |\tau(m',y)-(1-\lambda)\tau(x,y)| \}< \varepsilon \tau(x,y),
\end{equation}
then there exist $p,q\in [x,y]$ such that $p\ll m\ll q$, $\tau(p,q) < \delta(\varepsilon,\lambda)\tau(x,y)$, and
\[
m' \in J(p,q).
\]
\end{pop}

\begin{proof}
Fix $0 < \varepsilon < \min\{\lambda,1-\lambda\}$ and let $x\ll m'\ll y$ be such that \eqref{eq:almost intermediate points0} holds. Consider a comparison triangle $\triangle(\bar{x},\bar{m}',\bar{y})$ for $\triangle(x,m',y)$ in the Minkowski plane, and let $\bar{m}\in[\bar{x},\bar{y}]$ be the comparison point corresponding to $m$. Using the Lorentzian law of cosines, one can see that if $\bar{p},\bar{q}\in [\bar{x},\bar{y}]$ are such that $\bar{p}\ll\bar{m}\ll\bar{q}$, then $\bar{m}'\in J(\bar{p},\bar{q})$ if and only if
\begin{equation}\label{eq:almost intermediate points1}
\tau(\bar{x},\bar{p}) \leq \tau(x,m')(\cosh\widetilde{\ma}_x(m',y) - \sinh\widetilde{\ma}_x(m',y))
\end{equation}
and
\begin{equation}\label{eq:almost intermediate points2}
\tau(\bar{x},\bar{q}) \geq \tau(x,m')(\cosh\widetilde{\ma}_x(m',y) + \sinh\widetilde{\ma}_x(m',y)).
\end{equation}
On the other hand, another application of the Lorentzian law of cosines combined with \eqref{eq:almost intermediate points0} yields
\begin{equation}\label{eq:almost intermediate points3}
|\tau(x,m')\cosh\widetilde{\ma}_x(m',y) - \lambda\tau(x,y)| < \varepsilon\tau(x,y).
\end{equation}
Indeed, from \eqref{eq:almost intermediate points0} and the fact that $0<\varepsilon<\lambda$, it follows that
\[
(\lambda-\varepsilon)^2\tau(x,y)^2<\tau(x,m')^2<(\lambda+\varepsilon)^2\tau(x,y)^2,
\]
and analogously
\[
((1-\lambda)-\varepsilon)^2\tau(x,y)^2<\tau(m',y)^2< ((1-\lambda)+\varepsilon)^2\tau(x,y)^2.
\]
These inequalities imply
\[
2(\lambda-\varepsilon)\tau(x,y)^2 < \tau(x,m')^2+\tau(x,y)^2-\tau(m',y)^2 < 2(\lambda+\varepsilon)\tau(x,y)^2
\]
from which \eqref{eq:almost intermediate points3} easily follows by the Lorentzian law of cosines. From this, and analogous computations, we obtain 
\[
\tau(x,m')\sinh\widetilde{\ma}_x(m',y) < 2\left(\lambda\varepsilon\right)^{1/2}\tau(x,y). 
\]
In particular, if $\bar{p}_0,\bar{q}_0\in [\bar{x},\bar{y}]$ are such that equalities are attained in \eqref{eq:almost intermediate points1} and \eqref{eq:almost intermediate points2}, then 
\begin{equation}\label{eq:bar p and bar q}
\tau(\bar{p}_0,\bar{q}_0) < 4\left(\varepsilon \lambda\right)^{1/2}\tau(x,y)=:\delta(\varepsilon,\lambda)\tau(x,y).
\end{equation}
Finally, choose $p,q\in [x,y]$ whose comparison points in the previous configuration are slight perturbations of $\bar{p}_0,\bar{q}_0$, and such that \eqref{eq:almost intermediate points1}, \eqref{eq:almost intermediate points2}, and \eqref{eq:bar p and bar q} (with $\bar{p}$ and $\bar{q}$ replacing $\bar{p}_0$ and $\bar{q}_0$, respectively, which is possible by continuity of $\tau$ and the fact that the inequality is open) hold simultaneously. By the curvature assumption on $Y$, the claim follows (see \cite[Remark~1.21]{beran-saemann2023}).
\end{proof}

\begin{thm}\label{thm:main2}
Let $Y$ be a proper, globally hyperbolic, strongly causal, locally causally closed and regular \LpLS satisfying $\CBA(0)$ globally. Then $(\bd^+Y,\ma)$ every component of finiteness of $\ma$ is a global $\mathrm{CAT}(-1)$ space.
\end{thm}

\begin{proof}
Let $\xi_0,\xi_1,\xi_2,\mu\in \bd^+Y$ be such that $\ma(\xi_i,\xi_j) < \infty$ for all $i,j$, and such that 
\[\ma(\xi_0,\xi_1) = 2\ma(\xi_0,\mu) = 2\ma(\xi_1,\mu).\] 
Let $\triangle \overline{\xi}_0 \overline{\xi}_1 \overline{\xi}_2$ be a comparison triangle for $\triangle \xi_0\xi_1\xi_2$ in the hyperbolic plane, $\mathbb{H}^2$, and $\overline{\mu}$ be the comparison point corresponding to $\mu$. We want to show
\[
\ma(\xi_2,\mu) \leq d_{\mathbb{H}^2}(\overline{\xi}_2,\overline{\mu}).
\]
In order to achieve this, fix some $p\in Y$ and, abusing the notation, denote by $\xi_0,\xi_1,\xi_2,\mu$ the rays starting at $p$ and representing the corresponding points in $\bd^+Y$. By the triangle inequality for the angular metric in $\bd^+Y$, $\ma(\xi_2,\mu)<\infty$. Then, by \autoref{pop:asymptotic time separation}, for sufficiently large $c>1$ we have $\xi_2(s)\ll \mu(cs)$ for all $s\geq 0$ and  
\begin{equation}\label{eq:asymptotic time separation curvature thm}
\lim_{s\to\infty} \frac{1}{s^2} \tau(\xi_2(s),\mu(cs))^2= 1 + c^2 - 2c\cosh\ma(\xi_2,\mu).
\end{equation}
On the other hand, by identifying $\mathbb{H}^2$ with the subset of future-directed timelike unit vectors in the three-dimensional Minkowski space, $\mathbb{R}^{2,1}$, we obtain 
\begin{equation}\label{eq:distance in minkowski curvature thm}
\tau_{\mathbb{R}^{2,1}}(\overline{\xi}_2,c\overline{\mu})^2 = 1 + c^2 - 2c\cosh d_{\mathbb{H}^2}(\overline{\xi}_2,\overline{\mu}).
\end{equation}
Therefore, we only need to compare the left-hand sides of equations \eqref{eq:asymptotic time separation curvature thm} and \eqref{eq:distance in minkowski curvature thm}.

By the law of cosines in $\mathbb{R}^{2,1}$ and the fact that $\overline{\mu}$ is in the hyperbolic geodesic joining $\overline{\xi}_0$ with $\overline{\xi}_1$, one can show that, if $t\overline{\xi}_0$, $c\overline{\mu}$, and $t'\overline{\xi}_1$ are in a straight line in $\mathbb{R}^{2,1}$ for some $t,t'>0$ such that $t\overline{\xi}_0\ll t'\overline{\xi}_1$ (see \autoref{fig:comparison situation-cat(-1)}), then
\[
\frac{\tau_{\mathbb{R}^{2,1}}(t\overline{\xi}_0,c\overline{\mu})}{\tau_{\mathbb{R}^{2,1}}(c\overline{\mu},t'\overline{\xi}_1)} = \frac{t}{t'}.
\]

\begin{figure}
\def \globalscale {0.800000}
\begin{tikzpicture}[y=1cm, x=1cm, yscale=1.3*\globalscale,xscale=1.3*\globalscale, every node/.append style={scale=\globalscale}, inner sep=0pt, outer sep=0pt]
  \path[draw=black,line width=0.02cm,fill=blue!20,fill opacity=0.5] (7.8, 23.7) -- (6.8, 25.2) -- (10.7, 29.7) -- cycle;
  \path[draw=blue,line width=0.04cm] (7.8, 23.7) -- (8.0, 26.6);
  \path[draw=black,line width=0.02cm] (2.4, 27.3) -- (8.0, 21.7) -- (13.6, 27.3);

  \path[draw=black,line width=0.02cm,fill=red!30,fill opacity=0.5] (7.8, 23.7).. controls (7.4, 23.9) and (7.2, 24.1) .. (7.044, 24.5).. controls (7.7, 24.2) and (8.3, 24.3) .. (8.95, 24.5).. controls (8.65, 24.1) and (8.4, 23.8) .. (7.8, 23.7) -- cycle;
  \path[draw=red,line width=0.04cm] (8.0, 24.3).. controls (7.94, 24.0) and (7.85, 23.8) .. (7.8, 23.7);
  \path[draw=black,line width=0.02cm] (3.0, 27.3).. controls (6.6, 23.7) and (7.5, 23.7) .. (8.0, 23.7).. controls (8.5, 23.7) and (9.4, 23.7) .. (13.0, 27.3);
  
  \path[draw=black,line width=0.02cm,dashed] (8.0, 21.7) -- (7.8, 23.7);
  \path[draw=black,line width=0.02cm,dashed] (8.0, 21.7) -- (8.0, 26.6);
  \path[draw=black,line width=0.02cm,dashed] (10.7, 29.7) -- (8.0, 21.7) -- (6.8, 25.2);
  \path[draw=black,line width=0.02cm,dashed] (8.0, 27.3) ellipse (5.6cm and 0.4cm);
  \path[draw=black,line width=0.02cm,dashed] (8.0, 27.3) ellipse (5.0cm and 0.3cm);

  \path[fill=black,line width=0.02cm] (7.8, 23.7) circle (2pt);
  \path[fill=black,line width=0.02cm] (8.0, 26.6) circle (2pt);
  \path[fill=black,line width=0.02cm] (10.7, 29.7) circle (2pt);
  \path[fill=black,line width=0.02cm] (6.8, 25.2) circle (2pt);
  \path[fill=black,line width=0.02cm] (7.044, 24.5) circle (2pt);
  \path[fill=black,line width=0.02cm] (8.95, 24.5) circle (2pt);
  \path[fill=black,line width=0.02cm] (8.0, 24.3) circle (2pt);
  
  \node[text=black,anchor=south west,line width=0.02cm] (text17) at (7.5, 23.3){$\overline{\xi}_2$};
  \node[text=black,anchor=south west,line width=0.02cm] (text18) at (9.1, 24.4){$\overline{\xi}_1$};
  \node[text=black,anchor=south west,line width=0.02cm] (text19) at (6.7, 24.4){$\overline{\xi}_0$};
  \node[text=black,anchor=south west,line width=0.02cm] (text24) at (8.2, 24.){$\overline{\mu}$};
  \node[text=black,anchor=south west,line width=0.02cm] (text25) at (6.3, 25.2){$t\overline{\xi}_0$};
  \node[text=black,anchor=south west,line width=0.02cm] (text26) at (7.5, 26.5){$c\overline{\mu}$};
  \node[text=black,anchor=south west,line width=0.02cm] (text27) at (10.1, 29.7){$t'\overline{\xi}_1$};
  \node[text=black,anchor=south west,line width=0.02cm] (text27) at (11.3, 26.3){$\mathbb{H}^2$};
  \node[text=black,anchor=south west,line width=0.02cm] (text27) at (13.2, 26.5){$\mathbb{R}^{2,1}$};
\end{tikzpicture}
\caption{}
\label{fig:comparison situation-cat(-1)}
\end{figure}

If we further assume that $t,t'>0$ are sufficiently large such that $\overline{\xi}_2\ll t\overline{\xi}_0\ll t'\overline{\xi}_1$, then a straightforward computation, coupled with \autoref{pop:asymptotic time separation} and the global $\CBA(0)$ condition in $Y$ yield
\begin{align*}
\tau_{\mathbb{R}^{2,1}}(\overline{\xi}_2,c\overline{\mu})^2 &= \frac{t}{t+t'}\tau_{\mathbb{R}^{2,1}}(\overline{\xi}_2,t'\overline{\xi}_1)^2 + \frac{t'}{t+t'}\tau_{\mathbb{R}^{2,1}}(\overline{\xi}_2,t\overline{\xi}_0)^2 - \frac{tt'}{(t+t')^2}\tau_{\mathbb{R}^{2,1}}(t\overline{\xi}_0,t'\overline{\xi}_1)^2\\
&=\lim_{s\to\infty}\frac{1}{s^2}\left( 
\frac{t}{t+t'}\tau(\xi_2(s),\xi_1(t's))^2 + \frac{t'}{t+t'}\tau(\xi_2(s),\xi_0(ts))^2 - \frac{tt'}{(t+t')^2}\tau(\xi(ts),\xi_1(t's))^2
\right)\\
&\leq \lim_{s\to \infty} \frac{1}{s^2} \tau\left(\xi_2(s),m_{t,t'}(s)\right)^2
\end{align*}
where 
\[
m_{t,t'}(s) = [\xi_0(ts),\xi_1(t's)]\left(\frac{t}{t+t'}\right).
\]
Now observe that 
\begin{equation}\label{eq:cat(1) almost intermediate point 1}
\lim_{s\to \infty} \frac{1}{s}|\tau(\xi_0(ts),m_{t,t'}(s)) - \tau(\xi_0(ts),\mu(cs))| = 0
\end{equation}
and
\begin{equation}\label{eq:cat(1) almost intermediate point 2}
\lim_{s\to \infty} \frac{1}{s}|\tau(m_{t,t'}(s),\xi_1(t's)) - \tau(\mu(cs),\xi_1(t's))| = 0.
\end{equation}
Since $\lim_{s\to\infty} \frac{1}{s^2}\tau(\xi_0(ts),\xi_1(t's))^2 = t^2+t'^2-2tt'\cosh\ma(\xi_0,\xi_1) \in (0,\infty)$, then equations \eqref{eq:cat(1) almost intermediate point 1} and \eqref{eq:cat(1) almost intermediate point 2} imply that for every $\varepsilon>0$, the following holds for sufficiently large $s$:
\[
|\tau(\xi_0(ts),m_{t,t'}(s)) - \tau(\xi_0(ts),\mu(cs))| < \varepsilon\tau(\xi_0(ts),\xi_1(t's))
\]
\[
|\tau(m_{t,t'}(s),\xi_1(t's)) - \tau(\mu(cs),\xi_1(t's))| < \varepsilon\tau(\xi_0(ts),\xi_1(t's)).
\]
By \autoref{pop:almost intermediate points are close}, there exist $\delta(\varepsilon)>0$ such that $\lim_{\varepsilon\to 0} \delta(\varepsilon) = 0$ and such that 
\[
\mu(cs) \in J(p_\varepsilon(s),q_\varepsilon(s))
\]
for some $p_\varepsilon(s),q_\varepsilon(s)\in [\xi_0(ts),\xi_1(t's)]$ with $p_\varepsilon(s)\ll m_{t,t'}(s)\ll q_\varepsilon(s)$ and 
\[\tau(p_\varepsilon(s),q_\varepsilon(s)) < \delta(\varepsilon)\tau(\xi_0(ts),\xi_1(t's)).\]
By the reverse triangle inequality, 
\[
|\tau(\xi_2(s),m_{t,t'}(s))-\tau(\xi_2(s),\mu(cs))|\leq |\tau(\xi_2(s),p_\varepsilon(s))-\tau(\xi_2(s),q_\varepsilon(s))|
\]
which implies
\begin{equation}\label{eq:cat(1) final step}
\lim_{s\to \infty} \left|\frac{1}{s} \tau(\xi_2(s),m_{t,t'}(s)) - \frac{1}{s} \tau(\xi_2(s),\mu(cs))\right|
\leq \lim_{s\to \infty}\frac{1}{s}|\tau(\xi_2(s),p_\varepsilon(s))-\tau(\xi_2(s),q_\varepsilon(s))|.
\end{equation}
By the proof of \autoref{pop:almost intermediate points are close}, we can choose $p_\varepsilon(s)$ and $q_\varepsilon(s)$ such that they divide the segment $[\xi_0(ts),\xi_1(t's)]$ into fixed ratios. Therefore, by \autoref{lem:angle bisectors converge to midpoint ray} and \autoref{pop:midpoints}, $p_\varepsilon(s)\to p_\varepsilon$ and $q_\varepsilon(s)\to q_\varepsilon$ in the cone topology and
\[
\widetilde{\ma}_{p}(p_\varepsilon(s),q_\varepsilon(s) )\to \ma(p_\varepsilon,q_\varepsilon) 
\]
and by the Lorentzian law of cosines, $\ma(p_\varepsilon,q_\varepsilon) \to 0$ as $\varepsilon \to 0$. This, combined with equation \eqref{eq:cat(1) final step}, yields 
\[
\lim_{s\to \infty}\frac{1}{s} \tau(\xi_2(s),m_{t,t'}(s)) = \lim_{s\to \infty}\frac{1}{s} \tau(\xi_2(s),\mu(cs))
\]
and  the result follows.
\end{proof}

\section{Timelike ideal boundaries of product spaces}\label{sec:cones}

In this section we consider an example of space where the timelike ideal boundary can be described completely. Namely, product spaces of the form $\smash{\prescript{-}{}{\R}\times_f X}$, where $X$ is a proper $\CAT(0)$ space.

\begin{thm}\label{thm:minkowski product}
Let $Y=\prescript{-}{}{\mathbb{R}}\times_1 X$ be the Lorentzian product of $\mathbb{R}$ with a proper $\mathrm{CAT}(0)$ space $X$. Then $Y$ is a proper, globally hyperbolic and regular \LLS satisfying $\CBA(0)$ globally and $\bd^+Y$ is isometric to $[0,\infty)\times_{\sinh}\bd X$ where $\bd X$ is the ideal boundary of $X$ as a $\mathrm{CAT}(0)$ space. Analogously with $\bd^-Y$. 
\end{thm}
\begin{proof}
The first assertion, namely that $Y$ is a proper, globally hyperbolic and regular \LLS satisfying $\CBA(0)$ globally is an immediate consequence of Theorems~\ref{thm:structure-generalizedcones} and \ref{thm:curvature_cones}.

Let us therefore address the remaining assertions. In this case, the time separation is given for $y_i=(t_i,x_i)$, by $\tau(y_1,y_2) = \sqrt{(t_2-t_1)^2-d_X(x_1,x_2)^2}$ whenever the radicand is non-negative and $t_2\geq t_1$, and by $\tau(y_1,y_2)=0$ otherwise.

For each $x\in X$ and $a\in \R$, denote by $\gamma_{x,a}\colon [0,\infty)\to Y$ the \textit{vertical geodesic} ray given by \[\gamma_{x,a}(t)=(t+a,x)\] for all $t$. It is immediate to check that $\gamma_{x_1,a}$ and $\gamma_{x_2,b}$ are asymptotic for every $x_1,x_2\in X$ and $a,b\in \R$. Indeed, by the constancy of the warping function, it suffices to consider the case $a=b=0$. In such a case, it is enough to take a shift in the parameter by any $c>d_X(x_1,x_2)$. Call $\xi_0$ the equivalence class containing all vertical rays.

Consider now $\xi\in \bd^+Y$ and $\gamma=(\alpha,\beta)\in \xi$. Then by \autoref{thm:properties-geodesics-generalizedcones} we have that $\beta=\mathrm{pr}_X\gamma$ is minimizing and $v_\beta$ is constant. Therefore, either $\beta$ is (upon parametrization by its arclength) a geodesic ray in $X$ or it is a constant path. The latter case of course occurs when and only when $\gamma$ is vertical.

Consider $\xi\in\bd^+Y$ and $\gamma_1,\gamma_2\in \xi$. There exists then some parameter shift $c\in\R^+$ such that $\gamma_i(t)\ll \gamma_j(t+c)$ for all $t\in[0,\infty)$. Or, what is the same, with the notation $\gamma_i=(\alpha_i,\beta_i)$,
\begin{equation}\label{eq:ineq0-proof-straightcone}
\alpha_j(t+c)-\alpha_i(t)>d_X\bigl( \beta_i(t),\beta_j(t+c)\bigr),\quad \forall t.
\end{equation}

As $\gamma_i$'s are always parametrized by $\tau$-length, we have that $\dot{\alpha}_i^2=1+v_{\beta_i}^2$, which is constant, so $\smash{\alpha_i(t)=t\dot{\alpha}_i+k_i}$ for some constants $k_i$, and the previous conditions become 
\begin{equation}\label{eq:ineq1-proof-straightcone}
    t\left(\dot{\alpha}_2 -\dot{\alpha}_1\right)+c\dot{\alpha}_2+k_2-k_1 >
    d_X\bigl(\beta_1(t),\beta_2(t+c)\bigr) \geq d_X\bigl(\beta_1(t),\beta_2(t)\bigr) - d_X\bigl(\beta_2(t),\beta_2(t+c)\bigr)
\end{equation}
for all $t\in [0,\infty)$, and similarly (taking $\tilde{t}=t+c$),
\begin{equation}\label{eq:ineq2-proof-straightcone}
    \tilde{t} \left(\dot{\alpha}_1 -\dot{\alpha}_2\right)+c\dot{\alpha}_2+k_1-k_2 >
    d_X\bigl(\beta_1(\tilde{t}),\beta_2(\tilde{t}-c)\bigr) \geq d_X\bigl(\beta_1(\tilde{t}),\beta_2(\tilde{t})\bigr) - d_X\bigl(\beta_2(\tilde{t}),\beta_2(\tilde{t}-c)\bigr).
\end{equation}

The last terms in the right-hand sides of equations \eqref{eq:ineq1-proof-straightcone} and \eqref{eq:ineq2-proof-straightcone} are constantly $cv_{\beta_2}$, so calling $\delta=\dot{\alpha}_2-\dot{\alpha}_1$ we have that for all $t\geq c$, the term $d_X\bigl(\beta_1(t),\beta_2(t)\bigr)$ is bounded above by some constant plus $t\delta$ and by some other constant plus $-t\delta$. As a consequence, we must have $\delta=0$, and thus $v_{\beta_1}=v_{\beta_2}$. Now we have two cases: either $v_{\beta_1}=v_{\beta_2}=0$ or $v_{\beta_1}=v_{\beta_2}\neq0$. In the first case both rays $\gamma_i$ are vertical. In particular, we have proved that no non-vertical future directed timelike geodesic ray can be asymptotic to a vertical one. In other words, $\xi_0$ contains exclusively vertical rays. In the second case, we still have from any of the equations \eqref{eq:ineq1-proof-straightcone} and \eqref{eq:ineq2-proof-straightcone}, that $d_X\bigl(\beta_1(t),\beta_2(t)\bigr)$ is bounded above by a constant. Therefore, $\beta_1$ and $\beta_2$ (or, more precisely, their parametrizations by arclength) are asymptotic in the metric sense (see \autoref{def:metric-geodesic-rays-asymptoticity}).

With a similar argument, one can deduce that the past of every future directed timelike geodesic ray is the whole space $Y$, something which will be useful later on. Indeed (see also \autoref{lem:cones-past-whole-space} for the general argument) consider a point $p=(t_0,x)\in Y$ and a geodesic ray $\gamma=(\alpha,\beta)$. Then, $p$ is in the timelike past of $\gamma(s)$ for some $s\in[0,\infty)$ if and only if $d_X(x,\beta(s))<\alpha(s)-t_0$. We know that the left hand side is bounded above by $d_X(x,\beta(0))+d_X(\beta(0),\beta(s)):=d_0+v_\beta s$, whereas from $\dot{\alpha}^2=1+v_{\beta}^2$ one deduces there exists some $\tilde{s}\geq 0$ such that for every $s>\tilde{s}$,
\[
d_X(x,\beta(s)) \leq d_0+ s\, v_\beta  < \alpha(0) + s \sqrt{1+v_\beta^2}-t_0 = \alpha(s)-t_0.
\]

For each non-vertical $\xi\in\bd^+Y$, define the equivalence class $\xi_X=\mathrm{pr}_X\xi:=[\mathrm{pr}_X \gamma] \in \bd X$, where $\gamma$ is \textit{any} future directed timelike geodesic ray in $\xi$. Here, we are implicitly considering the arclength parametrization of $\mathrm{pr}_X \gamma$, so that it is a geodesic ray and we can consider the class of geodesic rays asymptotic to it. The map $\xi\mapsto \xi_X$ for non-vertical $\xi$ is well-defined by the previous considerations, i.e., asymptotic future-directed timelike rays have asymptotic (reparametrized) projections. For the vertical ideal point $\xi_0$ we can proceed similarly: the constant paths $\smash{\mathrm{pr}_X (\gamma_{x,a}) \colon [0,\infty]\to X}$ remain all at finite constant distance from one another, so we can call $\smash{\xi_X^0}$ the set of all constant paths, for future convenience.

Therefore we can define the map
 \[
 \begin{aligned}
 \Phi\colon &Y\to [0,\infty)\times_{\sinh}\bd X\\
 &\xi \longmapsto \bigl(\theta(\xi),\xi_X \bigr),
 \end{aligned}
 \]
 where $\theta(\xi):=\ma(\xi_0,\xi)$ and, as before, $\xi_0$ is the equivalence class of vertical rays. Notice that, as $\sinh(0)=0$, the warped product $[0,\infty)\times_{\sinh}\bd X$ has an \textit{apex} $P$, where $\bd X$ is collapsed into a point. This point is attained by $\Phi$ precisely at the vertical ray, i.e., $\smash{P=\Phi(\xi_0)=(0,\xi_X^0)}$. In a couple of lines we will show that the angle $\theta(\xi)$ is always finite. With these considerations, the map is readily seen to be well-defined. The rest of the proof will be to show that this map is an isometry.

The angle $\theta(\xi)=\ma(\xi,\xi_0)$, being an angle between classes of asymptotic geodesic rays, is defined as in \autoref{def:angular-distance}, that is,
\[
\theta (\xi) = \sup_{y \in Y} \bigl\{\ma_y (\xi_y, (\xi_0)_y)  \bigr\}.
\]

Recall here that the past of every future directed timelike geodesic ray is the whole space $Y$. We will now show that these angles are independent of choice of the point $y$.

Indeed, take some point $y\in Y$ and consider representatives $\xi_y$ and $(\xi_0)_y$ of $\xi$ and $\xi_0$, respectively. Then, as in \autoref{def:upper-angle},
\[
\ma_y\bigl(\xi_y, (\xi_0)_y\bigr)=\limsup_{\substack{(s,t) \in A \\ s,t \searrow 0}} 
       \widetilde{\ma}_{y}\bigl(\xi_y(s),(\xi_0)_y(t)\bigr).
\]
In turn, calling $\xi_y=(\alpha,\beta)$ for simplicity, one has for every $(s,t)\in A$ (see \autoref{def:upper-angle}),
\[
\begin{aligned}
    \cosh\Bigl(\widetilde{\ma}_{y}\bigl(\xi_y(s),(\xi_0)_y(t)\bigr)\Bigr) &=
        \frac{\tau\bigl(y,\xi_y(s)\bigr)^2+ \tau\bigl(y, (\xi_0)_y(t)\bigr)^2 -\tau\bigl(\xi_y(s),(\xi_0)_y(t)\bigr)^2}{2\tau\bigl(y,\xi_y(s)\bigr)\tau\bigl(y, (\xi_0)_y(t)\bigr)}
    =\\
    & = 
    \frac{s^2+t^2-\left(t^2+s^2-2ts\sqrt{1+v_\beta^2}\,\right)}{2ts}
    = \sqrt{1+v_\beta^2}
\end{aligned}
\]
and, therefore, by a previous discussion in this proof, it is independent of $y$. Moreover, this connects directly $\theta$ with $v_\beta$, making it possible to compute one from the other one \textit{via} 
\begin{equation}\label{eq:rel-theta-v}
\cosh^2\theta=1+v_\beta^2, \qquad \text{and thus also} \qquad
v_\beta=\sinh\theta
\end{equation}
and, in particular, we deduce that $\theta=\theta(\xi)$ is finite.

This already ensures that $\Phi$ is injective. Indeed, consider $\xi_1,\xi_2\in\bd^+Y$ such that $\Phi(\xi_1)=\Phi(\xi_2)$. If the common image is the apex $P$, then $\theta_1=\theta_2=0$, and thus both rays are the vertical ray. Otherwise, we still have $\theta_1=\theta_2$ and therefore the previous equality ensures that for every choice of representatives $\gamma_i=(\alpha_i,\beta_i)$ of $\xi_i$, one has $v_{\beta_1}=v_{\beta_2}$, which we will denote $v$ for easier reading. Moreover, $\beta_1$ and $\beta_2$, after reparametrization by arclength, will be asymptotic in the metric sense, i.e., $d_X\bigl(\beta_1(t),\beta_2(t)\bigr)\leq k$ for some constant $k$. So consider some shift $c\in\R^+$ in the parameter. We have 
\[
\alpha_j(t+c)-\alpha_i(t)=c\,\sqrt{1+v^2}\quad \text{and} \quad
d_X\bigl(\beta_i(t),\beta_j(t+c)\bigr) \geq d_X\bigl(\beta_j(t),\beta_j(t+c)\bigr)-d_X\bigl(\beta_i(t),\beta_j(t)\bigr)\geq cv-k.
\]
So, we deduce that for $c$ large enough we will have the condition for the rays being asymptotic, i.e., the left hand side before being strictly greater than the right hand side for all $t$. To sum up, $\xi_1=\xi_2$, as we wanted to conclude.

Proving that $\Phi$ is also surjective is almost immediate. Indeed, consider $(\theta,\Xi)\in[0,\infty)\times_{\sinh} \bd X$. Take a representative $\smash{\hat{\beta}\in \Xi}$, which is a geodesic ray $\smash{\hat{\beta}\colon [0,\infty)\to X}$ parametrized by arclength. Reparametrize $\smash{\hat{\beta}}$ in such a way that its (constant) metric speed, which we denote for simplicity $v_\beta$, satisfies equation~\eqref{eq:rel-theta-v}. Denote by $\beta$ such reparametrization and define $\alpha\colon [0,\infty)\to \R$ by $\smash{\alpha(t)=t(1+v_\beta^2)^{1/2}}$. The curve $\gamma=(\alpha,\beta)\colon [0,\infty)\to Y$ is a future directed timelike geodesic ray whose equivalence class is sent by $\Phi$ to $(\theta,\Xi)$.

Let us finally prove that $\Phi$ preserves distances. The distance in the warped product $[0,\infty)\times_{\sinh} \bd X$ is given by \cite[Definition~I.5.6, $k=-1$]{BH}:
\begin{equation}\label{eq:distance-elliptic-cone}
\cosh d\bigl((\theta_1,\Xi_1),(\theta_2,\Xi_2)\bigr) =
\cosh\theta_1\cosh\theta_2 - \sinh\theta_1\sinh\theta_2 \cos \bigl(d_{\bd X}(\Xi_1,\Xi_2)\bigr),
\end{equation}
where in turn, the distance on $\bd X$ is its \textit{angular distance}, as introduced in \autoref{subs:metric-ideal-boundary}.

So consider $\xi_1,\xi_2\in\bd^+Y$ and a point $y\in Y$, and take the corresponding representatives $\xi_y^i=(\alpha_{i,y},\beta_{i,y})$. Call $v_{i,y}=\sinh\theta_i$ the metric speed of the $\beta$'s and $\smash{\hat{\beta}_{i,y}}$ the corresponding reparametrizations by arclength, i.e., $\smash{\beta_{i,y}(t)=\hat{\beta}_{i,y}(tv_{i,y})}$. We have, for $\xi_y^1(s)\ll\xi_y^2(t)$ or vice-versa,
\[
\begin{aligned}
\cosh\widetilde{\ma}_y\bigl(\xi_y^1(s),\xi_y^2(t)\bigr) &=
\frac{\tau\bigl(y,\xi_y^1(s)\bigr)^2+ \tau\bigl(y,\xi_y^2(t)\bigr)^2-\tau\bigl(\xi_y^1(s),\xi_y^2(t)\bigr)^2}{2 \tau\bigl(y,\xi_y^1(s)\bigr) \tau\bigl(y,\xi_y^2(t)\bigr)}=\\
&=\frac{s^2 + t^2- 
\Bigl[\Bigl(s\sqrt{1+v_{1,y}^2}- t\sqrt{1+v_{2,y}^2}\, \Bigr)^2-
d_X\bigl(\beta_{1,y}(s),\beta_{2,y}(t)\bigr)^2\Bigr]}{2ts}\\
&=\frac{2ts\sqrt{1+v_{1,y}^2}\sqrt{1+v_{2,y}^2} + d_X\bigl(\beta_{1,y}(s),\beta_{2,y}(t)\bigr)^2-s^2v_{1,y}^2 -t^2v_{2,y}^2}{2ts}.
\end{aligned}
\]

After dividing by $2ts$, the first term in the right hand side is precisely $\cosh\theta_1\cosh\theta_2$. Substituting the $\beta$'s by their arclength reparametrizations and calling $S_y=sv_{1,y}$ and $T_y=tv_{2,y}$ we obtain
\[
\cosh\widetilde{\ma}_y\bigl(\xi_y^1(s),\xi_y^2(t)\bigr)=\cosh\theta_1\cosh\theta_2-\frac{S_y^2+T_y^2-d_X\bigl(\hat{\beta}_{1,y}(S_y),\hat{\beta}_{2,y}(T_y)\bigr)^2}{2T_yS_y}v_{1,y}v_{2,y}.
\]

Notice that the fraction on the right hand side of this equation is precisely the cosine at $x:=\mathrm{pr}_X y$ of the (metric) comparison angle $\smash{\widetilde{\ma}_x^X\bigl(\hat{\beta}_{1,y}(S_y),\hat{\beta}_{2,y}(T_y)\bigr)}$. The angle subtended by the curves $\hat{\beta}_{1,y}$ and $\hat{\beta}_{2,y}$ at $x$ is the $\limsup$ as $S_y,T_y\searrow 0$ of such an expression. Moreover, as $X$ is a $\CAT(0)$ space, the $\limsup$ is a limit. So, taking the $\limsup$ in the previous expression and noticing that $\cosh$ is continuous and strictly increasing on $[0,\infty)$, we deduce
\[
\cosh\ma_y(\xi_y^1,\xi_y^2)=
\limsup_{\substack{(s,t) \in A \\ s,t \searrow 0}}  \cosh\widetilde{\ma}_y \bigl(\xi_y^1(s),\xi_y^2(t)\bigr)= 
\cosh\theta_1\cosh\theta_2 - 
\sinh\theta_1\sinh\theta_2 \cos\ma_x^X(\beta_{1,y},\beta_{2,y}).
\]

Finally, again using that $\cosh$ is increasing and continuous on $[0,\infty)$ and that $\cos$ is decreasing and continuous on $[0,\pi]$, we deduce
\begin{equation}\label{eq:angular-distance-product}
\begin{aligned}
\cosh\ma(\xi_1,\xi_2)&=\sup_y \cosh \ma_y(\xi_y^1,\xi_y^2)=\cosh\theta_1 \cosh\theta_2 - 
\sinh\theta_1\sinh\theta_2 \inf_y \Bigl( \cos\ma_x^X(\beta_{1,y},\beta_{2,y})\Bigr)=\\
&=\cosh\theta_1 \cosh\theta_2 - 
\sinh\theta_1\sinh\theta_2 \cos\Bigl(\sup_y\ma_x^X(\beta_{1,y},\beta_{2,y})\Bigr).
\end{aligned}
\end{equation}

To conclude, notice that the supremum being over all $y\in Y$ is the same as being over all $x\in X$, as the projections $\beta_{i,y}$ depend not on $y$ but on $x=\mathrm{pr}_X y$. Indeed, if $y,y'\in Y$ have $\mathrm{pr}_X y=\mathrm{pr}_X y'$ then $\beta_{i,y}(\infty)=\beta_{i,y'}(\infty)=\xi_X^i$, and $\beta_{i,y}(0)=\beta_{i,y'}(0)=x$ by construction, so the uniqueness of representatives of a boundary point in the metric setting (see \autoref{subs:metric-ideal-boundary}) gives us $\beta_{i,y}=\beta_{i,y'}$. In other words, the $\sup$ on the right hand side is precisely the (metric) angular distance between $\xi_X^1$ and $\xi_X^2$, and therefore $\Phi$ is an isometry.
\end{proof}

\begin{rem}\label{rem:non-unit-cones}    
Notice that the fact that the (constant) warping function takes the value $1$ is unimportant. Indeed, the generalized cone $\smash{\prescript{-}{}{\mathbb{R}}\times_L X}$, for $L>0$, is equivalent to the generalized cone $\smash{\prescript{-}{}{\mathbb{R}}\times_1 \widetilde{X}}$, where $\smash{\widetilde{X}}$ is the metric space with rescaled distance $\smash{\tilde{d}:=Ld_X}$. This equivalence is in the sense that both have the same Lorentzian structures (see \autoref{def:generalized-cones}). However, Equation~\eqref{eq:rel-theta-v} relating the angle with the vertical ray and the horizontal speed will be changed in the former case to $\cosh^2\theta=1+L^2v_\beta^2$ and $\sinh\theta=Lv_\beta$.
\end{rem}

\begin{rem}
    The angular distance in the metric ideal completion $\bd X$ of a proper $\CAT(0)$ space $X$ takes values in $[0,\pi]$. Consider a Lorentzian product $Y=\prescript{-}{}{\mathbb{R}}\times_1 X$ and two future-directed timelike geodesic rays $\xi_i$ whose projections $\Xi_i$ to $X$ are at angular distance $0$ or $\pi$. Then, the (Lorentzian) angular distance between the timelike rays (see Equation~\eqref{eq:angular-distance-product}) becomes, respectively
    \[
     \ma(\xi_1,\xi_2) = \cosh^{-1}\left(
    \cosh\theta_1\cosh\theta_2 \mp \sinh\theta_1\sinh\theta_2\right)= |\theta_1\mp\theta_2|.
    \]

    In particular, when $X$ is a complete $\CAT(-1)$ space the (metric) angular distance between two different geodesic rays is always $\pi$ (see~\cite[Example~II.9.12(1)]{BH}) and, therefore, one has
    \[
    \ma(\xi_1,\xi_2) =\left\{
    \begin{aligned}
    |&\theta_1-\theta_2|, &&\text{if } \Xi_1=\Xi_2, \\
    &\theta_1+\theta_2,     &&\text{otherwise}. 
    \end{aligned}\right.
    \]
    In other words, in that case the future ideal boundary is isometric to the cone over the metric space $(\partial X,d_{\bd X})$ \cite[Proposition~3.6.12]{BBI}, where $d_{\bd X}$ is the discrete distance rescaled by $\pi$. Equivalently, it is isometric to the metric bouquet formed by gluing all segments of the form $[0,\infty)\times\{\xi\}$ where $\xi\in\bd X$, at the points $\{0\}\times \{\xi\}$.
\end{rem}

We also obtain the following partial converse of \autoref{thm:minkowski product} as an immediate consequence of results from \cite{barton-beran-che-gieger-roehrig-rott}. This is can be regarded as a Lorentzian analogue of \cite[Theorem~II.9.24]{BH}.

\begin{thm}\label{thm:rigidity}
Let $Y$ be proper, globally hyperbolic, strongly causal, locally causally closed, regularly localizable Lorentzian pre-length space satisfying $\CBA(0)$ globally and assume that 
$\bd^+ Y \cup \bd^-Y$ is isometric to a pseudo-metric space of the form
\[
\{1,-1\}\times [0,\infty)\times_{\sinh } S
\]
where $\bd^\pm Y$ is mapped to $\{\pm 1\}\times [0,\infty)\times_{\sinh} S$ under this isometry, and the distance between $(-1,0,a)$ and $(1,0,a)$ is 0. Then $Y = \prescript{-}{}{\mathbb{R}}\times X$ for some $\CAT(0)$ space $X$ such that $\bd X$ is isometric to $S$.
\end{thm}

\begin{proof}
By \autoref{pop:different-pasts-infinite-angle}, $Y = I^-(\xi) = I^+(\xi')$ for any $\xi \in \bd^+Y$ and $\xi'\in \bd^-Y$. Moreover, if $\xi\in \bd^+ Y$ and $\xi'\in \bd^-Y$ are the ideal points corresponding to $(1,0,a)$ and $(-1,0,a)$, respectively, under the given identification of $\bd^+Y\cup \bd^-Y$ with $\{1,-1\}\times [0,\infty)\times_{\sinh } S$, then $\ma(\xi,\xi')=0$. \autoref{pop:main1} implies that for any $x\in Y$ there is a timelike geodesic line $c_x\colon \mathbb{R} \to Y$ such that $c_x(\infty) = \xi$, $\bar{c}_x(\infty)=\xi'$ and $c_x(0)=x$, where $\bar{c}_x(t)=c_x(-t)$. By \cite[Corollary~5.11]{barton-beran-che-gieger-roehrig-rott}, it follows that the lines $c_x$ are pairwise weakly parallel, and by \cite[Theorem~1.1]{barton-beran-che-gieger-roehrig-rott}, it follows that $Y$ is isometric to a product $\prescript{-}{}{\mathbb{R}}\times_1 X$. \autoref{thm:minkowski product} implies that $\bd X$ is isometric to $S$.
\end{proof}

\section{Timelike ideal boundaries of generalized cones}

In this final section we consider generalized cones of the form $\smash{\prescript{-}{}{\R}\times_f X}$, as in \autoref{def:generalized-cones}. Most of the results presented here can immediately be transferred to the case  $\smash{\prescript{-}{}{(a,b)}\times_f X}$, where $b=\infty$ when we deal with the future ideal boundary, and $a=-\infty$ for the case of the past ideal boundary. Moreover, we will mainly focus on the future ideal boundary, since the past case can be treated similarly.

Though we will care mostly about generalized cones of the form $\smash{\prescript{-}{}{\R}\times_f X}$, the following result ensures that the treatment of the case $\smash{\prescript{-}{}{(a,\infty)}\times_f X}$ is completely analogous, under the assumption of the global $\CBA(0)$ condition.

\begin{lem}\label{lem:restricted-cone}
    Let $(X,d)$ be a complete locally compact length space and $f\colon \R\to (0,\infty)$ continuous. Denote $Y:=\smash{\prescript{-}{}{\R}\times_f X}$ and $Y_a:=\smash{\prescript{-}{}{(a,\infty)}\times_f X}$, where in this last expression we implicitly consider the restriction of $f$ to $(a,\infty)$, and assume that $Y$ (and therefore also $Y_a$) satisfies $\CBA(0)$ globally. Then the future timelike ideal boundaries of $Y$ and $Y_a$ with their respective angular distances are isometric. The past case is analogous.
\end{lem}
\begin{proof}
    The map given by
    \[
    \begin{aligned}
    \Psi_a\colon \bd^+&Y\to\bd^+Y_a\\
    &\xi\longmapsto \xi_a
    \end{aligned}
    \]
    where $\xi_a=\{\gamma\in\xi\mid \gamma\subset Y_a\}$, is well-defined, i.e., $\xi_a\in\bd^+Y_a$. Indeed, on the one hand, $\gamma_1\sim\gamma_2$ holds trivially for all $\gamma_1,\gamma_2\in\xi_a$. On the other hand, if a future-directed timelike geodesic ray $\gamma\subset Y_a$ is asymptotic to a ray $\gamma'\in\xi_a$, then $\gamma\sim\gamma'$ in $Y$, therefore, $\gamma\in\xi_a$. 
    
    Furthermore, $\Psi_a$ is a bijection: surjectivity is immediate, whereas injectivity follows from the fact that $\xi$ is an equivalence class and therefore different classes are disjoint.

    Let us prove that $\Psi_a$ preserves the angular metric. One has
    \[
    \ma^Y(\xi_1,\xi_2)=\sup_{y\in Y} \ma^Y_y(\xi_{1},\xi_{2})\geq \sup_{y\in Y_a} \ma^Y_y(\xi_{1},\xi_{2})=
    \sup_{y\in Y_a} \ma^{Y_a}_y(\xi_{1},\xi_{2})= \ma^{Y_a}\bigl(\Psi_a(\xi_1),\Psi_a(\xi_2)\bigr),
    \]
    where in the third step one just needs to notice that the expressions for the angles between geodesics are the same if they are ``seen'' in $Y$ or in $Y_a$.
    
    For the converse inequality, consider first that $\xi_1,\xi_2\in\bd^+Y$ generate different timelike pasts. Then, also the pasts of the rays in $\Psi_a(\xi_1)$ and $\Psi_a(\xi_2)$ have to be different, and \autoref{pop:different-pasts-infinite-angle} gives that \[\ma^{Y_a}\bigl(\Psi_a(\xi_1), \Psi_a(\xi_2)\bigr)=\ma^{Y}\bigl(\xi_1,\xi_2\bigr)=\infty.\]

    Otherwise, consider that the rays $\xi_1,\xi_2\in\bd^+Y$ generate the same past and take a point $y$ in such a common past. We will show that there exists $y'\in Y_a$ such that 
    \[
    \ma^Y_y(\xi_{1},\xi_{2})\leq
    \ma^{Y_a}_{y'}(\xi_{1},\xi_{2}).
    \]

    Indeed, denoting $\xi_{1,y}=\gamma=(\alpha,\beta)$, one has that $\dot{\alpha}\geq 1$ and, therefore $\alpha(t)\to\infty$ as $t\to\infty$. In particular, there is some $t\in \R$ such that $\alpha(t)>a$ or, in other words, $y':=\xi_{1,y}(t)\in Y_a$. As $\xi_1,\xi_2$ generate the same past, \autoref{pop:existence of pointwise rep} ensures that there exists a representative $\xi_{2,y'}$ asymptotic to $\xi_2$ and starting at $y'$. Notice that being $(X,d)$ complete and locally compact length space, it is proper, and therefore so is $Y$ with the product distance (see \autoref{def:generalized-cones}). Moreover, $(X,d)$ is geodesic, and thus $Y$ is globally hyperbolic and regular. 
    
    Now, by \autoref{pop:properties of angle}(\ref{item:increasing-angle}) we know that the map
    \[
    s\to \ma^Y_{\gamma(s)}(\xi_{1},\xi_{2}),
    \]
    which is well-defined for all $s\in [0,\infty)$ by the previous reasonings, is non-decreasing. As a consequence, we deduce the desired converse inequality.
\end{proof}

Later on, using an explicit characterization of the future ideal boundary in some special cases, we will be able to prove that, under some additional conditions on $f$ and $X$, the same result is true even when only $Y_a$ is assumed to be $\CBA(0)$ globally (see \autoref{pop:restricted-cone-2}). 

We present here two auxiliary results which will be useful further on: the first one to ensure that there exists some non-vertical ray, and the second one to ensure that the past of vertical rays is the whole space.
\begin{lem}\label{lem:auxiliary}
    Let $f\colon \R\to(0,\infty)$ be locally Lipschitz, satisfying either $f(t)\to0$ as $t\to\infty$ and $\int_0^\infty f(t) dt=\infty$, or $f$ is bounded away from $0$, i.e., $\liminf_{t\to\infty}f(t)>0$. Then, for every $A>0$, the initial value problem $\dot{\alpha}^2=1+A^2(f\circ \alpha)^{-2}$ with $\dot\alpha>0$ and $\alpha(0)=t_0$ has a maximal solution whose domain of definition is unbounded above.
\end{lem}
\begin{proof}
    Since $f$ is continuous and locally Lipschitz with values in $(0,\infty)$, the function
    \[
        g(t):=\sqrt{1+\frac{A^2}{f(t)^2}}
    \]
    is locally Lipschitz on $\R$, and hence the initial value problem of the statement admits a unique (forward) maximal solution defined on $(-\varepsilon,s_{\max})$. We will prove that actually $s_{\max}=\infty$.

    We have that $\dot{\alpha}\geq 1$ and therefore $\alpha(s)\to\infty$ as $s\to s_{\max}$. On the other hand,
    \[
    s=\int_{0}^{s}du=\int_{0}^{s}\frac{\dot{\alpha}(u)}{g(\alpha(u))}du=\int_{t_0}^{\alpha(s)}\frac{du}{g(u)}\implies
    s_{\max}=\int_{t_0}^{\infty}\frac{du}{g(u)}.
    \]

    Now assume the first possibility, namely that $f(t)\to0$ as $t\to\infty$ and $\int_0^\infty f(t) dt=\infty$. Let $R>0$ be such that $f(t)\leq A$ whenever $t\geq R$. Then, calling $a=\int_{t_0}^{R}\frac{du}{g(u)}$ one has
    \[
    s_{\max}=a+\int_{R}^{\infty}\frac{f(u)\, du}{\sqrt{f(u)^2+A^2}}\geq 
    a+\int_{R}^{\infty}\frac{f(u)\, du}{A\sqrt{2}}=\infty.
    \]

    Assume otherwise that $f$ is bounded away from $0$. Then $g(u)$ is bounded above by some constant $C>0$ and, as a consequence, $s_{\max}\geq \int_{t_0}^\infty \frac{1}{C}du=\infty$.
\end{proof}
    
\begin{lem}\label{lem:vertical-rays-whole-past}
    Let $(X,d)$ be a geodesic space, let $f\colon \R\to(0,\infty)$ be continuous satisfying $\smash{\int_0^\infty \frac{dt}{f(t)}=\infty}$ and consider the generalized cone $Y=\prescript{-}{}{\mathbb{R}}\times_f X$. Then the timelike past of every future-directed vertical ray in $Y$ is the whole space.
\end{lem}
\begin{proof}
    Consider a vertical ray $\gamma=(\alpha,\beta):[0,\infty)\to Y$. A point $p=(s,x)$ is in the timelike past of $\gamma$ if there is some $t\in[0,\infty)$ such that 
    \begin{equation}\label{eq:lem-vertical-rays}
            \int_{s}^{\alpha(t)}\frac{du}{f(u)}>d_X(x,\beta(t)).
    \end{equation}
    But the left hand side diverges as $t\to \infty$ by hypothesis (notice that $\alpha$ is linear with slope $1$), whereas the right hand side is constant as $\beta$ is.
\end{proof}
\begin{rem}\label{rem:vertical_rays}
    This result is known in the context of \textit{causal boundaries}. Indeed, the causal boundary of a generalized cone $\prescript{-}{}{\R} \times_f X$ in which the warping function satisfies $\smash{\int_0^\infty \frac{dt}{f(t)}=\infty}$ admits a maximal element $\smash{i^+}$, which corresponds to the whole space~$Y$ (see \cite{flores2013memoirs}). In other words, there are future directed timelike curves (for example, a future-directed vertical ray) whose past is the whole space. Notice also that if $\smash{\int_0^\infty \frac{dt}{f(t)}<\infty}$, then there are no future-directed vertical rays whose past is the whole space. Indeed, otherwise the left hand side in \eqref{eq:lem-vertical-rays} can be taken as small as desired by increasing $s$, whereas the right hand side is constant. This result will be strengthened in \autoref{lem:vertical-not-asymptotic}.
\end{rem}

\subsection{Existence or absence of non-vertical rays}
We now proceed to describe the existence or absence of future-directed vertical rays in the generalized product $Y=\prescript{-}{}{\mathbb{R}}\times_f X$ in terms of the behavior of $f$ at infinity.

The following result is a converse of \autoref{lem:auxiliary} when the hypotheses of that result are not met.

\begin{lem}\label{lem:vertical-rays-only-ones}
    Let $(X,d)$ be a geodesic length space, $f\colon \R\to(0,\infty)$ continuous satisfying $\smash{\int_0^\infty f(t)dt<\infty}$, and consider $Y=\prescript{-}{}{\mathbb{R}}\times_f X$. Then the only future directed timelike geodesic rays are the vertical ones. In particular, under the conditions of \autoref{pop:existence of pointwise rep}, $\bd^+Y$ consists of just a point. 
\end{lem}
\begin{proof}
    Consider $\gamma=(\alpha,\beta)\colon [0,a)\to Y$, for $a\in \R^+\cup\{+\infty\}$, a future directed timelike geodesic parametrized by arclength, i.e., in such a way that $\smash{\dot{\alpha}^2-v_\beta^2(f\circ\alpha)^2=1}$ (see \autoref{thm:properties-geodesics-generalizedcones}). As always, there is some constant $A\geq 0$ such that $\smash{\dot{\alpha}^2=1+A^2\left(f\circ\alpha\right)^{-2}}$. This is a condition just on the curve $\alpha\colon[0,a)\to\R$. 

    We will show that, if $A\neq 0$, then necessarily $a<\infty$. Indeed, under such an assumption
    \[
    a=\int_0^a dt= \int_0^a\frac{\dot{\alpha}(t)dt}{\sqrt{1+\frac{A^2}{f(\alpha(t))^2}}}= \int_{\alpha(0)}^{\alpha(t)} \frac{du}{\sqrt{1+\frac{A^2}{f(u)^2}}}\leq \int_{\alpha(0)}^\infty \frac{f(u)du}{\sqrt{f(u)^2+A^2}} < 
    \int_{\alpha(0)}^\infty \frac{1}{A}f(u)du<\infty.
    \]
    As a consequence, the curve $\gamma$ cannot be defined on the interval $[0,\infty)$ unless it is a vertical ray. 

    For the last assertion, we recall that under our hypotheses we have $\smash{\int_b^\infty \frac{1}{f(t)}dt=\infty}$, so every vertical ray has the whole space $Y$ as its timelike past (see \autoref{lem:vertical-rays-whole-past}). Consider a vertical ray $\gamma_p$. \autoref{pop:existence of pointwise rep} ensures then that for every $q=(t,x)\in Y$ there is a unique geodesic ray asymptotic to $\gamma_p$ starting at $q$. Such a ray has to be the vertical ray starting at $q$, as the only timelike rays are the vertical ones. In particular, all future directed timelike geodesic rays (being all vertical) are asymptotic. Alternatively, using that $f(t)\to 0$, one can directly show from the very definition that any two vertical rays are asymptotic (see \autoref{lem:vertical-rays-asymptotic} below), so that it is not necessary that the hypotheses of \autoref{pop:existence of pointwise rep} are satisfied.
\end{proof}

\begin{lem}\label{lem:existence-nonvertical-rays}
    Let $(X,d)$ be a proper geodesic space with infinite diameter, $f\colon \R\to(0,\infty)$ locally Lipschitz such that $\smash{\int_0^\infty f(t)dt=\infty}$. Then there are future directed timelike geodesic rays in the generalized cone $Y=\prescript{-}{}{\mathbb{R}}\times_f X$ which are non vertical.
\end{lem}
\begin{proof}
    Choose a point $y=(t,x)\in Y$ and $A>0$. Define $\alpha\colon [0,a)\to \R$ by the initial value problem $\alpha(0)=t$ and 
    \[
    \dot{\alpha}(s)=\sqrt{1+\frac{A^2}{f(\alpha(s))^{2}}\,},\quad \forall s\in[0,a).
    \]
    
    Notice that, as $f$ is locally Lipschitz and positive, the right hand side of the previous equation is locally Lipschitz and, therefore, the initial value problem admits a locally unique solution. From \autoref{lem:auxiliary}, we know that the maximal domain of definition of $\alpha$ is unbounded above. Hence, we will henceforth assume that $\alpha$ is defined on $[0,\infty)$.

    Now consider in $X$ a geodesic ray $\tilde{\beta}$ starting at $x$, which exists because $X$ is proper, geodesic, and has infinite diameter. This assertion follows from a standard application of Arzelà--Ascoli Theorem and a diagonal argument.

    Let us build a non-vertical timelike geodesic ray $\gamma=(\alpha,\beta)$. To that end, define $\beta\colon[0,\infty)\to X$ by
    \begin{equation}\label{eq:proof-existence-nonvert-rays-reparam}
    \beta(s):=\tilde{\beta}\left(\int_0^s \frac{A\,du}{f(\alpha(u))^2}\right),  \quad \forall s\in[0,\infty).
    \end{equation}

    Let us now prove that $\gamma=(\alpha,\beta)$ is a future directed timelike geodesic. Take two points in the image of $\gamma$, say $\gamma(s)$ and $\gamma(t)$ for $s<t$. As the space is globally hyperbolic (see \autoref{thm:structure-generalizedcones}), there is some future directed timelike curve $\eta=(\rho,\sigma)$ joining them which is maximal, i.e., whose length is the time separation between the extremes. Take such a curve $\eta \colon [a,b]\to Y$ to be parametrized by arclength. Then $\sigma$ is minimizing in $X$ between $\beta(s)$ and $\beta(t)$, and in particular $L(\sigma)=L(\beta|_{[s,t]})$. Moreover, there exists some $B>0$ such that $v_\sigma=B(f\circ\rho)^{-2}$. So we have
    \[
    \int_s^t v_\beta(u)du=\int_a^b v_\sigma(u)du \implies \int_s^t \frac{A\, du}{f(\alpha(u))^2}=\int_a^b \frac{B\, du}{f(\rho(u))^2}. 
    \]

    Therefore, changing variables and using that $\rho$ and $\alpha$ agree on their endpoints, we obtain
    \[
    \int_{\alpha(s)}^{\alpha(t)} \frac{A \, du}{f(u)^2 \, \dot{\alpha}(\alpha^{-1}(u))}=\int_{\alpha(s)}^{\alpha(t)} \frac{B \, du}{f(u)^2 \, \dot{\rho}(\rho^{-1}(u))}.
    \]

    Now substituting the expressions for $\dot{\alpha}$ and $\dot{\rho}$ obtained from the arclength parametrization one deduces
    \[
    \int_{\alpha(s)}^{\alpha(t)} \frac{A\, du}{f(u)\sqrt{f(u)^2+A^2}}= \int_{\alpha(s)}^{\alpha(t)} \frac{B \, du}{f(u)\sqrt{f(u)^2+B^2}}.
    \]

    In these expressions the integrand is strictly increasing, for every fixed $u$, with the parameter $A$ or $B$. As a consequence, the integrals are strictly increasing with $A$ and $B$, respectively and, therefore, $A=B$. But now, $\alpha$ and $\rho$ (after a translation of the parameter of the latter) are given by the same initial value problem, so they coincide by the local Lipschitzness discussed above. Finally, $v_\beta$ and $v_\sigma$ also coincide, so the Lorentzian length of $\gamma|_{[s,t]}$ and $\eta$ agree, and therefore $\gamma|_{[s,t]}$ is a maximal geodesic. As $s<t$ are arbitrary, the result follows.
\end{proof}

\begin{rem}\label{rem:uniquely-geodesic}
    Notice that if $X$ were assumed to be proper and \textit{uniquely} geodesic, the proof of \autoref{lem:existence-nonvertical-rays} would imply that the generalized cone is \textit{uniquely} timelike geodesic. Moreover, to deduce such a statement the assumptions on the diameter of $X$ and on $\smash{\int_0^\infty f(t)dt}$ being infinite are never used.
\end{rem}

\begin{rem}
    In equation~\eqref{eq:proof-existence-nonvert-rays-reparam} of the previous proof we define a curve $\beta$ from $\tilde{\beta}$ in what seems to be a reparametrization. However, notice that the map $\smash{[0,\infty)\ni s\mapsto \int_0^s \frac{A\,du}{f(\alpha(u))^2}\in [0,\infty)}$ is not necessarily surjective. In other words, $\beta$ could be just a geodesic \textit{segment} instead of a geodesic \textit{ray}. This happens if $\smash{\int_a^\infty \frac{du}{f(u)^2}<\infty}$ (see \autoref{lem:horizontal-part-is-ray} below). In such a case, one does not need the existence in $X$ of a metric geodesic ray for the proof to work.
\end{rem}

\subsection{Pasts of timelike rays}
The following results study conditions under which the past of every future-directed timelike ray is the whole space.

\begin{lem}\label{lem:cones-past-whole-space}
    Let $(X,d)$ be a geodesic length space, $f\colon \R\to(0,\infty)$ continuous satisfying $\smash{\int_0^\infty \frac{dt}{f(t)}=\infty}$  and $\liminf_{t\to\infty} f(t)>0$, and consider $Y=\prescript{-}{}{\mathbb{R}}\times_f X$.
    Then every future directed timelike geodesic ray has the whole space $Y$ as its timelike past. 
\end{lem}

\begin{rem}
    The two conditions in \autoref{lem:cones-past-whole-space} play opposite roles. Indeed, the condition about the integral of $1/f$ being infinite ensures that $f$ ``does not grow too fast", as opposed to the $\liminf$ one.
\end{rem}
\begin{proof}
    Let $y=(s,x)\in Y$ be any point, $\xi\in\bd^+Y$, $\gamma=(\alpha,\beta)\in\xi$. We have $\smash{\dot{\alpha}^2-\left(f\circ \alpha\right)^2 v_\beta^2=1}$, so $\dot{\alpha}(t)\geq 1$ for every $t\in[0,\infty)$. Thus $\alpha(t)\to\infty$ as $t\to \infty$. We want to show that there exists some parameter $t\in[0,\infty)$ such that $y\ll \gamma(t)$. This will happen if and only if
    \[
    \int_s^{\alpha(t)}\frac{du}{f(u)}=B+\int_{\alpha(0)}^{\alpha(t)}\frac{du}{f(u)}>d_X\bigl(x,\beta(t)\bigr),
    \]
    where $\smash{B=\int_s^{\alpha(0)}\frac{du}{f(u)}}$ is a constant (i.e., independent of $t$). Note that the integral in the middle term diverges by hypothesis. However, using the triangle inequality on the right hand side we obtain
    \[
    d_X\bigl(x,\beta(t)\bigr)\leq d_X\bigl(x,\beta(0)\bigr)+d_X\bigl(\beta(0),\beta(t)\bigr):=d_0+\int_0^t v_\beta(u) du,
    \]
    where $d_0$ is again independent of $t$ and we used in the last equality that $\beta$ is minimizing in $X$ (see \autoref{thm:properties-geodesics-generalizedcones}). From the same result we know that there is some $A\geq 0$ such that $v_\beta=A \left( f \circ \alpha \right)^{-2}$. Therefore,
    \[
    d_X\bigl(x,\beta(t)\bigr)\leq d_0+\int_0^t \frac{A}{f(\alpha(u))^2}\, du
    \]

    We will show that there is some $t$ such that 
    \[
    B+\int_{\alpha(0)}^{\alpha(t)}\frac{du}{f(u)}=B+\int_{0}^{t}\frac{\dot{\alpha}(u)}{f(\alpha(u))}\, du > d_0+\int_0^t \frac{A}{f(\alpha(u))^2}\, du,
    \]
    and we will be done.

    If $A=0$ there is nothing to prove, as $\alpha(t)\to\infty$ as $t\to \infty$, and $\smash{\int_a^\infty \frac{dt}{f(t)}=\infty}$ for every $a\in\R$. Otherwise assume that $\smash{\liminf_{t\to\infty} f(t) >0}$. Because $f$ is continuous, this means that for every $b\in\R$ there exists some $k_b>0$ such that $f(t)\geq k_b$ for every $t\geq b$. Then
    \[
    \dot{\alpha}(t)^2= 1+\frac{A^2}{f(\alpha(t))^2} \leq 1+\frac{A^2}{k_b^2}, \quad \forall t \text{ such that }\alpha(t)\geq b. 
    \]
    Thus,
    \begin{equation}\label{eq:past-whole-space}
    \begin{aligned}
    \int_0^t&\left[\frac{\dot{\alpha}(u)}{f(\alpha(u))} 
    -\frac{A}{f(\alpha(u))^2}\right] du 
    = \int_0^t\frac{\sqrt{f(\alpha(u))^2+A^2}-A}{f(\alpha(u))^2}\,du = 
    \int_0^t \frac{du}{\sqrt{f(\alpha(u))^2+A^2}+A}
    \\ 
    & \geq 
    \int_0^t \frac{du}{f(\alpha(u))+2A} \geq 
    \frac{1}{1+\frac{2A}{k_{\alpha(0)}}}\int_0^t \frac{du}{f(\alpha(u))}=
    \frac{1}{1+\frac{2A}{k_{\alpha(0)}}} \int_{\alpha(0)}^{\alpha(t)} \frac{du}{f(u)\, \dot{\alpha}\left(\alpha^{-1}(u)\right)}
    \\
    & \geq 
    \frac{1}{\bigl(1+\frac{2A}{k_{\alpha(0)}}\bigr)\sqrt{1+\frac{A^2}{k_{\alpha(0)}^2}}} \int_{\alpha(0)}^{\alpha(t)}\frac{du}{f(u)}\xrightarrow{t\to\infty} \infty,
    \end{aligned}
    \end{equation}
    where the first inequality follows from $\sqrt{x^2+y^2}\leq x+y$, for $x,y\geq0$, and in the second and third inequalities we used the bound on $\dot{\alpha}$. Therefore, the desired inequality holds for large enough~$t$.
\end{proof}

\begin{lem}\label{lem:cones-past-whole-space-ftozero}
    Let $(X,d)$ be a geodesic length space, $f\colon \R\to(0,\infty)$ continuous such that $f(t)\to 0$ as $t\to\infty$, but $\smash{\int_0^\infty f(t)dt=\infty}$, and consider $Y=\prescript{-}{}{\mathbb{R}}\times_f X$. Then every future directed timelike geodesic ray has the whole space $Y$ as its timelike past.
\end{lem}
\begin{proof}
    As in the proof of \autoref{lem:cones-past-whole-space}, consider the integral
    \[
    g(t)=\int_0^t\left[\frac{\dot{\alpha}(u)}{f(\alpha(u))} 
    -\frac{A}{f(\alpha(u))^2}\right] du.
    \]
    
    Again, we want to show that it diverges, but now the proof is even simpler as the one of \autoref{lem:cones-past-whole-space}. Indeed, consider $g'(t)$, which is just the term in between square brackets in the previous equation. For large enough $t\in[0,\infty)$ one has $f(t)\leq A$. Therefore, following similar reasonings as in Equation~\eqref{eq:past-whole-space}, one has
    \[
    g'(t)=\frac{\sqrt{f(\alpha(u))^2+A^2}-A}{f(\alpha(u))^2}=
    \frac{1}{\sqrt{f(\alpha(u))^2+A^2}+A}\geq \frac{1}{2A+f(\alpha(u))}\geq \frac{1}{3A}>0,
    \]
    and as a consequence $g$ diverges, as we wanted to prove.
\end{proof}

The last part of \autoref{rem:vertical_rays} asserts that if the first hypothesis of \autoref{lem:cones-past-whole-space} is not satisfied, i.e., if $\int_a^\infty 1/f(s) ds<\infty$, then there are no vertical rays whose past is the whole space. However, it does not assert explicitly whether in such a case different vertical rays can have the same past. This is solved in the following lemma, in which the assumption of infinite diameter becomes unnecessary.
\begin{lem}\label{lem:vertical-not-asymptotic}
     Let $(X,d)$ be a geodesic length space and let $f\colon \R\to(0,\infty)$ be continuous satisfying $\smash{\int_0^\infty \frac{dt}{f(t)}<\infty}$. Consider $Y=\prescript{-}{}{\mathbb{R}}\times_f X$. Then different vertical rays (or more rigorously, vertical rays with different horizontal components) have different pasts and, in particular, they cannot be asymptotic.
\end{lem}
\begin{proof}
    Let $x_1,x_2\in X$ be distinct points and consider the vertical rays $\gamma_1,\gamma_2$ above them which, without loss of generality, can be taken such that $\gamma_i(0)=(0,x_i)$ and, therefore, $\alpha_i(t)=t$. We will show that there is some $\tilde{s}\geq 0$ such that $\gamma_1(s)\not\ll\gamma_2(t)$ for any $t\in[0,\infty)$ and $s>\tilde{s}$.

    Indeed, we know that
    \begin{equation}\label{eq:vertical-not-asymptotic1}
        \gamma_1(s)\ll\gamma_2(t) \iff d_X(x_1,x_2)<\int_s^t \frac{du}{f(u)}.
    \end{equation}
    However,
    \begin{equation*}
    \int_s^t \frac{du}{f(u)}<\int_s^\infty \frac{du}{f(u)} \xrightarrow{s\to\infty} 0.
    \end{equation*}
    As a consequence, for large enough $s$, the right-hand side of \eqref{eq:vertical-not-asymptotic1} is false for any $t$, since $d(x_1,x_2)>0$ is independent of $s$. In other words, for sufficiently large $s$ the timelike relation on the left-hand side of \eqref{eq:vertical-not-asymptotic1} cannot happen  for any $t$. Therefore, $\gamma_1(s)$ is not in the past of $\gamma_2$, whereas it is obviously in the past of $\gamma_1$.
\end{proof}

\begin{rem}\label{rem:causal-boundary-finite-integral}
    The previous result has also been obtained in the context of the causal boundary. Indeed, the future causal boundary $\hat{\bd}Y$ of a generalized cone $Y$ satisfying the hypotheses of \autoref{lem:vertical-not-asymptotic} consists of a copy of the metric completion of the fiber $X$ \cite{flores2007product}, and each point of $\hat{\bd}Y$ which corresponds to an element $x\in X$ (i.e., which is not in the metric boundary of $X$) is generated by a vertical ray over $x$.
\end{rem}

\subsection{Horizontal component of a timelike ray}

The horizontal part of a timelike ray, that is, its projection to the fiber $X$, is a minimizing curve by \autoref{thm:properties-geodesics-generalizedcones}. The following result addresses under which conditions on the warping function these projections are metric geodesic rays.

\begin{lem}\label{lem:horizontal-part-is-ray}
    Let $(X,d)$ be a geodesic length space, $f\colon \R\to(0,\infty)$ continuous satisfying either $\smash{\int_0^\infty \frac{dt}{f(t)^2}=\infty}$ and $\liminf_{f\to\infty}f(t)>0$, or $f(t)\to 0$ as $t\to\infty$, and consider $Y=\prescript{-}{}{\mathbb{R}}\times_f X$. Then, every non-vertical future directed timelike geodesic ray has a horizontal component with infinite length, i.e., it can be parametrized as a (metric) geodesic ray. On the contrary, if $\smash{\int_0^\infty \frac{dt}{f(t)^2}<\infty}$, then the horizontal part of every ray has finite length. 
\end{lem}
\begin{proof}
    Consider a future directed timelike ray $\gamma=(\alpha,\beta)$. The length of its spatial projection is 
    \begin{equation}\label{eq:length-horizontal-part}
    d_X(\beta(0),\beta(t))=\int_0^t v_\beta(u)du=\int_0^t \frac{A\, du}{f(\alpha(u))^2}
    =
    \int_{\alpha(0)}^{\alpha(t)}\frac{A \, du}{f(u)^2\sqrt{1+\frac{A^2}{f(u)^2}}}.
    \end{equation}

    Using the same bound as in the proof of \autoref{lem:cones-past-whole-space} we have  
    \[
    d_X\bigl(\beta(0),\beta(t)\bigr)\geq \frac{A}{\sqrt{1+\frac{A^2}{k^2_{\alpha(0)}}}}\int_{\alpha(0)}^{\alpha(t)}\frac{du}{f(u)^2}\xrightarrow{t\to\infty} \infty,
    \]
    whenever $\gamma$ is non-vertical (i.e., whenever $A\neq 0$), as we wanted to prove. Moreover, from Equation~\eqref{eq:length-horizontal-part} the result also follows if $f(t)\to 0$ as $t\to\infty$. 
    
    Notice that, by \autoref{lem:vertical-rays-only-ones}, when $f(t)\to 0$ as $t\to\infty$ so fast that $\int_0^\infty f(t)dt<\infty$, there are no non-vertical timelike rays, so in those cases the lemma is vacuously true.

    For the last claim, consider a non-vertical future directed timelike ray $\gamma=(\alpha,\beta)$. From \eqref{eq:length-horizontal-part} we deduce
    \[
    d_X\bigl(\beta(0),\beta(t)\bigr)\leq \int_{\alpha(0)}^{\infty}\frac{A \, du}{f(u)^2\sqrt{1+\frac{A^2}{f(u)^2}}}\leq \int_{\alpha(0)}^{\infty}\frac{A\, du}{f(u)^2},
    \]
    and because the last term is finite and independent of $t$, we deduce the result.
\end{proof}

\subsection{Asymptoticity of vertical rays} 

The following results address conditions under which future-directed vertical rays are asymptotic.

\begin{lem}\label{lem:vertical-rays-asymptotic}
    Let $(X,d)$ be a geodesic length space, $f\colon \R\to(0,\infty)$ continuous such that $f(t)\to L\in[0,\infty)$ as $t\to\infty$, and consider $Y=\prescript{-}{}{\mathbb{R}}\times_f X$. Then vertical rays are all asymptotic.
\end{lem}
\begin{proof}
    Consider two vertical rays $\gamma_1,\gamma_2$. By \autoref{lem:cones-past-whole-space}, they both have the whole space as their past. Without loss of generality we can assume that $\alpha_1(0)=\alpha_2(0)=:t_0$ and, therefore, $\alpha_1(s)=\alpha_2(s)=s+t_0$ for every $s\in[0,\infty)$. We want to prove that there is some $c\geq 0$ such that $\gamma_i(s)\ll\gamma_j(s+c)$ for every $s\in[0,\infty)$. This happens if and only if for some $c\geq 0$ one has
    \[
    \int_{\alpha(s)}^{\alpha(s+c)}\frac{dt}{f(t)} = \int_{s+t_0}^{s+c+t_0}\frac{dt}{f(t)} > d_X(\beta_i(s),\beta_j(s+c))=d_0, \quad \forall s\in[0,\infty).
    \]
        
    From the convergence of $f(t)$ as $t\to\infty$, we deduce that there exists some $\tilde{t}$ such that $f(t)\leq \max\{2L,1\}$ for every $t\geq \tilde{t}$. As a consequence, the second term in the previous equation is bounded below by $\smash{\frac{c}{\max\{2L,1\}}}$. So the choice of $c\geq d_0\max\{2L,1\}$ suffices.
\end{proof}

\begin{cor}\label{cor:vertical-not-asymptotic}
    If $f(t)\to \infty$ as $t\to\infty$, then different vertical rays cannot be asymptotic.
\end{cor}
\begin{proof}
    Consider two distinct points $x_1,x_2\in X$ and any two vertical geodesic rays $\gamma_1,\gamma_2$ above them. Call $d_0=d_X(x_1,x_2)>0$. Then
    \[
    \gamma_1(s)\ll\gamma_2(s+c)\iff d_0<\int_{\alpha_1(s)}^{\alpha_2(s+c)}\frac{dt}{f(t)}= \int_{\alpha_1(0)+s}^{\alpha_2(0)+s+c}\frac{dt}{f(t)} \xrightarrow{s\to \infty} 0,
    \]
    so there is no choice of $c$ for which the integrals on the right hand side are strictly greater than $d_0$ for every $s\in[0,\infty)$. As a consequence, no two different vertical rays are asymptotic.
\end{proof}

Notice that in the subcase in which $f(t)\to \infty$ so fast that $\smash{\int_0^\infty \frac{dt}{f(t)}<\infty}$ we already knew (see \autoref{lem:vertical-not-asymptotic}) that different vertical rays had different pasts and, in particular, they cannot be asymptotic.

\subsection{Summary}

Let us now summarize the previous results. For simplicity, we establish the following nomenclature to distinguish the possible cases:

\begin{dfn}\label{def:nomenclature}
Let $f\colon \R\to (0,\infty)$ be continuous and fix $a\in\R$.
\begin{enumerate}
    \item If $f\to0$, we say that it converges to $0$ \textit{quickly} (resp.\ \textit{slowly}) if $\smash{\int_a^\infty f(t) dt<\infty}$ (resp.\ $=\infty$). 
    \item If $f\to\infty$, we say that it diverges \textit{quickly} if $\smash{\int_a^\infty \frac{dt}{f(t)}<\infty}$.
    \item If $f\to\infty$ but $\smash{\int_a^\infty \frac{dt}{f(t)}=\infty}$, we say that $f$ diverges \textit{slowly} if $\smash{\int_a^\infty \frac{dt}{f(t)^2}<\infty}$. Otherwise, we say that it diverges \textit{very slowly}.
\end{enumerate}
\end{dfn}

Of course, the definition is independent of the choice of $a\in\R$, because every condition involving it is satisfied for some choice of $a\in\R$ if and only if it is satisfied for every $a\in\R$.

\begin{rem}\label{rem:convex}
    Recall that if $f\colon \R\to (0,\infty)$ is convex, then either it converges to some non-negative value or it diverges. However, notice that $f$ cannot diverge very slowly according to the previous definition, as whenever it diverges it must do it at least linearly. 
\end{rem}

\begin{thm}
Let $(X,d)$ be a proper geodesic space with infinite diameter and let $f\colon \R\to (0,\infty)$ be locally Lipschitz. Consider the generalised cone $Y=\prescript{-}{}{\mathbb{R}}\times_f X$. Following the nomenclature from \autoref{def:nomenclature}, we summarize the results of the previous subsections in the following table:
\begin{table}[H]
    \centering
    \setlength{\extrarowheight}{2pt}
    \begin{tabular}{|p{3.5cm}||>{\centering\arraybackslash}p{2.8cm}|>{\centering\arraybackslash}p{2.4cm}|>{\centering\arraybackslash}p{2.5cm}|>{\centering\arraybackslash}p{2.4cm}|}\hline
    \setlength{\extrarowheight}{0.5cm}
        & Past of rays is the whole space & \,Vertical rays\newline are asymptotic & Not only vertical rays exist & Horizontal part is a ray\\ \hline\hline
        1. $f\to0$ quickly   & \checkmark (R.~\ref{rem:vertical_rays}, L.~\ref{lem:vertical-rays-only-ones}) & \checkmark (Lem.~\ref{lem:vertical-rays-asymptotic}) & \xmark (Lem.~\ref{lem:vertical-rays-only-ones}) & Not well defined \\ \hline
        2. $f\to0$ slowly   & \checkmark (Lem.~\ref{lem:cones-past-whole-space-ftozero}) & \checkmark (Lem.~\ref{lem:vertical-rays-asymptotic}) & \checkmark (Lem.~\ref{lem:existence-nonvertical-rays}) & \checkmark (Lem.~\ref{lem:horizontal-part-is-ray}) \\ \hline
        3. $f\to L>0$        & \checkmark (Lem.~\ref{lem:cones-past-whole-space}) & \checkmark (Lem.~\ref{lem:vertical-rays-asymptotic}) & \checkmark (Lem.~\ref{lem:existence-nonvertical-rays}) & \checkmark (Lem.~\ref{lem:horizontal-part-is-ray}) \\ \hline
        4. $f\to\infty$ very slowly    & \checkmark (Lem.~\ref{lem:cones-past-whole-space}) & \xmark (Cor.~\ref{cor:vertical-not-asymptotic}) & \checkmark (Lem.~\ref{lem:existence-nonvertical-rays}) & \checkmark (Lem.~\ref{lem:horizontal-part-is-ray}) \\ \hline
        5. $f\to\infty$ slowly & \checkmark (Lem.~\ref{lem:cones-past-whole-space}) & \xmark (Cor.~\ref{cor:vertical-not-asymptotic})  & \checkmark (Lem.~\ref{lem:existence-nonvertical-rays}) & \xmark (Lem.~\ref{lem:horizontal-part-is-ray})\\ \hline
        6. $f\to\infty$ quickly  & \xmark (Rem.~\ref{rem:vertical_rays})    & \xmark (Lem.~\ref{lem:vertical-not-asymptotic}) & \checkmark (Lem.~\ref{lem:existence-nonvertical-rays}) & \xmark (Lem.~\ref{lem:horizontal-part-is-ray})\\ \hline
    \end{tabular}
    \caption{Summary of the results \ref{lem:cones-past-whole-space}--\ref{cor:vertical-not-asymptotic}.}\label{tab:summary-cones-1}
\end{table}
\end{thm}

As discussed above, we are interested in working with generalized cones satisfying a global curvature bound from above by $0$, so as to be able to use the results from the previous sections. Therefore, from now on we will almost exclusively work with convex warping functions and, in particular, warping functions will not diverge very slowly (\autoref{rem:convex}). 

\subsection{Structure of the ideal boundary of generalized cones}

The goal of the remaining part of the section is to obtain the structure of some classes of generalized cones, namely those described in \autoref{tab:summary-cones-1} except the fourth class.

\begin{lem}\label{lem:unique-vertical-same-past}
    Let $(X,d)$ be a complete geodesic space with infinite diameter, $f\colon \R\to(0,\infty)$ diverging quickly (\autoref{def:nomenclature}). For every future-directed timelike geodesic ray in $Y=\prescript{-}{}{\mathbb{R}}\times_f X$ there is a unique\footnote{Uniqueness of vertical rays is here and henceforth understood in the sense that any other vertical ray with the same past will have the same horizontal component.} vertical ray with the same past. 
\end{lem}
\begin{proof}
    Take $\gamma=(\alpha,\beta)$ a future-directed timelike geodesic ray. If $\gamma$ is vertical there is nothing to prove, so assume that this is not the case. From \autoref{lem:horizontal-part-is-ray} and completeness of $X$, we deduce that $\beta(t)\to x$ for some $x\in X$. So let us consider the vertical ray $\gamma_x(t)=(t,x)$, for all $t\in[0,\infty)$. To prove the lemma it is enough to prove that each ray is contained in the past of the other one.

    Consider $t_0\in[0,\infty)$. One has $\gamma(t_0)\in I^-(\gamma_x)$ if and only if there is some $s\in [0,\infty)$ such that
    \[
    \int_{\alpha(t_0)}^s \frac{du}{f(u)}> d_X\bigl(\beta(t_0),x\bigr)=\int_{t_0}^\infty \frac{A\, du}{f(\alpha(u))^2}=\int_{\alpha(t_0)}^\infty \frac{A\,du}{f(u)\sqrt{f(u)^2+A^2}}\quad \left(<\int_{\alpha(t_0)}^\infty \frac{du}{f(u)}\right),
    \]
    from where one immediately deduces the existence of such a value $s$.

    For the converse inclusion, take again some $t_0\in [0,\infty)$. Similarly, $(t_0,x)\in I^-(\gamma)$ if and only if there is some $s\in[0,\infty)$ such that 
    \[
    \int_{t_0}^{\alpha(s)} \frac{du}{f(u)} > d_X\bigl(\beta(s),x),
    \]
    but the right hand side converges to $0$ as $s\to \infty$, whereas the left hand side is increasing with $s$.

    Finally, the uniqueness follows from the fact that, in these conditions, different vertical geodesic rays have different pasts (\autoref{lem:vertical-not-asymptotic}).
\end{proof}

\begin{rem}
    As in \autoref{rem:vertical_rays}, under the assumptions of \autoref{lem:unique-vertical-same-past}, the future causal boundary of the generalized cone $Y=\prescript{-}{}{\mathbb{R}}\times_f X$ (which, in our case, is the set of all timelike pasts of inextendible timelike curves) is a copy at infinity of the fiber $\hat{\bd} Y = \{\infty\} \times X$, being each of its elements the past of a vertical curve \cite{flores2007product}. Therefore, the past of every future-directed timelike curve coincides with the past of some vertical ray. In other words, \autoref{lem:unique-vertical-same-past} is also true by changing in the statement timelike geodesic \textit{ray} by timelike \textit{curve}.
\end{rem}

\begin{lem}\label{lem:same-past-thus-asymptotic}
    Let $(X,d)$ be a complete geodesic space, $f\colon \R\to(0,\infty)$ convex diverging quickly (\autoref{def:nomenclature}). If two future directed timelike geodesic rays in $Y=\prescript{-}{}{\mathbb{R}}\times_f X$ have the same past, then they are asymptotic. 
\end{lem}
\begin{proof}
    We will prove that every future directed timelike geodesic is asymptotic to the unique vertical ray with the same past (cf.\ \autoref{lem:unique-vertical-same-past}), which by transitivity will imply the result. 

    Consider $\gamma=(\alpha,\beta)$ a (non vertical, otherwise there is nothing to prove) future directed timelike geodesic ray and call $\gamma_x$ the unique vertical geodesic ray with the same past. Without loss of generality, one may assume that $\alpha(0)=0$ and take $\gamma_x$ to be such that $\gamma_x(t)=(t,x)$ for every $t\in[0,\infty)$. We want to show that there exists some $c>0$ such that for every $t\in [0,\infty)$ one has (a) $\gamma_x(t)\ll \gamma(t+c)$, and (b) $\gamma(t)\ll\gamma_x(t+c)$. In other words, for $v_\beta(s)=A f(\alpha(s))^{-2}$:
    
    \begin{alignat*}{3}
      \text{(a)} \quad  && d_X\bigl(\beta(t+c),x\bigr) &= \int_{\alpha(t+c)}^{\infty} \frac{A\, du}{f(u)\sqrt{f(u)^2+A^2}} & & < \int_{t}^{\alpha(t+c)} \frac{du}{f(u)}, \\
      \text{(b)} \quad && d_X\bigl(\beta(t),x\bigr) &= \int_{\alpha(t)}^{\infty} \frac{A\, du}{f(u)\sqrt{f(u)^2+A^2}} & &< \int_{\alpha(t)}^{t+c}\frac{du}{f(u)}.
    \end{alignat*}
    
    As $f$ is convex, in particular it is almost everywhere differentiable. In addition, $f(t)\xrightarrow{t\to\infty}\infty$, so that there exists some $u_0>0$ such that $m_0:=f'(u_0)>0$ and $f'(u)\geq m_0$ for every differentiability point $u\geq u_0$ of $f$. In particular, $f$ is strictly increasing in $[u_0,\infty)$. The convexity of $f$ also implies that $f(u)\geq f(v)+m_0(u-v)$ for every $u\geq v\geq u_0$. 
    
    Let us prove (a) first. Choose some ``auxiliary" $c>u_0$. On the one hand, one has $\dot{\alpha}\geq 1$ and $\alpha(0)=0$, so $\alpha(t+c)\geq t+c$. Therefore, 
    \[
    \int_t^{\alpha(t+c)}\frac{du}{f(u)}\geq \int_{\max\{t,u_0\}}^{t+c} \frac{du}{f(u)} \geq \frac{t+c-\max\{t,u_0\}}{f(t+c)},
    \]
    where in the last inequality we used that $f$ is increasing in $[u_0,\infty)$ and $t+ c \geq \max\{t,u_0\}$. On the other hand, one has $\sqrt{f(u)^2+A^2}\geq f(u)\geq f(v)+m_0(u-v)$ for every $u\geq v\geq u_0$, so taking $v=t+c$:
    \[
    \int_{\alpha(t+c)}^{\infty} \frac{A\, du}{f(u)\sqrt{f(u)^2+A^2}}\leq
    \int_{t+c}^{\infty} \frac{A\, du}{f(u)^2}\leq
    \int_{t+c}^{\infty} \frac{A\, du}{\bigl(f(t+c)+m_0(u-(t+c))\bigr)^2}=\frac{A}{m_0f(t+c)}.
    \]

    Combining both results one obtains that
    \[
    \int_{t}^{\alpha(t+c)} \frac{du}{f(u)} - \int_{\alpha(t+c)}^{\infty} \frac{A\, du}{f(u)\sqrt{f(u)^2+A^2}} \geq \frac{t+c-\max\{t,u_0\}-\frac{A}{m_0}}{f(t+c)},
    \]
    and in particular the difference is strictly positive whenever $c> u_0+\frac{A}{m_0}$, which is compatible with the previous condition $c>u_0$.

    Let us now prove (b). On the one hand, using that $1-\frac{1}{\sqrt{1+x}}\leq\frac{x}{2}$, one deduces
    \[
    \alpha(t)-t=\int_0^{\alpha(t)} \Biggl( 1-\frac{1}{\sqrt{1+\frac{A^2}{f(u)^2}}}\Biggr) du \leq \frac{A^2}{2} \int_0^\infty \frac{du}{f(u)^2}<\infty,
    \]
    and in particular, as the left hand side is increasing, the limit $0\leq L:=\lim_{t\to\infty} \bigl(\alpha(t)-t\bigr)$ exists and is finite. As a consequence, $\alpha(t)\in[t,t+L]$ and, for every $c\geq \max\{L,u_0\}$, one has
    \[
    \int_{\alpha(t)}^{t+c}\frac{du}{f(u)}\geq \int_{\max\{t+L,u_0\}}^{t+c}\frac{du}{f(u)} \geq \frac{c-\max\{L,u_0\}}{f(t+c)}.
    \]

    On the other hand, from the proof of (a), we already have
    \[
    \int_{t+c}^{\infty} \frac{A\, du}{f(u)\sqrt{f(u)^2+A^2}}\leq \int_{t+c}^{\infty} \frac{A\, du}{f(u)^2}\leq \frac{A}{m_0f(t+c)}.
    \]

    Combining again both results, we deduce that 
    \[
    \int_{\alpha(t)}^{t+c}\frac{du}{f(u)}-\int_{\alpha(t)}^{\infty} \frac{A\, du}{f(u)\sqrt{f(u)^2+A^2}} \geq \frac{c-\max\{L,u_0\}-\frac{A}{m_0}}{f(t+c)},
    \]
    so the difference is strictly positive whenever $c>\max\{L,u_0\}+\frac{A}{m_0}$, which is also compatible with the condition obtained in part (a).
\end{proof}

\begin{cor}\label{cor:unique-vertical-asymptotic}
    Let $(X,d)$ be a complete geodesic space and let $f\colon \R\to(0,\infty)$ be convex and diverging. For every future directed timelike geodesic ray in $Y=\prescript{-}{}{\mathbb{R}}\times_f X$ there is a unique vertical ray asymptotic to it. 
\end{cor}
\begin{proof}
    Let $\gamma$ be a future directed timelike geodesic ray in $Y$. Assume in the first place that $f$ diverges quickly (\autoref{def:nomenclature}). Then, there is a unique vertical ray with the same past (\autoref{lem:unique-vertical-same-past}) and that ray is asymptotic to $\gamma$ (\autoref{lem:same-past-thus-asymptotic}). Any other vertical ray has a different past and therefore cannot be asymptotic to $\gamma$.

    If $f$ diverges slowly (\autoref{def:nomenclature}), every future directed timelike geodesic ray $\gamma$ has the whole space as its timelike past. However, the length of the horizontal part of $\gamma$ is finite (\autoref{lem:horizontal-part-is-ray}) and thus completeness of $X$ guarantees $\beta(t)\to x$ for some $x\in X$. The rest of the proof of \autoref{lem:same-past-thus-asymptotic} still works, so we deduce that $\gamma_x$ is asymptotic to $\gamma$. Furthermore, \autoref{cor:vertical-not-asymptotic} ensures that different vertical rays cannot be asymptotic, from where we deduce the uniqueness.
    
    In the proof of \autoref{lem:same-past-thus-asymptotic}, the condition $\smash{\int_a^\infty \frac{dt}{f(t)^2}<\infty}$ is needed in order for the limit $L$ to be finite. In other words, the proof would not work if $f$ diverged very slowly. In any case, this cannot happen if $f$ is assumed to be convex and diverging (\autoref{rem:convex}). 
\end{proof}

The last few results (Lemmata~\ref{lem:unique-vertical-same-past} and \ref{lem:same-past-thus-asymptotic}, and \autoref{cor:unique-vertical-asymptotic}) deal with diverging warping functions. In these cases, we have proved that the ideal boundary is a copy of the base $X$. Moreover in the case in which the warping function diverges quickly, the timelike pasts generated by different ideal points are different and, by \autoref{pop:different-pasts-infinite-angle}, the angular distance between any pair of different ideal points is infinite.

Now we turn to the remaining possibilities for a convex warping function $f\colon \R\to\infty$: namely that it converges to $0$ or a positive value. The first of these results (\autoref{lem:asymptotic-thus-same-A}) shows that if two rays are asymptotic, then, up to translation of the parameter, their vertical parts go to infinity at the same rate. Similarly, the subsequent two results (Lemmata~\ref{lem:lorentzian-asymptotic-thus-metric-asymptotic-caseL} and \ref{lem:lorentzian-asymptotic-thus-metric-asymptotic-case0}) show that, under appropriate conditions, if two rays are asymptotic, then their horizontal parts are asymptotic in the metric sense.

\begin{lem}\label{lem:asymptotic-thus-same-A}
    Let $(X,d)$ be a geodesic space and $f\colon \R\to (0,\infty)$ continuous and non-diverging. Let $\gamma_i=(\alpha_i,\beta_i)$ be two asymptotic future-directed timelike geodesic rays and call $A_i\geq0$ the constants such that $v_{\beta_i}=A_i(f\circ\alpha_i)^{-2}$. Then $A_1=A_2$ and, in particular, there exists some $a\in \R$ such that $\dot{\alpha}_1(t)=\dot{\alpha}_2(t+a)$.
\end{lem}
\begin{proof}
    If $f\to 0$ quickly only vertical rays exist, then $A_1=A_2=0$ (\autoref{lem:vertical-rays-only-ones}) and there is nothing else to prove. So we will work in the other cases, i.e., if $f\to 0$ slowly or $f\to L>0$.

    Call $T_i:=\alpha_i^{-1}\colon [\alpha_i(0),\infty)\to [0,\infty)$, which is well-defined as the $\alpha_i$'s are strictly increasing and diverging. Denote $\alpha_0=\max\{\alpha_1(0),\alpha_2(0)\}$ and $\Delta=T_2-T_1\colon [\alpha_0,\infty)\to\R$. The condition $\gamma_1\sim \gamma_2$ implies that there is some $c>0$ such that for every $t\in[0,\infty)$ one has $\alpha_i(t+c) > \alpha_j(t)$. Applying the maps $T_1,T_2$ we deduce, for every $t\geq t_0:=\max\{T_1(\alpha_0),T_2(\alpha_0)\}$,
    \[
    t+c>T_i(\alpha_j(t))\implies T_i(\alpha_j(t))-T_j(\alpha_j(t))<c\implies \bigl|\Delta(\alpha_j(t))\bigr|<c,
    \]
    and as a consequence, for $u\geq\max\{\alpha_1(t_0),\alpha_2(t_0)\}$ or, what matters, for large enough $u$, one has $|\Delta(u)|<c$. Notice also that 
    \[
    \Delta'(u)=T_2'(u)-T_1'(u)=\frac{f(u)}{\sqrt{f(u)^2+A_2^2}}-\frac{f(u)}{\sqrt{f(u)^2+A_1^2}}.
    \]

    Assume now that $f\to L>0$ and, for a contradiction, assume that $A_1>A_2$. Then $\lim_{u\to\infty}\Delta'(u)>0$. In particular, for every $\varepsilon>0$ there exists some $u_0>0$ such that $\Delta'(u)\geq \varepsilon$, for every $u\geq u_0$. So for every $u\geq u_0$ one has
    \[
    \Delta(u)=\Delta(u_0)+\int_{u_0}^u \Delta'(x)dx\geq \Delta(u_0)+\varepsilon(u-u_0)\xrightarrow{u\to\infty} \infty,
    \]
    which contradicts the assumption that $|\Delta(u)|<c$ for large enough $u$.

    Assume otherwise that $f\to 0$ slowly. Assume again for a contradiction that $A_1>A_2$. As in the previous step, we would like to find a uniform lower bound on $\Delta'(u)$ to achieve a contradiction. To that end, consider $\lambda>0$ and define $\smash{\phi(A;\lambda):=\frac{\lambda}{\sqrt{A^2+\lambda^2}}}$. Notice that $T_i'(u)=\phi(A_i;f(u))$. By the Mean Value Theorem, there exists some $\tilde{A}\in(A_2,A_1)$ such that
    \[
    \phi(A_2;\lambda)-\phi(A_1;\lambda)=(A_1-A_2)\frac{\tilde{A}\lambda}{(\lambda^2+\tilde{A}^2)^{3/2}}.
    \]
    
    For fixed $\lambda$, the fraction on the right decreases (w.r.t.\ $\tilde{A}$) for $\smash{\tilde{A}\geq \frac{\lambda}{\sqrt{2}}}$, so taking $\lambda$ so small that both $\smash{A_1,A_2\geq \frac{\lambda}{\sqrt{2}}}$, one can bound it by
    \[
    \phi(A_2;\lambda)-\phi(A_1;\lambda)\geq(A_1-A_2)\frac{A_1\lambda}{(\lambda^2+A_1^2)^{3/2}}\geq \frac{A_1-A_2}{2^{3/2}A_1^2}\lambda.
    \]
    
    Now, as $f\to 0$, there exists some $u_0>0$ such that $f(u)\leq \smash{\frac{\lambda}{\sqrt{2}}}$ for every $u\geq u_0$. So for $u\geq u_0$, one has
    \[
    \Delta(u)=\Delta(u_0)+\int_{u_0}^u \Delta'(x)dx\geq \Delta(u_0)+\frac{A_1-A_2}{2^{3/2}A_1^2} \int_{u_0}^u f(s)ds\xrightarrow{u\to\infty} \infty,
    \]
    which again contradicts the assumption that $|\Delta(u)|<c$ for large enough $u$. In both steps, the case $A_2>A_1$ is analogous.
\end{proof}

\begin{lem}\label{lem:lorentzian-asymptotic-thus-metric-asymptotic-caseL}
    Let $(X,d)$ be a geodesic space and $f\colon \R\to (0,\infty)$ convex and converging to $L>0$. Let $\gamma_1,\gamma_2$ be two non-vertical asymptotic future-directed timelike geodesic rays. Then, the arclength parametrizations $\tilde{\beta}_1,\tilde{\beta}_2$ of their horizontal components $\beta_1,\beta_2$ are asymptotic in the metric sense (\autoref{def:metric-geodesic-rays-asymptoticity}).
\end{lem}
\begin{proof}
    From \autoref{lem:asymptotic-thus-same-A}, we have that $A_1=A_2=:A$ and, in particular, there exists some $d\in \R$ such that $\alpha_2(t)=\alpha_1(t+d)$, for every $t\in[\max\{0,-d\},\infty)$. For simplicity, assume that $\alpha_1(0)\leq \alpha_2(0)$. In this case, $d=\alpha_1^{-1}(\alpha_2(0))\geq 0$.
  
    Let us now assume that $f\to L>0$. By the convexity of $f$, it must be decreasing. In particular, there is some value $u_0>0$ such that $f(u)\leq 2L$ for every $u\geq u_0$. Moreover, $\dot{\alpha}_i$ is monotone increasing and converges to $m:=\sqrt{1+A^2/L^2}<\infty$.
    
    Consider the arclength parametrizations $\tilde{\beta}_i$ of $\beta_i$. One can easily check (see Equation~\eqref{eq:proof-existence-nonvert-rays-reparam}) that for every $t\in[0,\infty)$,
    \begin{equation}\label{eq:cones-reparam-T}
    \tilde{\beta}_i(t)=\beta_i\bigl(T_i^{-1}(t)\bigr),\qquad \text{where} \quad T_i(s)=\int_0^s v_{\beta_i}(u)\,du.
    \end{equation}

    Notice that the relation between $\alpha_1$ and $\alpha_2$ gives that $T_2(t)=T_1(t+d)-T_1(d)$. Call $b=T_1(d)$. In other words, $\tilde{\beta}_2(t)=\beta_2\bigl(T_1^{-1}(t+b)-d\bigr)$, for all $t\in[0,\infty)$.

    We want to prove that $\tilde{\beta}_1\sim\tilde{\beta}_2$ or, what is the same, that $d_X(\tilde{\beta}_1(t),\tilde{\beta}_2(t))$ is bounded. From the asymptoticity of $\gamma_1$ and $\gamma_2$ we know that there exists some $c>0$ such that
    \[
    d_X\bigl(\beta_i(t),\beta_j(t+c)\bigr)<\int_{\alpha_i(t)}^{\alpha_j(t+c)} \frac{du}{f(u)}, \qquad \forall t\in[0,\infty).
    \]

    By the triangle inequality of $d_X$, we have
    \[
    \begin{aligned}
    d_X\bigl(\tilde{\beta}_1(t),\tilde{\beta}_2(t)\bigr)&
    = 
    d_X\Bigl(\beta_1\bigl(T_1^{-1}(t)\bigr), \beta_2\bigl(T_1^{-1}(t+b)-d\bigr)\Bigr)\leq\\
    &\leq
    d_X\Bigl(\beta_1\bigl(T_1^{-1}(t)\bigr),\beta_2\bigl(T_1^{-1}(t)+c\bigr)\Bigr) + d_X\Bigl(\beta_2\bigl(T_1^{-1}(t)+c\bigr), \beta_2\bigl(T_1^{-1}(t+b)-d\bigr)\Bigr).
    \end{aligned}
    \]

    On the one hand, to bound the first term we can use the asymptoticity of $\gamma_1$ and $\gamma_2$ to obtain
    \[
    d_X\Bigl(\beta_1\bigl(T_1^{-1}(t)\bigr),\beta_2\bigl(T_1^{-1}(t)+c\bigr)\Bigr) <
    \int^{\alpha_1(T_1^{-1}(t)+c+d)}_{\alpha_1(T_1^{-1}(t))} \frac{du}{f(u)} = \int_{T_1^{-1}(t)}^{T_1^{-1}(t)+c+d} \frac{\dot{\alpha}_1(s)\, ds}{f(\alpha_1(s))}\leq (c+d)\frac{m}{L}.
    \]

    On the other hand, to bound the second term we have
    \[
    d_X\Bigl(\beta_2\bigl(T_1^{-1}(t)+c\bigr), \beta_2\bigl(T_1^{-1}(t+b)-d\bigr)\Bigr) =
    \left|
    \int_{T_1^{-1}(t)+c}^{T_1^{-1}(t+b)-d} \frac{A\, du}{f(\alpha_2(u))^2}
    \right| \leq
    \left|
    T_1^{-1}(t+b)-d-T_1^{-1}(t)-c
    \right| \frac{A}{L^2},
    \]
    so that one only needs to bound the term inside the absolute value in the right hand side. To that end, notice that for $t$ large enough (explicitly, such that $\alpha_1(T_1^{-1}(t))\geq u_0$), one has
    \[
    b=\int_{T_1^{-1}(t)}^{T_1^{-1}(t+b)} \frac{A\, du}{f(\alpha_1(u))^2} \geq \bigl(T_1^{-1}(t+b)-T_1^{-1}(t)\bigr)\frac{A}{4L^2},
    \]
    from where one deduces that $T_1^{-1}(t+b)-T_1^{-1}(t)>0$ is bounded.
\end{proof}

\begin{lem}\label{lem:lorentzian-asymptotic-thus-metric-asymptotic-case0}
    Let $(X,d)$ be a complete $\CAT(0)$ space and $f\colon \R\to (0,\infty)$ convex and converging to $0$ slowly (see \autoref{def:nomenclature}). Let $\gamma_1,\gamma_2$ be two non-vertical asymptotic future-directed timelike geodesic rays. Then, the arclength parametrizations $\smash{\tilde{\beta}_1,\tilde{\beta}_2}$ of their horizontal components $\beta_1,\beta_2$ are asymptotic in the metric sense (\autoref{def:metric-geodesic-rays-asymptoticity}).
\end{lem}
\begin{proof}
    Using the notation from the previous proof, we define 
    \[
    h_i(t):=\int_0^{\alpha_i(t)}\frac{ds}{f(s)}-T_i(t).
    \]

    One can readily check that $h_i'(t)\in(0,\frac{1}{2A}]$ and, therefore, $h_i(t)\leq \frac{t}{2A}$. 

    The asymptoticity condition of $\gamma_1$ and $\gamma_2$ implies the existence of some $c>0$ such that for every $t\in[0,\infty)$ one has 
    \[
    d(\beta_1(t),\beta_2(t+c))<h_2(t+c)-h_1(t)+T_2(t+c)-T_1(t)
    \]
    or, for the arclength parametrizations and calling for simplicity $r_1:=T_1(t)$ and $r_2:=T_2(t+c)$,
    \[
    d(\tilde{\beta}_1(r_1),\tilde{\beta}_2(r_2))<h_2(t+2)-h_1(t)+r_2-r_1\leq r_2-r_1+\frac{t+c}{2A}.
    \]

    Now fix some $p\in X$ and define $x_t:=\tilde{\beta}_1(r_1)$, $y_t:=\tilde{\beta}_2(r_2)$, $a_t:=d(p,x_t)$, $b_t:=d(p,y_t)$, $c_t:=d(x_t,y_t)$. We will work with the Euclidean comparison triangle $\triangle px_ty_t$, of side lengths $a_t$, $b_t$ and $c_t$.

    Call $\varphi_t$ the comparison angle at $p$. By the Law of Cosines, we have
    \[
    c_t^2=a_t^2+b_t^2-2a_tb_t\cos(\varphi_t)\implies c_t-(b_t-a_t)=\frac{2a_tb_t(1-\cos\varphi_t)}{c_t+(b_t-a_t)}\geq a_t(1-\cos\varphi_t),
    \]
    where in the last step we used the triangle inequality $c_t\leq a_t+b_t$.

    On another note, we have
    \[
    |b_t-a_t-(r_2-r_1)|\leq |b_t-r_2|+|a_t-r_1|\leq a_0+b_0=:D,
    \]
    where $a_0:=d(p,\tilde{\beta}_1(0))$ and $b_0:=d(p,\tilde{\beta}_2(0))$, and again we used the triangle inequality $|a_t-r_1|\leq a_0$ and $|b_t-r_2|\leq b_0$. To sum up, we have
    \[
    a_t(1-\cos\varphi_t)\leq c_t-(b_t-a_t)\leq c_t-(r_2-r_1)+D<\frac{t+c}{2A}+D.
    \]

    Now, using that $1-\cos\varphi\geq \frac{\varphi^2}{9}$ whenever $|\varphi|\leq \pi$, and that $a_t\geq r_1-a_0$, we deduce
    \[
    \varphi_t^2\leq \frac{9}{2A}\left(\frac{t+c+2AD}{r_1-a_0}\right)\xrightarrow{t\to\infty}0, 
    \]
    where in the last step we used the fact that $r_1$ grows faster than linearly as
    \[
    \frac{r_1}{t}=\frac{T_1(t)}{t}\geq\frac{1}{t}\int_{t/2}^t \frac{A\, ds}{f(\alpha_1(s))^2}\geq \frac{A}{2f(\alpha_1(t/2))^2}\xrightarrow{t\to \infty}\infty.
    \]

    From the convergence of the angle $\varphi_t\to0$ we will deduce that the rays $\tilde{\beta}_1$ and $\tilde{\beta}_2$ are asymptotic in the metric sense. Indeed, as in the proof of \cite[Proposition~II.8.2]{BH}, the geodesic intervals $[p,x_n]$ and $[p,y_n]$ converge pointwise to geodesic rays $\beta'_{1,p},\beta'_{2,p}$ which can be shown to be (cf.\ the aforementioned result) the unique geodesic rays asymptotic to $\smash{\tilde{\beta}_1}$ and $\smash{\tilde{\beta}_2}$, respectively, starting at $p$.

    Moreover, the continuity of the angle between curves starting at a fixed point \cite[Proposition~I.3.3]{BH} ensures that
    \[
    \ma_p(\beta_{1,p}',\beta_{2,p}')=\lim_{t\to\infty} \varphi_t=0.
    \]

    From the arbitrariness of the point $p$, we deduce that (see Eq.~\eqref{eq:def-metric-angular-distance}) 
    \[
    \ma\bigl([\tilde{\beta}_1],[\tilde{\beta}_2]\bigr)=\sup_{p\in X}\ma_p(\beta'_{1,p},\beta'_{2,p})=0,
    \]
    and because the angular distance in a complete $\CAT(0)$ space is indeed a distance \cite[Proposition~II.9.5]{BH}, we deduce that  $\smash{\tilde{\beta}_1\sim\tilde{\beta}_2}$.
\end{proof}

We turn now to the description of the ideal boundary of generalized cones with convex and non-diverging warping function.

\begin{lem}\label{lem:bijection-product-idealboundary-caseL}
    Let $(X,d)$ be a geodesic space and $f\colon \R\to (0,\infty)$ convex and converging to $L>0$. Let $Y=\prescript{-}{}{\mathbb{R}}\times_f X$ and denote $\sim$ the minimal equivalence relation in $[0,\infty)\times \bd X$ which identifies points whose first component is $0$. For a ray $\gamma=(\alpha,\beta)$, let $A\geq 0$ be the constant such that $v_\beta=A(f\circ\alpha)^{-2}$. Then the map
    \[
    \begin{aligned}
    \Phi\colon \bd^+Y &\longrightarrow \faktor{[0,\infty)\times \bd X}{\sim}\\
     \xi \;\; &\longmapsto  \bigl[\bigl(\sinh A,[\mathrm{pr}_X \xi]\bigr)\bigr],
    \end{aligned}
    \]
    is bijective.
\end{lem}
\begin{proof}
    The map $\Phi$ is well-defined by Lemmata~\ref{lem:asymptotic-thus-same-A} and \ref{lem:lorentzian-asymptotic-thus-metric-asymptotic-caseL}. Notice that the codomain of $\Phi$ is just the set $[0,\infty)\times \bd X$ with the subset $\{0\}\times \bd X$ identified to a point, which is usually called the \textit{apex}.
    
    Now, surjectivity follows from the proof of \autoref{lem:existence-nonvertical-rays}: morally, for every $A>0$, every horizontal ray $\tilde{\beta}$ can be lifted (after appropriate reparametrization) to a future-directed timelike geodesic ray $\gamma$ such that the metric speed of its horizontal part $\beta$ is $v_\beta=A(f\circ\alpha)^{-2}$; such a ray is sent by $\Phi$ to $\smash{(\sinh A,\tilde{\beta}(\infty))}$. For $A=0$, notice that the vertical rays are sent by $\Phi$ to the apex.

    To prove injectivity notice first that the only rays sent by $\Phi$ to the apex are the vertical ones, i.e., those with $A=0$. These are all asymptotic by \autoref{lem:vertical-rays-asymptotic}. So for the rest of the proof assume that $A>0$ and take $\tilde{\beta}_1,\tilde{\beta}_2$ asymptotic geodesic rays (parametrized by arclength) in $X$, and two real numbers $\alpha_1,\alpha_2$.

    Let us lift $\tilde{\beta}_1,\tilde{\beta}_2$ to two future-directed timelike geodesic rays $\gamma_i=(\alpha_i,\beta_i)$ as in the proof of \autoref{lem:existence-nonvertical-rays} by declaring
    \[
    \dot{\alpha}_i^2=1+\frac{A^2}{(f\circ \alpha_i)^2},\qquad \alpha_i(0)=\alpha_i,\quad \text{and}\quad 
    \beta_i(t)=\tilde{\beta}_i(T_i(t)), \quad\text{where}\quad 
    T_i(t)=\int_0^t \frac{A\,du}{f(\alpha_i(u))^2}.
    \]

    By the previous results, these rays $\gamma_1,\gamma_2$ are sent via $\Phi$ to the same point. We will now prove that they are asymptotic, which will conclude the proof.

    First assume, without loss of generality, that $\alpha_1\leq\alpha_2$. Then we have, as in the proof of \autoref{lem:lorentzian-asymptotic-thus-metric-asymptotic-caseL}, that $\alpha_2(t)=\alpha_1(t+d)$ for some fixed $d\geq0$. We also have that $T_2(t)=T_1(t+d)-T_1(d)$. Let us call $D:=\sup\{d(\tilde{\beta}_1(t),\tilde{\beta}_2(t))\mid t\in [0,\infty)\}$, which is finite by the asymptoticity $\tilde{\beta}_1\sim\tilde{\beta}_2$.

    On the one hand, for $c>0$ and calling $b:=T_1(d)$ as before, we have
    \[
    d(\beta_1(t),\beta_2(t+c))\leq D+\bigl|T_1(t+c+d)-T_1(d)-T_1(t)\bigr|\leq D+b+\int_t^{t+c+d} \frac{A\, ds}{f(\alpha_1(s))^2}.
    \]

    On the other hand, one can check that 
    \[
    \int_{\alpha_1(t)}^{\alpha_2(t+c)} \frac{du}{f(u)}=\int_t^{t+c+d} \frac{\sqrt{f(\alpha_1(s))^2+A^2}}{f(\alpha_1(s))^2}\,ds=
    \int_t^{t+c+d}\left(
    \frac{1}{\sqrt{f(\alpha_1(s))^2+A^2}+A}+\frac{A}{f(\alpha_1(s))^2}
    \right)ds.
    \]

    Notice that, as $f$ is decreasing, the first fraction in the right hand side of the previous equation is increasing with $s$. In particular, it is not smaller than its value at $s=0$. Call $m$ such a value. So subtracting the previous equations we obtain
    \[
    \int_{\alpha_1(t)}^{\alpha_2(t+c)} \frac{du}{f(u)}-d(\beta_1(t),\beta_2(t+c))\geq 
    \int_t^{t+c+d}
    \frac{1}{\sqrt{f(\alpha_1(s))^2+A^2}+A}-D-b\geq m(c+d)-D-b,
    \]
    which is strictly positive whenever $c>\frac{D+b}{m}-d$. As a consequence, $\gamma_1$ and $\gamma_2$ are asymptotic
\end{proof}

\begin{lem}\label{lem:bijection-product-idealboundary-case0}
    Let $(X,d)$ be a complete $\CAT(0)$ space and $f\colon \R\to (0,\infty)$ convex and converging to $0$ slowly. Let $Y=\prescript{-}{}{\mathbb{R}}\times_f X$ and, as in \autoref{lem:bijection-product-idealboundary-caseL}, let $\sim$ be the minimal equivalence relation in $[0,\infty)\times \bd X$ which identifies points whose first component is $0$. For a ray $\gamma=(\alpha,\beta)$, let $A\geq 0$ be the constant such that $v_\beta=A(f\circ\alpha)^{-2}$. Then, the map
    \[
    \begin{aligned}
    \Phi\colon \bd^+Y &\longrightarrow \faktor{[0,\infty)\times \bd X}{\sim}\\
     \xi \;\; &\longmapsto  \bigl[\bigl(\sinh A,[\mathrm{pr}_X \xi]\bigr)\bigr],
    \end{aligned}
    \]
    is bijective.
\end{lem}
\begin{proof}
    The proof is essentially the same as the proof of \autoref{lem:bijection-product-idealboundary-caseL}. The hypotheses on $(X,d)$ serve to ensure that $\Phi$ is well defined via Lemmata~\ref{lem:asymptotic-thus-same-A} and \ref{lem:lorentzian-asymptotic-thus-metric-asymptotic-case0}.
\end{proof}

\begin{pop}\label{pop:restricted-cone-2}
    The same conclusion of \autoref{lem:restricted-cone}, namely that the future timelike ideal boundaries of $Y:=\smash{\prescript{-}{}{\R}\times_f X}$ and $Y_a:=\smash{\prescript{-}{}{(a,\infty)}\times_f X}$ are isometric, is obtained if one requires $f$ to be convex and converging to $L>0$ but requires only that there is some $a\in \R$ such that $Y_a$ satisfies $\CBA(0)$ globally, rather than the whole cone $Y$ satisfying $\CBA(0)$ globally. Analogously, if one requires $f$ to be convex and converging to $0$ and $X$ to be $\CAT(0)$.
\end{pop}
\begin{proof}
    The only part of the corresponding proof in which the global $\CBA(0)$ condition of $Y$ is used is to prove that when the rays $\xi_1,\xi_2\in\bd^+Y$ generate the same timelike past one has for every $y\in I^-(\xi_i)$ there exists some $y'\in Y_a$ such that
    \[
    \ma^Y_y(\xi_{1},\xi_{2})\leq
    \ma^{Y_a}_{y'}(\xi_{1},\xi_{2}).
    \]

    Here we will prove, using the previous characterization of the ideal boundary as a set, that the angle between two rays at a point $y=(t,x)$ is non-decreasing with $t$, giving the desired inequality for any $y'=(t',x)$ with any $t'>\max\{a,t_0\}$.

    Consider $y=(t_0,x)$ a point in $Y$ and call $f_0:=f(t_0)$. Let $\varepsilon>0$ and call $f_1:=f_0-\varepsilon$. As $f$ is decreasing, there exists $\delta>0$ such that $f(r)\in [f_1, f_0]$ for every $r\in [t_0,t_0+\delta)$. 

    Let us consider the products $Y_{f_0}:=\prescript{-}{}{\mathbb{R}}\times_{f_0} X$ and $Y_{f_1}:=\prescript{-}{}{\mathbb{R}}\times_{f_1} X$. Let $\gamma=(\alpha,\beta)$ be any future-directed timelike curve in $Y_{f_0}$, i.e., $\gamma$ is absolutely continuous (with respect to the product distance), $\dot{\alpha}>0$ and $\dot{\alpha}>f_0v_\beta$ almost everywhere (\autoref{def:generalized-cones}). Assume moreover that the $t$-coordinate of $\gamma$ is contained in $[t_0,t_0+\delta)$. Then, being $f_1\leq f\leq f_0$, $\gamma$ is also future-directed timelike for the Lorentzian structures in $Y_f$ and $Y_{f_1}$.

    The Lorentzian lengths of $\gamma$ (\autoref{def:lorentzian-length-cones}) with respect to the different Lorentzian structures of a curve $\gamma$ whose $t$-coordinate is in the interval $[t_0,t_0+\delta)$, are related by
    \[
    \mathcal{L}_{f_0}(\gamma)\leq \mathcal{L}_f(\gamma)\leq \mathcal{L}_{f_1}(\gamma),
    \]
    where we use the curly notation to distinguish the Lorentzian length from the limit of $f$.

    Similarly, any two points $p,q\in Y$ whose $t$-coordinate is in the interval $[t_0,t_0+\delta)$ which are timelike related with respect to the Lorentzian structure in $Y_{f_0}$ are also timelike related with respect to the Lorentzian structures in $Y_f$ and $Y_{f_1}$. Moreover, their time separations satisfy
    \begin{equation}\label{eq:cones-relation-timesep}
    \begin{aligned}
    \tau_{f_0}(p,q)&=
    \sup\bigl\{
        \mathcal{L}_{f_0}(\gamma)\mid \gamma \text{ is } f_0\text{-timelike from } p \text{ to } q
    \bigr\}\\
    &\leq \sup\bigl\{
        \mathcal{L}_{f_0}(\gamma)\mid \gamma \text{ is } f\text{-timelike from } p \text{ to } q
    \bigr\}\\
    &\leq \sup\bigl\{
        \mathcal{L}_f(\gamma)\mid \gamma \text{ is } f\text{-timelike from } p \text{ to } q
    \bigr\}=\tau_f(p,q)\\
    &\leq \sup\bigl\{
        \mathcal{L}_{f_1}(\gamma)\mid \gamma \text{ is } f\text{-timelike from } p \text{ to } q
    \bigr\}\\
    &\leq \sup\bigl\{
        \mathcal{L}_{f_1}(\gamma)\mid \gamma \text{ is } f_1\text{-timelike from } p \text{ to } q
    \bigr\}=\tau_{f_1}(p,q).
    \end{aligned}
    \end{equation}

    Furthermore, the time separation in the case of a constant warping function is easily computed to be, for $p=(t,x)\ll (s,y)=q$
    \[
    \tau_L\bigl((t_1,x_1),(t_2,x_2)\bigr)=\sqrt{(t_2-t_1)^2-L^2d_X(x_1,x_2)}.
    \]

    The idea now is to consider two pairs $(B_i,\Xi_i)\in [0,\infty)\times \bd X$, build the corresponding future-directed timelike geodesic rays in $Y_f$ with the bijection $\Phi$ from \autoref{lem:bijection-product-idealboundary-caseL} or \autoref{lem:bijection-product-idealboundary-case0}, and bound appropriately the angle between them from above and from below.

    So, consider $(B_i,\Xi_i)\in [0,\infty)\times \bd X$ and call $A_i=\sinh^{-1}(B_i)$. The corresponding geodesic rays $\gamma_i=(\alpha_i,\beta_i)$ are given by $\smash{\dot{\alpha}_i^2=1+A_i^2(f\circ\alpha_i)^{-2}}$ and, calling $\smash{\tilde{\beta}_i}$ the unique (metric) geodesic ray in $X$ (thus parametrized by arclength) in the equivalence class $\Xi_i$ starting at $x$ \cite[Proposition~II.8.2]{BH}, then $\beta_i$ is the reparametrization with metric speed $v_\beta=A_i(f\circ\alpha_i)^{-2}$ (see Equation~\eqref{eq:proof-existence-nonvert-rays-reparam}).
    
    We have (see \autoref{def:upper-angle})
    \begin{equation}\label{eq:estimation-angle-cone}
    \cosh\ma_y^{Y_f}(\gamma_1,\gamma_2)= 
    \limsup_{\substack{(s,t) \in A_f \\ s,t \searrow 0}} \frac{s^2+t^2-\tau_f(\gamma_1(s),\gamma_2(t))^2}{2st}\geq 
    \limsup_{\substack{(s,t) \in A_f \\ s,t \searrow 0}}\frac{s^2+t^2-\tau_{f_1}(\gamma_1(s),\gamma_2(t))^2}{2st},
    \end{equation}
    where in the last inequality we used Equation~\eqref{eq:cones-relation-timesep}. One can easily obtain the analogous bound from above with $\tau_{f_0}$. The time separation on the right hand side can be explicitly computed to be, for points $\gamma_1(s),\gamma_2(t)$ such that $\alpha_1(s),\alpha_2(t)\in[t_0,t_0+\delta)$,
    \[
    \begin{aligned}
    &\tau_{f_1}(\gamma_1(s),\gamma_2(t))^2=(\alpha_2(t)-\alpha_1(s))^2-f_1^2 d_X(\beta_1(s),\beta_2(t))^2\\
    &=\left(\int_0^t \dot{\alpha}_2(u)\,du-\int_0^s \dot{\alpha}_1(u)\,du\right)^2-f_1^2 d_X(\beta_1(s),\beta_2(t))^2\\
    &\leq s^2\left(1+\frac{A_1^2}{f_1^2}\right)+t^2\left(1+\frac{A_2^2}{f_1^2}\right)-2st\sqrt{1+\frac{A_1^2}{f_0^2}}\sqrt{1+\frac{A_2^2}{f_0^2}}-f_1^2d_X(\beta_1(s),\beta_2(t))^2,
    \end{aligned}
    \]
    where in the last step we used the expression of $\dot{\alpha}_i$ and the fact that $f_1\leq f\leq f_0$ for the concerned points.

    Using this bound on the Equation~\eqref{eq:estimation-angle-cone} we obtain
    \[
    \cosh\ma_y^{Y_f}(\gamma_1,\gamma_2)\geq 
    \sqrt{1+\frac{A_1^2}{f_0^2}}\sqrt{1+\frac{A_2^2}{f_0^2}}-
    \liminf_{\substack{(s,t) \in A_f \\ s,t \searrow 0}} \frac{s^2\frac{A_1^2}{f_1^2}+t^2\frac{A_2^2}{f_1^2}- f_1^2 d_X(\beta_1(s),\beta_2(t))^2}{2st}.
    \]

    Notice that the square roots on the right hand side are each precisely $\cosh\theta_i^{f_0}$, where $\theta_i^{f_0}$ are the angles with the vertical ray in $Y_{f_0}$ (see \autoref{rem:non-unit-cones}). Now, calling $\tilde{s}=T_1(s)$ and $\tilde{t}=T_2(t)$, where $T_i$ is defined as in Equation~\eqref{eq:cones-reparam-T}, one has $\smash{s\frac{A_1}{f_0^2}\leq \tilde{s}\leq s\frac{A_1}{f_1^2}}$, and similarly for $\smash{\tilde{t}}$. On another note, notice that the numerator in the fraction of the last term need not be positive.

    Let us denote for simplicity $d=d_X(\beta_1(s),\beta_2(t))=d_X(\tilde{\beta}_1(\tilde{s}),\tilde{\beta}_2(\tilde{t}))$, where the $\tilde{\beta}_i$'s are the arclength parametrizations of the $\beta_i$'s if they are not constant (i.e., if the corresponding $A_i$ is not $0$) and a constant map otherwise. One can bound the fraction in last term of the previous equation by the following, where the $\max$ is needed to choose the bound on the denominator, as the numerator could be positive or negative:
    \[
    \begin{aligned}
    &\frac{s^2\frac{A_1^2}{f_1^2}+t^2\frac{A_2^2}{f_1^2} - f_1^2 d^2}{2st}\leq\max\biggl\{
    \frac{\frac{f_0^4}{f_1^2}(\tilde{s}^{\,2}+\tilde{t}^{\,2}) - f_1^2 d^2}{2\tilde{s}\tilde{t}\frac{f_1^4}{A_1A_2}},\:
    \frac{\frac{f_0^4}{f_1^2}(\tilde{s}^{\,2}+\tilde{t}^{\,2}) - f_1^2d^2}{2\tilde{s}\tilde{t}\frac{f_0^4}{A_1A_2}}
    \biggr\}\\
    &=
    \max\biggl\{
    \sinh\theta_1^{f_1}\sinh\theta_2^{f_1}
    \frac{\bigl(\frac{f_0}{f_1}\bigr)^4(\tilde{s}^{\,2}+\tilde{t}^{\,2})- d^2}{2\tilde{s}\tilde{t}},\:
    \sinh\theta_1^{f_0} \sinh\theta_2^{f_0} {\textstyle\bigl(\frac{f_1}{f_0}\bigr)^2} \frac{\bigl(\frac{f_0}{f_1}\bigr)^4 (\tilde{s}^{\,2}+\tilde{t}^{\,2})- d^2}{2\tilde{s}\tilde{t}}
    \biggr\},
    \end{aligned}
    \]
    where in the last step we used again \autoref{rem:non-unit-cones}.

    The fraction on both terms of the $\max$ on the last line can be decomposed as 
    \begin{equation}\label{eq:auxiliary-decomp-fraction}
    \frac{\bigl(\frac{f_0}{f_1}\bigr)^4 (\tilde{s}^{\,2}+\tilde{t}^{\,2})- d^2}{2\tilde{s}\tilde{t}}=
    \frac{\tilde{s}^{\,2}+\tilde{t}^{\,2}- d_X(\tilde{\beta}_1(\tilde{s}),\tilde{\beta}_2(\tilde{t}))^2}{2\tilde{s}\tilde{t}}+
    a_\varepsilon\,\frac{\tilde{s}^{\,2}+\tilde{t}^{\,2}}{2\tilde{s}\tilde{t}},
    \end{equation}
    where $a_\varepsilon=(1+\frac{\varepsilon}{f_1})^4-1$. Notice that, on the one hand, the fraction in the first term of the right hand side has a limit as $\smash{\tilde{s},\tilde{t}\searrow 0}$ because of the curvature bound of $X$. Moreover, such a limit is, by definition, the cosine of the angle at $x$ formed by the geodesic rays $\tilde{\beta}_1$ and $\tilde{\beta}_2$. On the other hand, we can bound from above the $\liminf$ of the fraction in the second term of the right hand side with the following considerations:
    
    -- The set $A_f$ (see \autoref{def:upper-angle}) corresponding to the Lorentzian structure of $Y_f$ contains the set $A_{f_0}$, corresponding to the Lorentzian structure of $Y_{f_0}$.

    -- We will find a convenient set $G$ contained in $A_{f_0}$, so that we will have $\liminf_{(s,t) \in A_f;\: s,t\searrow0}\leq \liminf_{(s,t) \in G;\: s,t\searrow0}$ and we will be able to bound this last term from above:
    \[
    \begin{aligned}
        A_{f_0}&\supseteq
        \biggl\{
        (s,t)\mid s,t>0,\; \int_0^s \dot{\alpha}_1(u)\,du-\int_0^t \dot{\alpha}_2(u)\,du\geq f_0\Bigl(\int_0^s v_{\beta_1}(u)\,du+\int_0^t v_{\beta_2}(u)\,du\Bigr)
        \biggr\}\\
        &\supseteq 
        \biggl\{
        (s,t)\mid s,t>0,\; 
        s\sqrt{1+\frac{A_1^2}{f_0^2}}-t\sqrt{1+\frac{A_2^2}{f_1^2}}\geq f_0\frac{sA_1+tA_2}{f_1^2}
        \biggr\}\\
        &\supseteq
        \biggl\{
        (s,t)\mid s,t>0,\; 
        sf_1\sqrt{f_0^2+A_1^2}-tf_0\sqrt{f_1^2+A_2^2}=\frac{f_0^2}{f_1}(sA_1+tA_2)
        \biggr\}=:G.
    \end{aligned}
    \]
    Notice that, given $A_1>0$, one has that $G\neq\varnothing$ whenever $\varepsilon>0$ is small enough.
    
    -- On another note, reversing the variable changes between $s/t$ and $\tilde{s}/\tilde{t}$ we have
    \[
    \frac{\tilde{s}^{\,2}+\tilde{t}^{\,2}}{2\tilde{s}\tilde{t}}\leq 
    \Bigl(\frac{f_0}{f_1}\Bigr)^4\:\frac{s^2A_1^2+t^2A_2^2}{2stA_1A_2}.
    \]
    
    -- Now it is readily checked that 
    \[
    \liminf_{\substack{(s,t) \in A_f \\ \tilde{s},\tilde{t} \searrow 0}}
    \frac{\tilde{s}^{\,2}+\tilde{t}^{\,2}}{2\tilde{s}\tilde{t}}\leq
    \Bigl(\frac{f_0}{f_1}\Bigr)^4
    \liminf_{\substack{(s,t) \in G \\ s,t \searrow 0}} \frac{s^2A_1^2+t^2A_2^2}{2stA_1A_2}= 
    \Bigl(\frac{f_0}{f_1}\Bigr)^4 \frac{A_1^2+A_2^2k}{2A_1A_2k}=:C_\varepsilon,
    \]
    where $k$ is a positive value (it is the slope of the line $ks=t$ defining $G$, so it is positive for $\varepsilon$ small enough) which depends only on $f-0$, $\varepsilon$, $A_1$, and $A_2$. Moreover, as $\varepsilon\to0^+$, $k$ converges to a strictly positive value and, therefore, $C_\varepsilon$ also converges to a strictly positive value $C$.

    Finally, using the elementary inequality $\liminf(A_n+B_n)\leq \limsup A_n+\liminf B_n$ on Equation~\eqref{eq:auxiliary-decomp-fraction}, we can continue with the previous bounds by
    \[
    \begin{aligned}
    &\liminf_{\substack{(s,t) \in A_f \\ s,t \searrow 0}}  \frac{s^2\frac{A_1^2}{f_1^2}+t^2\frac{A_2^2}{f_1^2}- f_1^2d_X(\beta_1(s),\beta_2(t))^2}{2st}\leq\\
    &\leq
    \max\Bigl\{
    \sinh\theta_1^{f_1}\sinh\theta_2^{f_1} \bigl(\cos\ma_x^X(\tilde{\beta}_1,\tilde{\beta}_2)+a_\varepsilon C_\varepsilon\bigr),\,
    \sinh\theta_1^{f_0} \sinh\theta_2^{f_0} {\textstyle\bigl(\frac{f_1}{f_0}\bigr)^2} \bigl(\cos\ma_x^X(\tilde{\beta}_1,\tilde{\beta}_2)+a_\varepsilon C_\varepsilon\bigr)
    \Bigr\}.
    \end{aligned}
    \]

    Notice that from the arbitrariness of $\varepsilon>0$, together with the facts that $a_\varepsilon\to 0$, $\cosh\theta_i^{f_1}\to\cosh\theta_i^{f_0}$ and $\sinh\theta_i^{f_1}\to \sinh\theta_i^{f_0}$ as $\varepsilon\searrow0$, one deduces that 
    \[
    \begin{aligned}
    \cosh\ma_y^{Y_f}(\gamma_1,\gamma_2)\geq
    \cosh\theta_1^{f_0}\cosh\theta_2^{f_0}-\sinh\theta_1^{f_0}\sinh\theta_2^{f_0}
    \cos\ma_x^X(\tilde{\beta}_1,\tilde{\beta}_2).
    \end{aligned}
    \]

    For the opposite inequality one can work similarly to obtain an identical bound from above. Notice that in that case it is not necessary to define an analogous set to $G$, as now we just need that 
    \[
    \liminf_{\substack{(s,t) \in A_f \\ \tilde{s},\tilde{t} \searrow 0}}
    \frac{\tilde{s}^{\,2}+\tilde{t}^{\,2}}{2\tilde{s}\tilde{t}}
    \geq
    \liminf_{\tilde{s},\tilde{t} \searrow 0}
    \frac{\tilde{s}^{\,2}+\tilde{t}^{\,2}}{2\tilde{s}\tilde{t}}=1.
    \]

    Having both inequalities, one must have
    \[
    \cosh\ma_y^{Y_f}(\gamma_1,\gamma_2)=\cosh\theta_1^{f_0}\cosh\theta_2^{f_0}-\sinh\theta_1^{f_0}\sinh\theta_2^{f_0}
    \cos\ma_x^X(\tilde{\beta}_1,\tilde{\beta}_2).
    \]

    Therefore, if $x$ is kept constant, the angle between the representatives of the ideal points $\xi_i$ is increasing with the $t$-coordinate of $y=(t_0,x)$. So considering the point $y'=(t',x)$, where $t'>\max\{t_0,a\}$, we obtain the result.
\end{proof}

\begin{lem}
    Let $(X,d)$ be a proper geodesic space and $f\colon \R\to (0,\infty)$ convex and converging to $L>0$. Assume that there exists $\varepsilon>0$ such that $X\in\CAT(-\varepsilon)$. Let $Y_f:=\prescript{-}{}{\mathbb{R}}\times_f X$ and let $Y_L:=\prescript{-}{}{\mathbb{R}}\times_L X$. Then $\bd^+Y_f$ and $\bd^+Y_L$ are isometric. In particular, $\bd^+Y_f$ is isometric to the metric warped product $[0,\infty)\times_{\sinh}\bd X$.
\end{lem}
\begin{proof}
    We will show that the map $\Phi$ from \autoref{lem:bijection-product-idealboundary-caseL}, which is a bijection, is also distance preserving. We will denote $\Phi_f$ the one corresponding to $\bd^+Y_f$, and $\Phi_L$ the one corresponding to $\bd^+Y_L$. 

    Following the reasonings in \autoref{rem:cones-global-curvature} we deduce that there exists some $a\in \R$ such that $Y_a:=\prescript{-}{}{(a,\infty)}\times_f X$ has timelike curvature globally bounded above by $0$. From the proof of \autoref{pop:restricted-cone-2}, we know that for any two pairs $(B_i,\Xi_i)\in [0,\infty)\times \bd X$ one has
    \[
    \cosh\ma_{(t,x)}^{Y_f}(\gamma_1,\gamma_2)\xrightarrow{t\to \infty} \cosh\theta_1^{L}\cosh\theta_2^{L}-\sinh\theta_1^{L}\sinh\theta_2^{L}
    \cos\ma_x^X(\tilde{\beta}_1,\tilde{\beta}_2),
    \]
    where the curves $\gamma_i$ are the representatives at $y=(t,x)$ of $\xi_i^f:=\Phi^{-1}_f(B_i,\Xi_i)\in \bd^+Y_f$. Moreover, the convergence is in a monotonically increasing way. As a consequence, 
    \[
    \cosh\ma^{Y_f}(\xi_1^f,\xi_2^f)=\cosh\theta_1^{L}\cosh\theta_2^{L}-\sinh\theta_1^{L}\sinh\theta_2^{L}
    \cos\ma^X(\tilde{\beta}_1,\tilde{\beta}_2)=
    \cosh\ma^{Y_L}(\xi_1^L,\xi_2^L),
    \]
    from where the result follows.
\end{proof}

\begin{lem}\label{lem:isometry-product-idealboundary-case0}
    Let $(X,d)$ be a complete locally compact $\CAT(-\varepsilon)$ space for some $\varepsilon>0$, and $f\colon \R\to (0,\infty)$ convex and converging to $0$ slowly. Let $Y_f:=\prescript{-}{}{\mathbb{R}}\times_f X$. Then every two different ideal points $\xi_0\in\bd^+Y_f$ are at infinite angular distance from each other. In particular, $\bd^+Y_f$ is not isometric to $[0,\infty)\times_{\sinh}\bd X$, but to the quotient set of \autoref{lem:bijection-product-idealboundary-case0} endowed with the discrete (extended) distance with value $\infty$.
\end{lem}
\begin{proof}
    Again, from the reasonings in \autoref{rem:cones-global-curvature} we deduce that there exists some $a\in \R$ such that $\smash{Y_a:=\prescript{-}{}{(a,\infty)}\times_f X}$ has timelike curvature globally bounded above by $0$. Now, consider $(B_i,\Xi_i)$ in the quotient space $[0,\infty)\times \bd X$ under the equivalence relation that identifies points with first component $0$. From the proof of \autoref{pop:restricted-cone-2}, we know that, calling $\smash{\xi^f_i:=\Phi^{-1}_f(B_i,\Xi_i)\in \bd^+Y_f}$, one has
    \[
    \cosh\ma_{(t,x)}^{Y_f}(\gamma_1,\gamma_2)= \cosh\theta_1^{f(t)}\cosh\theta_2^{f(t)}-\sinh\theta_1^{f(t)}\sinh\theta_2^{f(t)}\cos\ma_x^X(\tilde{\beta}_1,\tilde{\beta}_2).
    \]
    where the curves $\gamma_i$ are the representatives at $y=(t,x)$ of the corresponding rays.
    
    Assume, on the one hand, that $B_1\neq B_2$. Then 
    \[
    \cosh\ma_{(t,x)}^{Y_f}(\gamma_1,\gamma_2)\geq \cosh\bigl(\theta_1^{f(t)}-\theta_2^{f(t)}\bigr) \xrightarrow{t\to\infty}\infty,
    \]
    and the convergence is monotone increasing. Therefore, 
    \[
    \ma^{Y_f}(\xi_1^f,\xi_2^f)=\infty.
    \]
    
    If on the other hand $B_1=B_2$, we have
    \[
    \cosh\ma_{(t,x)}^{Y_f}(\gamma_1,\gamma_2)= 1+ \bigl(1-\cos\ma_x^X(\tilde{\beta}_1,\tilde{\beta}_2)\bigr) \sinh^2\bigl(\theta_1^{f(t)}\bigr),
    \]
    and this term is identically $1$ if the cosine is $1$ (i.e., if $\Xi_1=\Xi_2$), and diverges as $t\to \infty$ otherwise.
\end{proof}

\begin{rem}\label{rem:cone-topology-angular-topology}
    This class of generalized cones provides an example of ideal boundary in which the topology induced by the angular metric does not coincide with the restriction of the cone topology, i.e., the former is strictly finer than the latter (see \autoref{pop:cone-top-coarser}). Indeed, consider $Y_f=\prescript{-}{}{\mathbb{R}}\times_f X$ as in \autoref{lem:isometry-product-idealboundary-case0}, a sequence of different positive real numbers $A_n\to A$ with $A_n\neq A$ for all $n$, and a metric ideal ray $\Xi\in\bd X$. This provides a sequence $(\xi_n)_n$ of timelike ideal points via the map $\smash{\Phi^{-1}}$ from \autoref{lem:bijection-product-idealboundary-case0} by considering $\xi_n:=\smash{\Phi^{-1}}(A_n,\Xi)$. Define also $\xi:=\smash{\Phi^{-1}}(A,\Xi)$. With respect to the angular metric the sequence does not converge, as $\ma(\xi_m,\xi_n)=\infty$ whenever $m\neq n$. However, considering a point $p=(t,x)\in Y_f$ (which will be in the timelike past of every ray in the sequence, as all of them generate the same past), one has
    \[
    \cosh\ma_{(t,x)}(\xi_n,\xi)=
    \cosh\theta_n^{f(t)}\cosh\theta^{f(t)}-\sinh\theta_n^{f(t)}\sinh\theta^{f(t)}=
    \sqrt{1+\frac{A_n^2}{f(t)^2}}\sqrt{1+\frac{A^2}{f(t)^2}}-\frac{A_n\, A}{f(t)^2}\,,
    \]
    and therefore $\ma_{(t,x)}(\xi_n,\xi)\xrightarrow{n\to\infty}0$. Consequently, for every $R,\varepsilon>0$ and every point $p\in Y_f$, the basic neighborhood $V_{p,\varepsilon,R}(\xi)$ contains a tail of the sequence $\xi_n$. In other words, $\xi_n\to\xi$ with respect to the cone topology.
\end{rem}

We summarize the structure of the future ideal boundary of a generalized cone, depending on the behavior of its warping function, in the following result.
\begin{thm}
Let $(X,d)$ be a complete geodesic space and let $f\colon \R\to (0,\infty)$ be convex. We have the following about the future ideal boundary $\bd^+Y$:
\begin{enumerate}
    \item If $f\to0$ quickly, it consists of just a point.
    \item If $f\to0$ slowly and $X$ is locally compact and $\CAT(-\varepsilon)$ for some $\varepsilon>0$, then $\bd^+Y$ is isometric to the quotient of the set $[0,\infty)\times \bd X$ under the equivalence relation $(0,\xi)\sim(0,\xi')$, endowed with the discrete distance with value infinity.
    \item If $f\to L>0$ and $X$ is locally compact and $\CAT(-\varepsilon)$ for some $\varepsilon>0$, then, $\bd^+Y$ is isometric to the metric warped product $[0,\infty)\times_{\sinh} \bd X$.
    \item If $f\to\infty$ slowly, it is in bijection with the set $X$. 
    \item If $f\to\infty$ quickly, it is isometric to the set $X$ endowed with the discrete distance with value infinity. 
\end{enumerate}
\end{thm}

\bibliographystyle{alphaurl}
\bibliography{references}
\end{document}